\documentclass[letterpaper]{amsart}
\pdfoutput=1
\newtheorem{theorem}{Theorem}[section]
\newtheorem{lemma}[theorem]{Lemma}

\newtheorem{corollary}[theorem]{Corollary}

\theoremstyle{definition}
\newtheorem{definition}[theorem]{Definition}
\newtheorem{example}[theorem]{Example}

\theoremstyle{remark}
\newtheorem{remark}[theorem]{Remark}
\usepackage{hyperref}
\usepackage{graphicx}
\usepackage{color}
\usepackage{amsmath}
\usepackage{amsfonts}
\usepackage{euler}
\usepackage{amssymb}
\usepackage{epsfig}
\usepackage{rotating}
\usepackage{graphics}
\usepackage{dynkin-diagrams}
\usepackage{amsrefs}

\newcommand{\bz}{\bar{z}}

\newcommand{\f}{\varphi}

\newcommand{\dz}{\wedge}

\newcommand{\C}{{\bf C }}

\newcommand{\ba}{\begin{array}}

\newcommand{\ea}{\end{array}}

\newcommand{\beq}{\begin{eqnarray}}

\newcommand{\eeq}{\end{eqnarray}}

\newtheorem{lm}{lemma}

\newtheorem{thee}{theorem}

\newtheorem{proo}{proposition}

\newtheorem{co}{corollary}

\newtheorem{rem}{remark}

\newtheorem{deff}{definition}

\newcommand{\bd}{\begin{deff}}

\newcommand{\ed}{\end{deff}}

\newcommand{\bl}{\begin{lm}}

\newcommand{\el}{\end{lm}}

\newcommand{\bp}{\begin{proo}}

\newcommand{\ep}{\end{proo}}

\newcommand{\bt}{\begin{thee}}

\newcommand{\et}{\end{thee}}

\newcommand{\bc}{\begin{co}}

\newcommand{\ec}{\end{co}}

\newcommand{\brm}{\begin{rem}}

\newcommand{\erm}{\end{rem}}

\newcommand{\der}{{\rm d}}

\usepackage{t1enc}
\def\frak{\mathfrak}

\newcommand{\newc}{\newcommand}

\newcommand{\id}{\operatorname{id}}

\renewcommand{\Im}{\operatorname{Im}}
\newcommand{\rea}{\operatorname{Re}}

\let\ccdot\cdot
\def\cdot{\hbox to 2.5pt{\hss$\ccdot$\hss}}

\newc{\aR}{\mbox{\boldmath{$ R$}}}
\newc{\aS}{\mbox{\boldmath{$ S$}}}
\newc{\aT}{\mbox{\boldmath{$ T$}}}
\newc{\aW}{\mbox{\boldmath{$ W$}}}

\newc{\aK}{\mbox{\boldmath{$ K$}}}
\newc{\aL}{\mbox{\boldmath{$ L$}}}

\newcommand{\bbC}{\mathbb{C}}

\newcommand{\bbN}{\mathbb{N}}

\usepackage{amssymb}
\usepackage{amscd}

\let\f=\varphi

\newcommand{\hook}{\raisebox{-0.35ex}{\makebox[0.6em][r]
{\scriptsize $-$}}\hspace{-0.15em}\raisebox{0.25ex}{\makebox[0.4em][l]{\tiny
 $|$}}}

\newcommand{\bma}{\begin{pmatrix}}
\newcommand{\ema}{\end{pmatrix}}

\let\t=\tau

\newc{\obstrn}[2]{B^{#1}_{#2}}

\newcommand{\rpl}                         
{\mbox{$
\begin{picture}(12.7,8)(-.5,-1)
\put(0,0.2){$+$}
\put(4.2,2.8){\oval(8,8)[r]}
\end{picture}$}}

\newcommand{\lpl}                         
{\mbox{$
\begin{picture}(12.7,8)(-.5,-1)
\put(2,0.2){$+$}
\put(6.2,2.8){\oval(8,8)[l]}
\end{picture}$}}

\usepackage{ifthen}

\newc{\tensor}[1]{#1}
\newc{\Mvariable}[1]{\mbox{#1}}
\newc{\down}[1]{{}_{#1}}
\newc{\up}[1]{{}^{#1}}

\newc{\JulyStrut}{\rule{0mm}{6mm}}
\newc{\midtenPan}{\mbox{\sf S}}
\newc{\midten}{\mbox{\sf T}}
\newc{\midtenEi}{\mbox{\sf U}}
\newc{\ATen}{\mbox{\sf E}}
\newc{\BTen}{\mbox{\sf F}}
\newc{\CTen}{\mbox{\sf G}}

\def\sideremark#1{\ifvmode\leavevmode\fi\vadjust{\vbox to0pt{\vss
 \hbox to 0pt{\hskip\hsize\hskip1em
 \vbox{\hsize3cm\tiny\raggedright\pretolerance10000
 \noindent #1\hfill}\hss}\vbox to8pt{\vfil}\vss}}}%

\newcommand{\Span}{\mathrm{Span}}

\numberwithin{equation}{section}

\newcounter{romenumi}
\newcommand{\labelromenumi}{(\roman{romenumi})}

\def\bc{\begin{cases}}
\def\ec{\end{cases}}

\def\a{\alpha}

\def\l{\label}
\def\t{\tilde}

\def\cb{{\mathcal B}}

\def\ba{{\mathbb A}}

\def\bc{{\mathbb C}}
\def\bd{{\mathbb D}}

\def\bl{{\mathbb L}}

\def\bp{{\mathbb P}}

\def\br{{\mathbb R}}

\def\bt{{\mathbb T}}

\def\bz{{\mathbb Z}}

\def\er{\eqref}

\def\bz{\mathbb Z}

\def\br{\mathbb R}
\def\bc{\mathbb C}
\def\bp{\mathbb P}

\def\bean{\begin{eqnarray}}
\def\eean{\end{eqnarray}}
\def\bea{\begin{eqnarray*}}
\def\eea{\end{eqnarray*}}
\def\beq{\begin{equation}}
\def\eeq{\end{equation}}
\def\beq*{\begin{equation*}}
\def\eeq*{\end{equation*}}
\def\bal{\begin{align*}}
\def\eal{\end{align*}}
\def\baln{\begin{align}}
\def\ealn{\end{align}}

\def\beg{\begin{gather*}}
\def\eng{\end{gather*}}
\def\bqu{\begin{question}}
\def\equ{\end{question}}

\def\ed{\eta^{-1}(Y)}

\def\ban{\begin{proof}[Answer]}
\def\ean{\end{proof}}

\def\on{\operatorname}

\def\bqu{\begin{question}}
\def\equ{\end{question}}

\def\0110{\begin{matrix} 0 & 1\\1&0\end{matrix}}

\def\t{\tilde}

\def\SO{\mathrm{SO}}

\def\fg{\mathfrak{g}}

\def\fn{\mathfrak{n}}
\def\fo{\mathfrak{o}}

\def\fs{\mathfrak{s}}

\def\fu{\mathfrak{u}}

\def\ban{\begin{proof}[Answer]}
\def\ean{\end{proof}}

\def\ben{\begin{equation}}
\def\een{\end{equation}}

\def\j1{{(j+1)}}

\def\ad{\text{ad}}

\def\bl{B\"acklund }

\def\f32{{}_3F_2}

\DeclareMathSymbol{\minu}{\mathbin}{AMSa}{"39}

\newenvironment{roof}[2]{\paragraph{{\bf Proof of {#1} {#2}.}}}{\hfill$\square$}
\begin{document}
\title[Parabolic CR geometries]{Accidental CR structures}
\vskip 1.truecm
\author{C. Denson Hill} \address{Department of Mathematics, Stony Brook University, Stony Brook, NY 11794, USA}
\email{Dhill@math.stonybrook.edu}
\author{Jo\"el Merker}\address{Laboratoire de Math\'ematiques d'Orsay, Universit\'e Paris-Saclay, Facult\'e des Sciences, 91405 Orsay Cedex, France.}\email{joel.merker@universite-paris-saclay.fr}
\author{Zhaohu Nie} \address{Department of Mathematics and Statistics, Utah State University, Logan, UT 84322-3900, USA}
\email{zhaohu.nie@usu.edu}
\author{Pawe\l~ Nurowski} \address{Center for Theoretical Physics,
Polish Academy of Sciences, Al. Lotnik\'ow 32/46, 02-668 Warszawa, Poland}
\email{nurowski@cft.edu.pl}
\thanks{The research was funded from the Norwegian Financial Mechanism 2014-2021 with project registration number 2019/34/H/ST1/00636.}

\date{\today}
\begin{abstract} We noticed a discrepancy between \'Elie Cartan and Sigurdur Helgason about the lowest possible dimension in which the simple exceptional Lie group ${\bf E}_8$ can be realized. This raised the question about the lowest dimensions in which various \emph{real forms} of the exceptional groups ${\bf E}_\ell$ can be realized. Cartan claims that ${\bf E}_6$ can be realized in dimension 16. However Cartan refers to the \emph{complex} group ${\bf E}_6$, or its \emph{split real form} $E_I$. His claim is also valid in the case of the real form denoted by $E_{IV}$. We find however that the real forms $E_{II}$ and $E_{III}$ of ${\bf E}_6$ can \emph{not} be realized in dimension 16 \`a la Cartan. In this paper we realize them in dimension 24 as groups of CR automorphisms of certain CR structures of higher codimension.

  As a byproduct of these two realizations, we provide a full list of \emph{CR structures} $(M,H,J)$ \emph{and their CR embeddings in} an appropriate $\bbC^N$, which satisfy the following conditions:
\begin{itemize}
  \item they have real codimension $k>1$,
\item the real vector distribution $H$ proper for the action of the complex structure $J$ is such that $[H,H]+H=\mathrm{T}M$,
\item the local group $G_J$ of CR automorphisms of the structure $(M,H,J)$ is simple, acts transitively on $M$ and has isotropy $P$ being a parabolic subgroup in $G_J$,
  \item the local symmetry group $G$ of the vector distribution $H$ on $M$ coincides with the group $G_J$ of CR automorphisms of $(M,H,J)$. 
\end{itemize}
Because all the CR structures from our list satisfy the last property we call them \emph{accidental}. Our CR structures of higher codimension with the exceptional symmetries $E_{II}$ and $E_{III}$ are particular entries in this list.
\end{abstract}
\maketitle
\vspace{-1truecm}
\newcommand{\ii}{{\bf i}}
\newcommand{\jj}{{\bf j}}
\newcommand{\kk}{{\bf k}}
\newcommand{\bbS}{\mathbb{S}}
\newcommand{\bbR}{\mathbb{R}}
\newcommand{\sog}{\mathbf{SO}}
\newcommand{\spg}{\mathbf{Sp}}
\newcommand{\glg}{\mathbf{GL}}
\newcommand{\slg}{\mathbf{SL}}
\newcommand{\og}{\mathbf{O}}
\newcommand{\soa}{\frak{so}}
\newcommand{\spa}{\frak{sp}}
\newcommand{\gla}{\frak{gl}}
\newcommand{\sla}{\frak{sl}}
\newcommand{\sua}{\frak{su}}
\newcommand{\dr}{\mathrm{d}}
\newcommand{\sug}{\mathbf{SU}}
\newcommand{\cspg}{\mathbf{CSp}}
\newcommand{\gat}{\tilde{\gamma}}
\newcommand{\Gat}{\tilde{\Gamma}}
\newcommand{\thet}{\tilde{\theta}}
\newcommand{\Thet}{\tilde{T}}
\newcommand{\rt}{\tilde{r}}
\newcommand{\st}{\sqrt{3}}
\newcommand{\kat}{\tilde{\kappa}}
\newcommand{\kz}{{K^{{~}^{\hskip-3.1mm\circ}}}}
\newcommand{\di}{{\rm div}}
\newcommand{\curl}{{\rm curl}}
\newcommand{\tn}{{\mathcal N}}
\newcommand{\ten}{{\Upsilon}}
\newcommand{\invol}[2]{\draw[latex-latex] (root #1) to
  [out=-30,in=-150] (root #2);}
\newcommand{\invok}[2]{\draw[latex-latex] (root #1) to
[out=-90,in=-90] (root #2);}
\section{Introduction}\label{intr}
In the German version of his PhD thesis \cite{CartanPhd} \'Elie Cartan gives a realization of the simple exceptional Lie group ${\bf F}_4$ as a symmetry group of a certain \emph{rank eight} vector distribution in \emph{dimension fifteen}. Sigurdur Helgason in \cite{He} reports on this fact as follows:

\begin{quote}
  Cartan represented ${\bf F}_4$ (...) by the Pfaffian system in $\bbR^{15}$ (...).  Similar results for ${\bf E}_6$ in $\bbR^{16}$, ${\bf E}_7$ in $\bbR^{27}$ and ${\bf E}_8$ in $\bbR^{29}$ are indicated in \cite{CartanPhd}. Unfortunately, detailed proofs of these remarkable representations of the exceptional groups do not seem to be available.
\end{quote}

Nowadays the information invoked by Helgason is \emph{in principle} understood in the context of \emph{parabolic geometries} \cite{CS}. We used the phrase `it is understood \emph{in principle}' because:
\begin{itemize}
  \item Cartan's paper has a \emph{misprint}, and his $\bbR^{29}$ as a space for a realization of ${\bf E}_8$, which is written in the third sentence before the end of his paper, should be $\bbR^{57}$; Helgason in \cite{He} copied this misprint without confronting it with the last sentence of Cartan's paper, which gives the correct space $\bbR^{57}$;
  \item Helgason's use of the word `similar' in the phrase  \emph{Similar results for ${\bf E}_6$ in $\bbR^{16}$}[...]\emph{are indicated in \cite{CartanPhd}} is not particularly appropriate, as it will be clear in the sequel;
    \item more importantly, the \emph{explicit} realizations of ${\bf E}_6$, ${\bf E}_7$ and ${\bf E}_8$ \emph{really} similar to Cartan's realization of ${\bf F}_4$ are still missing.
    \end{itemize}
Looking at Cartan's description of the realization of ${\bf F}_4$ in $\bbR^{15}$ given at the bottom of page 418 and the top of page 419 in \cite{CartanPhd}, one sees that Cartan speaks about a realization of the \emph{split real form}\footnote{In this paper we use the same notation for the real forms of the simple exceptional Lie groups as the notation presented in Table B.4 on pages 612-615 of \cite{CS}.}  $F_I$ of the complex simple exceptional Lie group ${\bf F}_4$ in $\bbR^{15}$ as a \emph{symmetry group} of a rank $8$-distribution $H$ in $\bbR^{15}$ defined as the \emph{annihilator} of the following 1-forms
$$\begin{aligned}
  &\lambda^{ij}=\der u^{ij}+x^i\der x^j+y^k\der y^l,\\
  &\lambda^7=\der u^7+\sum_{i=1}^4y^i\der x^i.
\end{aligned}$$
Here $u^{ij}$ with $1\leq i<j\leq 4$, $u^7$, $x^i$ with $i=1,2,3,4$, and $y^i$ with $i=1,2,3,4$, are coordinates in $\bbR^{15}$, and indices $k,l$ in the first formula above are such that the quadruple of indices $(ijkl)$, with $i<j$, is an \emph{even} permutation of numbers $1,2,3,4$.

Thus, Cartan is \emph{very explicit} with providing a realization of the \emph{split real form} $F_I$ of the complex simple exceptional Lie group ${\bf F}_4$:

It is the local symmetry group of a distribution
$$H=\{\Gamma(\mathrm{T}\bbR^{15})\ni X\,\,:\,\,X\hook\lambda^{ij}=X\hook\lambda^7=0,\,\,\forall \,\,1\leq i<j\leq 4\},$$
i.e. it is the local transformation group on $\bbR^{15}$ whose Lie algebra is spanned by all real vector fields $Y$ on $\bbR^{15}$ satisfying
$$\begin{aligned}
  \big( {\mathcal L}_Y&\lambda^{ij}\big)\dz\lambda^{12}\dz\lambda^{13}\dz\lambda^{14}\dz\lambda^{23}\dz\lambda^{24}\dz\lambda^{34}\dz\lambda^7=0,\quad 1\leq i<j\leq 4,\\
  \big( {\mathcal L}_Y&\lambda^7\big)\dz\lambda^{12}\dz\lambda^{13}\dz\lambda^{14}\dz\lambda^{23}\dz\lambda^{24}\dz\lambda^{34}\dz\lambda^7=0.
  \end{aligned}$$

Cartan's \emph{detailed} description of the ${\bf F}_4$ realization, is in contrast with his discussion of realizations of real forms of ${\bf E}_6$, ${\bf E}_7$ and ${\bf E}_8$ in the respective $\bbR^{16}$, $\bbR^{27}$ and $\bbR^{57}$. In these cases Cartan only specifies the commutation relations between generators of the corresponding real Lie algebras, and observes, that in each of the ${\bf E}_\ell$ cases, $\ell=6,7,8$, they include real subalgebras of respective real dimensions 62, 106, 191. This only means that there are $E_\ell$-homogeneous spaces of dimension 16, 27 and 57 for the respective $\ell=6,7,8$. That is all that Cartan says! In particular, he says nothing about that what are the geometric structures on these spaces which are ${\bf E}_\ell$ homogeneous. And a closer look at his 62, 106 and 191 dimensional subalgebras of these ${\bf E}_\ell$s shows that the corresponding geometric structures are \emph{very different} from the structure of the rank 8 bracket generating distribution in dimension 15, which realizes the real form $F_I$ of ${\bf F}_4$ as its local symmetry.

In short: Cartan's geometry with symmetry ${\bf F}_4$ in dimension 15, viewed as a \emph{parabolic geometry}, is \emph{two}-step graded, whereas Cartan's homogeneous spaces with symmetry ${\bf E}_6$ in dimension 16, with symmetry ${\bf E}_7$ in dimension 27, and with symmetry ${\bf E}_8$ in dimension 57, viewed as \emph{parabolic geometries}, are \emph{one}-step graded. In particular, his geometry in dimension 16 with symmetry ${\bf E}_6$ is an $\bbR{\bf Spin}(5,5)$ geometry.

If somebody is interested in the details of these ${\bf F}_4$, ${\bf E}_6$, ${\bf E}_7$ and ${\bf E}_8$ realizations, we direct her to a forthcoming paper \cite{nurowskiG2CR}. What is important for our current paper is that, as far as the ${\bf E}_6$ realization in dimension 16 is concerned, Cartan in \cite{CartanPhd} realized \emph{one particular real form of} ${\bf E}_6$ \emph{only}, namely the \emph{split real form} $E_I$, with the Satake diagram\\
\centerline{ \begin{dynkinDiagram}[edge length=.4cm]{E}{oooooo}
  \end{dynkinDiagram}.}
The parabolic geometry he considered was of type $(E_I,\bbR{\bf Spin}(5,5))$, which corresponds to the following crossing\\
\centerline{ \begin{dynkinDiagram}[edge length=.4cm]{E}{ooooox}
\end{dynkinDiagram}}
on this diagram. Cartan in the last sentence of \cite{CartanPhd} writes:
\begin{quote}
  Die f\"unf specielle einfachen Gruppen mit 14 bez. 52, 78, 133, 248 Parametern k\"onnen in weniger als 5 bez. 15, 16, 27, 57 Ver\"anderlichen nich existiren\footnote{In this statement there is another misprint of \cite{CartanPhd}, since the correct number 133 of the dimension of ${\bf E}_7$ is erroneously printed  as 123.}.
   \end{quote}
So he claims, in particular, that the lowest dimension in which the group ${\bf E}_6$ is realized is 16. This is however true only if he does not care about \emph{which real form of} ${\bf E}_6$ \emph{he wants to realize}. If he wanted to realize the real form $E_{II}$ or $E_{III}$ of ${\bf E}_6$ in dimension 16, his method of realization of $E_I$ would not work.
This is quite visible in the Satake diagrams\\
\centerline{\begin{dynkinDiagram}[edge length=.4cm]{E}{oooooo}
\invol{1}{6}\invol{3}{5}
 \end{dynkinDiagram},\hspace{0.3cm}\begin{dynkinDiagram}[edge length=.4cm]{E}{oooooo}
\invol{1}{6}
\end{dynkinDiagram}}
of $E_{II}$ and $E_{III}$: Because the first and the last nodes of these diagrams are complex conjugated, when crossing one of them, one has to cross the conjugated one as well; therefore no 62-dimensional parabolic subgroup corresponding to one cross exists in these groups. If we want to make a realization of $E_{II}$ and $E_{III}$ in the way Cartan did it for $E_I$ we should cross the lateral root in the Satake diagrams of these groups. This results in the automatic cross on the opposite lateral root, which corresponds to a choice of a parabolic subgroup of 54 dimensions, and giving the realization in dimension 24.
We display this result in the following corollary, 

\begin{corollary}
  The real forms $E_{II}$ and $E_{III}$ of the simple exceptional Lie group ${\bf E}_6$ can be realized in dimension 24. The realizations are given by the groups $E_{II}$ and $E_{III}$ being automorphisms groups of rank 16 bracket generating distributions with growth vectors $(16,24)$ defined in our respective Corollaries \ref{co13} and \ref{co14}. These realizations happen to be the same as $E_{II}$ and $E_{III}$ being groups of CR automorphisms of the accidental CR structures described in our respective Theorems \ref{e2} and \ref{e3}.
\end{corollary}

These respective realizations of $E_{II}$ and $E_{III}$ in $\bbR^{24}$ are very much in the spirit of Cartan's realization of the real form $F_I$ of ${\bf F}_4$ as a symmetry of a \emph{bracket generating} distribution. As explained below in Corollaries
\ref{co13} and \ref{co14}, the groups $E_{II}$ and $E_{III}$ can be realized as symmetries of certain rank 16 bracket generating distributions in dimension 24, and the realizations of these groups as symmetries of parabolic geometries are 2-step graded. These groups also \emph{accidentally} turn out to be symmetries of certain CR geometries of higher codimension naturally associated with these 2-step rank 16 distributions. As such, they were our main motivation for writing this paper. More specifically, the ${\bf E}_6$ realizations as symmetries of CR structures of higher codimension as presented in our Theorems \ref{e2}, \ref{e3}, are examples of a \emph{positive} answer to the following problem:

\emph{Consider a parabolic geometry totally defined in terms of an even-rank real distribution $H$ on a real manifold $M$. Is it possible that $H$ admits an integrable complex structure $J$ compatible with this geometry? Compatibility here means that the geometry of $H$ itself defines a unique complex structure $J$ on $H$ in a canonical way; in particular that means that all the local differential invariants of $H$ on $M$ coincide with all the local differential invariants of the structure $(H,J)$ on $M$?}

This problem arose during our discussions with Katja Sagerschnig, and in this paper we answer it in the affirmative in the case of parabolic geometries with 2-step gradings. The full list of parabolic geometries having these properties is contained in Theorems \ref{e2}, \ref{e3}, \ref{so}, \ref{so*} and \ref{su}. The relevant maximal groups $G$ of symmetries are respectively $E_{II}$, $E_{III}$, $\sog(\ell-1,\ell+1)$ with $\ell\geq 4$,  $\sog^*(4m+2)$ with $m\geq 2$ and $\sug(t+s,r+t+s)$ with $(r,t)\neq 0$, $r\geq 0$, $t\geq 0$, $s\geq 2$. 

In all these five cases we describe a CR structure $(M,H,J)$ corresponding to the flat model of the corresponding parabolic geometry. We do it by finding explicit embeddings of each of these CR structures in an appropriate $\bbC^N$. We state their CR dimension $n$ and CR codimension $k$. Forgetting about the complex structure $J$ on the CR distribution $H$ we define $H$ totally in real terms, and due to our setting, we may say that the symmetry of the distribution itself is the same as the group of CR automorphisms of $(M,H,J)$.

We also compared our list of 2-step parabolic `accidental CR' geometries with K. Yamaguchi's list \cite{Yam} of \emph{nonrigid} parabolic geometries. It turns out (see our Theorem \ref{summm} in the last section of this article) that the only nonrigid geometries on our list are those with the group $G$ being any of the $SU(t+s,r+t+s)$, $s\geq 2$, $(r,t)\neq 0$ graded by the second and second last roots. We stress that these nonrigid parabolic geometries correspond to CR manifolds of \emph{higher} codimension (i.e. they are \emph{not} of hypersurface type).   

Our results are in conformation  with the classification of semisimple Levi-Tanaka algebras and their corresponding CR manifolds by Medori and Nacinovich \cite{MN1, medori}. We compare our approach with theirs in Section \ref{sec5}. Our techniques are different.

Returning to our list of 2-step parabolic geometries with an `accidental' CR structure, we find on it very interesting \emph{geometric realizations} of simple Lie groups $\sug(p,q)$, $\sog(\ell-1,\ell+1)$, $\sog^*(4m+2)$ and, in particular,  the above mentioned \emph{two real forms} ${E_{II}}$ \emph{and} ${E_{III}}$ \emph{of the exceptional simple complex Lie group} ${\bf E_6}$. We describe them in our Theorems \ref{e2}, \ref{e3}, \ref{so}, \ref{so*} and \ref{su}. Here, in this introduction, as two highlights, we focus on the two ${\bf E_6}$\emph{-homogeneous examples}.

\vspace{0.5cm}
\noindent
    {\bf The case of $E_{II}$ symmetry.}\\
Consider $\bbC^{16}$ with holomorphic coordinates $(w^1,w^2,\dots,w^8,z^1,z^2,\dots, z^8)$, and its subset $M^{24}_{E_{II}}\subset \bbC^{16}$  defined by:
$$\boxed{\begin{aligned}
    M^{24}_{E_{II}}\,=\,\, \Big\{~&\bbC^{16}\ni (w^1,w^2,\dots,w^8,z^1,z^2,\dots z^8)\,\,\mathrm{s.t.}\\
&\Im\,w_1
\,=\,
\rea\,
\big(
z_1\,\overline{z}_4
+
z_2\,\overline{z}_3
\big)
\\
&\Im\,w_2
\,=\,
\rea\,
\big(
z_1\,\overline{z}_6
+
z_2\,\overline{z}_5
\big)
\\
&\Im\,w_3
\,=\,
\Im\,
\big(
z_1\,\overline{z}_7
+
z_5\,\overline{z}_3
\big)
\\
&\Im\,w_4
\,=\,
\Im\,
\big(
z_2\,\overline{z}_7
+
z_3\,\overline{z}_6
-
z_5\,\overline{z}_4
-
z_8\,\overline{z}_1
\big)
\\
&\Im\,w_5
\,=\,
\rea\,
\big(
z_2\,\overline{z}_7
+
z_3\,\overline{z}_6
-
z_5\,\overline{z}_4
-
z_8\,\overline{z}_1
\big)
\\
&\Im\,w_6
\,=\,
\Im\,
\big(
z_2\,\overline{z}_8
+
z_6\,\overline{z}_4
\big)
\\
&\Im\,w_7
\,=\,
\rea\,
\big(
z_3\,\overline{z}_8
+
z_4\,\overline{z}_7
\big)
\\
&\Im\,w_8
\,=\,
\rea\,
\big(
z_5\,\overline{z}_8
+
z_6\,\overline{z}_7
\big)
    \,\,\Big\}
\end{aligned}}.$$

We have the following theorem.
\begin{theorem}\label{e2}
The set $M^{24}_{E_{II}}\subset \bbC^{16}$ is a real 24-dimensional embedded CR manifold, acquiring the CR structure of CR dimension $n=8$ and CR codimension $k=8$ from the ambient complex space $\bbC^{16}$. Its local group of CR automorphisms is isomorphic to the real simple exceptional Lie group $E_{II}$ with the Lie algebra having the Satake diagram \begin{dynkinDiagram}[edge length=.4cm]{E}{oooooo}
\invol{1}{6}\invol{3}{5}
 \end{dynkinDiagram}. It is locally CR equivalent to the flat model $E_{II}/P_{(2)}$ of a 24-dimensional parabolic geometry of type $(E_{II},P_{(2)})$, where the real parabolic subgroup $P_{(2)}$ in $E_{II}$ is determined by the following crossing on the corresponding $\mathfrak{e}_6$ Satake diagram: 
\begin{dynkinDiagram}[edge length=.4cm
  ]{E}{toooot}
\invol{1}{6}\invol{3}{5}
 \end{dynkinDiagram}.
\end{theorem}

\vspace{0.5cm}
\noindent
{\bf The case of $E_{III}$ symmetry.}\\
Consider $\bbC^{16}$ with holomorphic coordinates $(w^1,w^2,\dots,w^8,z^1,z^2,\dots, z^8)$, and its subset $M^{24}_{E_{III}}\subset \bbC^{16}$  defined by:
$$\boxed{\begin{aligned}
M^{24}_{E_{III}}\,=\,\, \Big\{~&\bbC^{16}\ni (w^1,w^2,\dots,w^8,z^1,z^2,\dots z^8)\,\,\mathrm{s.t.}\\
&\mathrm{Im}(w^1)~=~\mathrm{Re}(z_1\bar{z}_8+z_2\bar{z}_4+z_3\bar{z}_7+z_5\bar{z}_6)\\
&\mathrm{Im}(w^2)~=~\mathrm{Re}(z_1\bar{z}_4-z_2\bar{z}_8-z_3\bar{z}_6+z_5\bar{z}_7)\\
&\mathrm{Im}(w^3)~=~\mathrm{Re}(z_1\bar{z}_7+z_2\bar{z}_6-z_3\bar{z}_8-z_4\bar{z}_5)\\
&\mathrm{Im}(w^4)~=~\mathrm{Re}\big(i(z_1\bar{z}_2+z_3\bar{z}_5+z_4\bar{z}_8+z_6\bar{z}_7)\big)\\
&\mathrm{Im}(w^5)~=~\mathrm{Re}(z_1\bar{z}_6-z_2\bar{z}_7+z_3\bar{z}_4-z_5\bar{z}_8)\\
&\mathrm{Im}(w^6)~=~\mathrm{Re}\big(i(z_1\bar{z}_5+z_2\bar{z}_3-z_4\bar{z}_7+z_6\bar{z}_8)\big)\\
&\mathrm{Im}(w^7)~=~\mathrm{Re}\big(i(z_1\bar{z}_3-z_2\bar{z}_5+z_4\bar{z}_6+z_7\bar{z}_8)\big)\\
&\mathrm{Im}(w^8)~=~|z_1|^2+|z_2|^2+|z_3|^2+|z_4|^2+|z_5|^2+|z_6|^2+|z_7|^2+|z_8|^2\,\,\Big\}
\end{aligned}}.$$

We have the following theorem.
\begin{theorem}\label{e3}
The set $M^{24}_{E_{III}}\subset \bbC^{16}$ is a real 24-dimensional embedded CR manifold, acquiring the CR structure of CR dimension $n=8$ and CR codimension $k=8$ from the ambient complex space $\bbC^{16}$. Its local group of CR automorphisms is isomorphic to the real simple exceptional Lie group $E_{III}$ with the Lie algebra having the Satake diagram \begin{dynkinDiagram}[edge length=.4cm]{E}{oo***o}
\invol{1}{6}
 \end{dynkinDiagram}. It is locally CR equivalent to the flat model $E_{III}/P_{(1)}$ of a 24-dimensional parabolic geometry of type $(E_{III},P_{(1)})$, where the real parabolic subgroup $P_{(1)}$ in $E_{III}$ is determined by the following crossing on the corresponding $\mathfrak{e}_6$ Satake diagram: 
\begin{dynkinDiagram}[edge length=.4cm
  ]{E}{to***t}
\invol{1}{6}
 \end{dynkinDiagram}.
\end{theorem}

A \emph{CR structure} on a manifold $M$ is a triple $(M,H,J)$ where $H$ is rank $\ell=2n$ \emph{vector distribution} on $M$ and $J:H\to H$, such that $J^2=-\mathrm{id}_H$, is an \emph{integrable complex structure} on $H$. CR structures have their invariants, one of them being their \emph{symmetry group} $G_J$, also called the \emph{group of CR automorphisms}. If this group acts transitively on $M$ the CR structure $(M,H,J)$ is called \emph{homogeneous}. A CR structure $(M,H,J)$ may be viewed as an \emph{additional structure}, a \emph{decoration}, on the \emph{distribution structure} $(M,H)$, i.e. a manifold equipped $M$ with a rank $\ell=2n$ vector distribution. The distribution structures $(M,H)$ also have invariants, also allowing for the notion of \emph{their symmetry group} $G$, and their \emph{homogeneity}. 

In this paper we focus on CR structures $(M,H,J)$ which are \emph{accidental}. These are \emph{homogeneous} CR structures $(M,H,J)$ for which the \emph{undecorated} distribution structure $(M,H)$ is also \emph{homogeneous} and the corresponding \emph{symmetry groups} $G_J$ and $G$ \emph{coincide}. It follows that the embedded CR structures given in Theorems \ref{e2} and \ref{e3} are \emph{accidental} (see Section \ref{acci} and Definition \ref{rrr} for more details). Because of their \emph{accidentality}, the CR structures involved in Theorems \ref{e2} and \ref{e3} provide also nice realizations of the real exceptional simple Lie groups $E_{II}$ and $E_{III}$ in \emph{purely real} terms. Actually these two theorems imply the following two corollaries. In them these groups are identified as transformation groups of symmetries of two particular vector distributions near the origin of $\bbR^{24}$.

\begin{corollary}\label{co13}
  Consider a 24-dimensional manifold $M^{24}_{II}$ locally parametrized by real coordinates $(u^1,u^2,\dots,u^8,x^1,x^2,\dots,x^8,y^1,y^2,\dots,y^8)$ and the Pfaffian system of 1-forms $[\lambda^1,\lambda^2,\dots,\lambda^8]$  on $M^{24}_{II}$ given by
    $$\begin{aligned}
    \lambda^1=\,\,&\der u^1+\tfrac12\big(x^1\der y^4+x^2\der y^3+x^3\der y^2+x^4\der y^1-y^1\der x^4-y^2\der x^3-y^3\der x^2-y^4\der x^1\big)\\
    \lambda^2=\,\,&\der u^2+\tfrac12\big(x^1\der y^6+x^2\der y^5+x^5\der y^2+x^6\der y^1-y^1\der x^6-y^2\der x^5-y^5\der x^2-y^6\der x^1\big)\\
    \lambda^3=\,\,&\der u^3+\tfrac12\big(x^1\der x^7-x^3\der x^5+x^5\der x^3-x^7\der x^1+y^1\der y^7-y^3\der y^5+y^5\der y^3-y^7\der y^1\big)\\
    \lambda^4=\,\,&\der u^4+\tfrac12\big(x^1\der x^8+x^2\der x^7+x^3\der x^6+x^4\der x^5-x^5\der x^4-x^6\der x^3-x^7\der x^2-x^8\der x^1+\\&\quad\quad\quad\quad y^1\der y^8+y^2\der y^7+y^3\der y^6+y^4\der y^5-y^5\der y^4-y^6\der y^3-y^7\der y^2-y^8\der y^1\big)\\
    \lambda^5=\,\,&\der u^5+\tfrac12\big(y^1\der x^8-y^2\der x^7-y^3\der x^6+y^4\der x^5+y^5\der x^4-y^6\der x^3-y^7\der x^2+y^8\der x^1-\\&\quad\quad\quad\quad x^1\der y^8+x^2\der y^7+x^3\der y^6-x^4\der y^5-x^5\der y^4+x^6\der y^3+x^7\der y^2-x^8\der y^1\big)\\
    \lambda^6=\,\,&\der u^6+\tfrac12\big(x^2\der x^8-x^4\der x^6+x^6\der x^4-x^8\der x^2+y^2\der y^8-y^4\der y^6+y^6\der y^4-y^8\der y^2\big)\\
    \lambda^7=\,\,&\der u^7+\tfrac12\big(x^3\der y^8+x^4\der y^7+x^7\der y^4+x^8\der y^3-y^3\der x^8-y^4\der x^7-y^7\der x^4-y^8\der x^3\big)\\
    \lambda^8=\,\,&\der u^8+\tfrac12\big(x^5\der y^8+x^6\der y^7+x^7\der y^6+x^8\der y^5-y^5\der x^8-y^6\der x^7-y^7\der x^6-y^8\der x^5\big).
  \end{aligned}
  $$
  Then the rank 16 distribution $H$ defined on $M^{24}_{II}$ via
  $$H=\{\mathrm{T}M^{24}_{II}\ni X\,\,\mathrm{s.t.}\,\,X\hook \lambda^i=0,\,\,\forall i=1,2,\dots, 8\}$$ has the real exceptional simple Lie group $E_{II}$ as its group of symmetries.
\end{corollary}

We also have:

\begin{corollary}\label{co14}
  Consider a 24-dimensional manifold $M^{24}_{II}$ locally parametrized by real coordinates $(u^1,u^2,\dots,u^8,x^1,x^2,\dots,x^8,y^1,y^2,\dots,y^8)$ and the Pfaffian system of 1-forms $[\lambda^1,\lambda^2,\dots,\lambda^8]$  on $M^{24}_{II}$ given by$$\begin{aligned}
    \lambda^1=\,\,&\der u^1+x^1\der y^8+x^2\der y^4+x^3\der y^7+x^4\der y^2+x^5\der y^6+x^6\der y^5+x^7\der y^3+x^8\der y^1\\
    \lambda^2=\,\,&\der u^2+x^1\der y^4+y^8\der x^2+y^6\der x^3+x^4\der y^1+x^5\der y^7+y^3\der x^6+x^7\der y^5+y^2 \der x^8\\
    \lambda^3=\,\,&\der u^3+x^1\der y^7+x^2\der y^6+y^8\der x^3+y^5\der x^4+y^4\der x^5+x^6\der y^2+x^7\der y^1+y^3\der x^8\\
    \lambda^4=\,\,&\der u^4+x^2\der x^1+x^5\der x^3+x^8\der x^4+x^7\der x^6+y^2\der y^1+y^5\der y^3+y^8\der y^4+y^7\der y^6\\
    \lambda^5=\,\,&\der u^5+x^1\der y^6+y^7\der x^2+x^3\der y^4+x^4\der y^3+y^8\der x^5+x^6\der y^1+y^2\der x^7+y^5\der x^8\\
    \lambda^6=\,\,&\der u^6+x^5\der x^1+x^3\der x^2+x^4\der x^7+x^8\der x^6+y^5\der y^1+y^3\der y^2+y^4\der y^7+y^8\der y^6\\
    \lambda^7=\,\,&\der u^7+x^3\der x^1+x^2\der x^5+x^6\der x^4+x^8\der x^7+y^3\der y^1+y^2\der y^5+y^6\der y^4+y^8\der y^7\\
    \lambda^8=\,\,&\der u^8+x^1\der y^1+x^2\der y^2+x^3\der y^3+x^4\der y^4+x^5\der y^5+x^6\der y^6+x^7\der y^7+x^8\der y^8.\\
  \end{aligned}
  $$
    Then the rank 16 distribution $H$ defined on $M^{24}_{III}$ via
  $$H=\{\mathrm{T}M^{24}_{III}\ni X\,\,\mathrm{s.t.}\,\,X\hook \lambda^i=0,\,\,\forall i=1,2,\dots, 8\}$$ has the real exceptional simple Lie group $E_{III}$ as its group of symmetries.
\end{corollary}

We close this Introduction to focus reader's attention on the questions we address in this paper:
\begin{enumerate}
\item Can a homogeneous bracket generating real distribution $H$, of rank $\ell=2n$ on a real manifold $M=G/P$, admit an integrable complex structure $J$ on $H$, such that the associated CR structure $(M,H,J)$ on $M$ is homogeneous, having the same group $G$ of CR automorphisms as the group of automorphisms of the naked distribution $H$? According to our terminology, which was briefly introduced between Theorem \ref{e3} and Corollary \ref{co13}, this question asks \emph{if there exist accidental CR structures}?\label{enu1}
\item There is a large class of homogeneous bracket generating distributions $H$ which define \emph{flat models of} various \emph{parabolic geometries}. In this context, the above question can be restricted to: \emph{do there exist accidental} parabolic \emph{CR structures}?
\item The most geometrically studied CR structures $(M,H,J)$ are the \emph{hypersurface type} CR structures  with \emph{nondegenerate Levi form}. They are examples of parabolic geometries. All these are \emph{not} accidental, since their distributions $H$ are \emph{contact distributions} having \emph{infinite dimensional} group of automorphisms, whereas their groups of CR automorphisms can not be larger than  finite dimensional simple Lie groups $\sug(p,q)$. 
\item There are also known nontrivial classes of \emph{non}-accidental CR structures associated with other parabolic geometries than those of hypersurface type Levi-nondegenerate CRs. 
  For example one can take the exceptional simple $G_2$ parabolic geometry associated with a generic rank 2 distribution $H$ on a five manifold $M$. Here, at least locally one can always find $J$ on $H$ so that the real codimension three and complex dimension one CR structure $(M,H,J)$ has lower symmetry than the 14-dimensional group $G_2$. Actually it is well known \cite{nurowskiG2CR} that on the $G_2$ flat rank 2 distribution $H$ in dimension 5, one can put a CR structure $J$ with group of CR automorphisms of dimension no larger than 7. Such examples are however easy to make, because in them the distribution $H$ has rank 2, i.e. the only requirement $J$ has to fulfill is the algebraic constraint $J^2=-\mathrm{id}_H$, and one does not need to worry about the integrability of $J$ in $H$.
  \item When $n>1$, given a rank $\ell=2n$ distribution on $M$, to put $J$ on $H$, apart from the algebraic constraint $J^2=-\mathrm{id}_H$, one needs to impose \emph{nonlinear differential} constraints on $J$, which may be incompatible with $H$. One therefore may think, that if the parabolic geometry associated with rank $\ell=2n$ distribution $H$ admits an integrable complex structure $J$ in $H$, then it \emph{must be accidental}. This brings yet another question: \emph{is it true that if a flat distribution $H$ of rank $\ell=2n\geq 4$ defining a parabolic geometry of type $(G,P)$ on a manifold $M=G/P$ admits an integrable $J$, then the associated CR structure $(M,H,J)$ is accidental}?
    \end{enumerate}
We believe that our explanations and examples included in this paper clarify all the issues enumerated above. 
\section{Basic notions and motivation}
\subsection{Distribution structure on a manifold}
This is a pair $(M,H)$, where $M$ is a smooth manifold, and $H$ is a \emph{vector distribution} on $M$. We recall that a rank $\ell$ \emph{vector distribution} $H$ on an $m$-dimensional smooth manifold $M$, or an $\ell$-\emph{distribution} $H$ on $M$, for short, is a smooth assignment $M\ni p\to H_p\subset \mathrm{T}_pM$ of an $\ell$-dimensional vector subspace $H_p$ of $\mathrm{T}_pM$, to each point $p\in M$. In the spirit of Felix Klein, one associates a \emph{geometry} with such objects by saying that two $\ell$-distributions $H_1$ and $H_2$ on $M$ are (locally) \emph{equivalent} if there exists a (local) diffeomorphism $\phi: M\to M$ on $M$ such that $\phi_*H_1=H_2$. It follows, that if the distributions are \emph{not} integrable, i.e. if they do \emph{not} satisfy the Frobenius condition $[H,H]\subset H$, then starting from dimension $m=5$ of $M$, there exist locally nonequivalent $\ell$-distributions. From now on we will only consider \emph{non}integrable $\ell$-distributions on $M$.

Let us introduce the simplest \emph{local invariant} of an $\ell$-distribution $H$, namely its \emph{growth vector}. It is determined by considering a sequence of distributions on $M$ defined inductively as follows: $D_{-1}:=H$, $D_{-i-1}:=[D_{-1},D_{-i}]+D_{-i}$ for $i\in\bbN$. The growth vector of $H$ is related to the ranks of these distributions. Explicitly, it is the nondecreasing sequence of integers $(r_{-1},r_{-2},\dots,r_{-i},\dots)$, with each $r_{-i}$ being the rank of the corresponding distribution $D_{-i}$, $r_{-i}=\mathrm{rank}(D_{-i})$. In particular $r_{-1}=\ell$. In general the growth vector can vary from point to point on $M$, but in this paper, we will only consider distribution structures $(M,H)$ with (locally) \emph{constant} growth vector. We further mention that the distribution $H$ from the structure $(M,H)$ is \emph{bracket generating} if $r_{-s}=m=\dim(M)$ for some $s\in \bbN$. In the case of bracket generating $\ell$-distributions with constant growth vector, to give more information about them, one includes the growth vector $(\ell,r_{-2},\dots,m)$ in their name. One therefore has such names as e.g. a $(2,3,5)$-distribution, which denotes a rank 2 distribution in diemnsion 5, with a constant growth vector $(2,3,5)$.

Another simple (local) \emph{invariant} of an $\ell$-distribution $H$ on $M$ is its (local) \emph{group of automorphisms} $G$. This consists of those (local) diffeomorphisms $\phi:M\to M$, called \emph{symmetries}, which \emph{preserve} $H$, i.e. are such that $\phi_*H=H$. The group multiplication in $G$ is the composition of the symmetries as diffeomorphisms. For short the (local) group $G$ of automorphisms of $(M,H)$ is called the group of (local) symmetries of the rank $\ell$-distribution $H$. If a (local) group $G$ of symmetries of $H$ acts transitively on $(M,H)$, then the $\ell$-distribution $H$ is called (locally) \emph{homogeneous}. In such case, the manifold $M$ is locally diffeomorphic to $M=G/P$, where $P$ is the isotropy subgroup of $G$, which preserves $H_p$ at a point $p\in M$. It is known that there exist locally (and globally) nonequivalent homogeneous $\ell$-distributions on manifolds. From the local point of view, and in particular in the homogeneous case, to see that the group $G$ of symmetries of two distributions are different, it is enough to consider the algebraic structure of the space of \emph{vector fields} $Y$ on $M$ such that $[Y,H]\subset H$. These are called \emph{infinitesimal symmetries} of $H$ and form a \emph{Lie algebra} $\mathfrak{g}$ \emph{of symmetries} of the pair $(M,H)$. The Lie algebra $\mathfrak{g}$ is obviously a local invariant of an $\ell$-distribution $H$ on $M$. It is the Lie algebra of $G$.  

One can locally define an $\ell$-distribution $H$ on $M$ by distinguishing a rank $(m-\ell)$ subbundle $H^\perp\subset \mathrm{T}^*M$ of the cotangent bundle $\mathrm{T}^*M$  and saying that $H$ consists of all vector fields $X$ on $M$ which \emph{annihilate} $H^\perp$, $H:=\{X\in\mathrm{T}M\,|\,X\hook H^\perp=0\}$.
\begin{example}\label{ex1}
  As an example of a distribution defined in this way let us consider the following $4$-distribution $H$ in dimension $m=7$.
  Let $M=\bbR^{7}$ with Cartesian coordinates $(x^1,x^2,x^3,x^4,x^5,x^6,x^7)$ and let $$H^\perp=\Span(\lambda^1,\lambda^2,\lambda^3),$$ with the 1-forms $\lambda^i$, $i=1,2,3$, being defined by:
  \begin{equation}
  \begin{aligned}
    \lambda^1&=\der x^5+x^1\der x^4+x^2\der x^3,\\
    \lambda^2&=\der x^6+x^3\der x^4+x^1\der x^2,\\
    \lambda^3&=\der x^7+x^3\der x^1+x^2\der x^4.
  \end{aligned}
  \label{baby}
  \end{equation}
  The corresponding $4$-distribution $H$ is
  \begin{equation}H=\Span(X_1,X_2,X_3,X_4),\label{babyh}\end{equation}
  with the four vector fields $X_i$, $i=1,2,3,4$ annihilating $H^\perp$, given by
  \begin{equation}
    \begin{aligned}X_1&=\partial_1-x^3\partial_7,\\ X_2&=\partial_2-x^1\partial_6,\\ X_3&=\partial_3-x^2\partial_5,\\ X_4&=\partial_4-x^1\partial_5-x^3\partial_6-x^2\partial_7.\end{aligned}\label{baby*}\end{equation}
  One checks, that the $4$-distribution $H$ is \emph{not} integrable, actually $$[H,H]+H=\mathrm{T}M,$$
  because
  \begin{equation}
    \begin{aligned}
{}
      [X_3,X_2]&=[X_4,X_1]=\partial_5,\\
      [X_2,X_1]&=[X_4,X_3]=\partial_6,\\
      [X_1,X_3]&=[X_4,X_2]=\partial_7.
    \end{aligned}
    \end{equation}
  Thus, the distribution structure $(M,H)$ is enforced on $M=\bbR^7$ by a $(4,7)$ distribution $H$.

  Moreover, after solving the symmetry equations
  $$
    [Y,X_i]\dz X_1\dz X_2\dz X_3\dz X_4=0,\quad \mathrm{for}\quad i=1,2,3,4,
    $$
    one finds that the \emph{Lie algebra} $\mathfrak{g}$ \emph{of infinitesimal symmetries} of $H$ is isomorphic to the \emph{simple} Lie algebra $\spa(1,2)$. Thus this distribution has a 21 dimensional Lie algebra of symmetries. In particular, as can be easily checked, vector fields:
    $$\tiny{\begin{aligned}
      Y_{10}&=x^3\partial_1+x^4\partial_2-x^1\partial_3-x^2\partial_4+(x^1x^2-x^3 x^4)\partial_5+\tfrac12\big((x^1)^2+(x^2)^2-(x^3)^2-(x^4)^2\big)\partial_7,\\
      Y_{12}&=\partial_4\end{aligned}}$$
    are infinitesimal symetries of $(M,H)$.

    It follows \cite{CS} that the distribution structure given by a pair $(M=\bbR^7,H)$, as above, is locally diffeomorphic to the \emph{flat model} $M=\spg(1,2)/P$ of a \emph{parabolic geometry} of type $(\spg(1,2),P)$, where $P$ is a \emph{parabolic subgroup} in $\spg(1,2)$ related to the following crossed Satake Diagram: $\tikzset{/Dynkin diagram/fold style/.style={stealth-stealth,thin,
shorten <=1mm,shorten >=1mm}}\begin{dynkinDiagram}[edge length=.5cm]{C}{*t*}
 \end{dynkinDiagram}$. It gives an example of a nonintegrable, \emph{bracket generating}, homogeneous distribution on a manifold with a \emph{simple} symmetry group $G$ (in this case $G=\spg(1,2)$).  
\end{example}
\subsection{Tanaka prolongation and symmetry.} We only provide the minimal information, as regards our needs, about the Tanaka theory, refering a reader interested in details to the original papers \cite{tanaka}.

    In the context of distributions, the Tanaka theory is mainly applied to the case when a vector distribution $H$ on a real $N$-dimensional manifold $M^N$ defines a $p$-\emph{step filtration} \begin{equation}{\mathcal D}^{-1}\subset{\mathcal D}^{-2}\subset\dots\subset{\mathcal D}^{-p}=\mathrm{T}M^N\label{pstep1}\end{equation} of the tangent bundle of $M^N$, with the filtered components ${\mathcal D}^{-k}$ defined by:
    \begin{equation}{\mathcal D}^{-1}=H,\quad\mathrm{and}\quad {\mathcal D}^{-k-1}=[{\mathcal D}^{-1},{\mathcal D}^{-k}]+{\mathcal D}^{-k}\quad \forall k<p,\,\, k,p\in\bbN.\label{pstep2}\end{equation}  This defines a \emph{graded vector space} $$\mathfrak{n}_\minu=\mathfrak{n}_{\minu p}\oplus\mathfrak{n}_{\minu (p\minu 1)}\oplus\cdots\oplus\mathfrak{n}_{\minu 1}$$
    via $$\mathfrak{n}_{\minu k}={\mathcal D}^{-k}/{\mathcal D}^{-k+1},$$
    which, at every point $x\in M^N$ defines a $p$-step \emph{nilpotent Lie algebra} $(\mathfrak{n}_\minu(x),[\cdot,\cdot]_x)$ with the Lie bracket $[\cdot,\cdot]_x$ induced by the Lie bracket of vector fields on $M^N$. The Lie algebra $(\mathfrak{n}_\minu(x),[\cdot,\cdot]_x)$ is called the \emph{symbol algebra} (or the \emph{nilpotent aproximation}) of the distribution $H$. It serves as a \emph{local algebraic differential invariant of the distribution} $H$ around $x\in M^N$. 

    Let us, from now on, restrict to the case of distributions whose symbol algebras $\mathfrak{n}_\minu(x)$ are \emph{constant} over $M^N$, i.e such that $\mathfrak{n}_\minu(x)=\mathfrak{n}_\minu$ for all $x\in M^N$. In particular, \emph{homogeneous} distributions are examples of those.

    Every distribution with a constant symbol, has therefore a \emph{unique} $p$-step nilpotent Lie algebra $(\mathfrak{n}_\minu,[\cdot,\cdot])$ associated to it. This characterizes it \emph{algebraically}.

    Although (even locally) there exist \emph{non}equivalent distributions with the same constant symbol, they are all sort of a \emph{perturbation of}, or better to say, they are \emph{modelled on}, a standard distribution with this symbol. This standard distribution is called the \emph{flat model} for distributions with a given symbol $\mathfrak{n}_\minu$. It is naturally defined on a manifold $Nil$, which is the Lie group of the symbol algebra $\mathfrak{n}_\minu$, as follows:
    
    A gradation $\mathfrak{n}_\minu=\mathfrak{n}_{\minu p}\oplus\mathfrak{n}_{\minu (p\minu 1)}\oplus\cdots\oplus\mathfrak{n}_{\minu 1}$ in the symbol algebra $\mathfrak{n}_\minu$, is mirrored in the Lie algebra
$$\mathfrak{n}^{Nil}_\minu=\mathfrak{n}^{Nil}_{\minu p}\oplus\mathfrak{n}^{Nil}_{\minu (p\minu 1)}\oplus\dots\oplus\mathfrak{n}^{Nil}_{\minu 1}$$
    of the left invariant vector fields on $Nil$, which is isomorphic to $\mathfrak{n}_\minu$. Then this defines the filtration in $Nil$, with the filtered components ${\mathcal D}^{-k}$ spanned over all smooth \emph{functions} $f\in{\mathcal F}(Nil)$ on $Nil$, by those left invariant vector fields on $Nil$, which belong to $\mathfrak{n}^{Nil}_{\minu k}\oplus \mathfrak{n}^{Nil}_{\minu k+1}\oplus\dots\oplus \mathfrak{n}^{Nil}_{\minu 1}$, namely
$${\mathcal D}^{-k}=\Span_{{\mathcal F}(Nil)}(\mathfrak{n}^{Nil}_{\minu k}\oplus \mathfrak{n}^{Nil}_{\minu k+1}\oplus\dots\oplus \mathfrak{n}^{Nil}_{\minu 1}).$$
    It follows that the first step $H={\mathcal D}^{-1}$ in this filtration is the distribution on $Nil$ with the symbol $\mathfrak{n}_\minu$. This serves as a \emph{flat model} for all the distributions with the constant symbol $\mathfrak{n}_\minu$.

    The symbol algebra $\mathfrak{n}_\minu$ of the distribution $H$ is, as its alternative name suggests, its \emph{algebraic approximation}. In particular, it captures information about the local properties of the \emph{maximal possible} group of automorphisms of all distributions with a given \emph{constant} symbol $\mathfrak{n}_\minu$.

    It is not a surprise that the maximal symmetry for all distributions with a constant symbol algebra $\mathfrak{n}_\minu$, is realized for the natural distribution structure  $(Nil,H={\mathcal D}^{-1})$ on the nilpotent Lie group $Nil$ associated with $\mathfrak{n}_\minu$. Moreover, the algebraic structure of the maximal group of automorphisms of distributions $H$ with a constant symbol $\mathfrak{n}_\minu$, namely the maximal Lie algebra of the automorphisms $\mathfrak{aut}(H)$ for all these distributions $H$, is obtained from the symbol $\mathfrak{n}_\minu$, by an algebraic procedure called the \emph{Tanaka prolongation}.  This goes as follows:

    We start with a symbol algebra $\mathfrak{n}_\minu$, which is a $p$-\emph{step nilpotent Lie algebra}, i.e. a real Lie algebra $\mathfrak{n}_\minu$, which is $p$-\emph{graded} in the sense that it is a \emph{direct sum} 
    $$\mathfrak{n}_\minu=\mathfrak{n}_{\minu p}\oplus\mathfrak{n}_{\minu (p\minu 1)}\oplus\cdots\oplus\mathfrak{n}_{\minu 1}$$
    of $p$ \emph{vector spaces} $\mathfrak{n}_{\minu j}$, $j=1,2,\dots,p$, and that it is equipped with a Lie bracket $[\cdot,\cdot]$, such that
    $$[\mathfrak{n}_{\minu i},\mathfrak{n}_{\minu j}]\subset\begin{cases}\mathfrak{n}_{\minu (i+j)}&\mathrm{if }\,\, 2\leq i+j\leq p\\\{0\}&\mathrm{if }\,\, p<i+j\end{cases}.$$
    The \emph{Tanaka prolongation of} $\mathfrak{n}_\minu$ is a graded Lie algebra given by a direct sum
    \begin{equation}
      \mathfrak{g}_T(\mathfrak{n}_\minu)=\mathfrak{n}_\minu\oplus\mathfrak{n}_0\oplus\mathfrak{n}_1\oplus\dots\oplus\mathfrak{n}_j\oplus\cdots,\label{gt1}\end{equation} with
      \begin{equation}\mathfrak{n}_k=\Big\{\bigoplus_{j<0}\mathfrak{n}_{k+j}\otimes\mathfrak{n}_j^*\ni A\,\,\mathrm{s.t.}\,\,A[X,Y]=[AX,Y]+[X,AY]\Big\}\label{gt}\end{equation}
    for each $k\geq 0$.
    In particular, $\mathfrak{n}_0$ is the Lie algebra of \emph{all derivations of} $\mathfrak{n}_\minu$ preserving its direct-sum-of-vector-spaces-$\mathfrak{n}_{\minu j}$ structure. Setting $[A,X]=AX$ for all $A\in \mathfrak{n}_k$ with $k\geq 0$ and for all $X\in\mathfrak{n}_\minu$ makes the condition in \eqref{gt} into the Jacobi identity. Moreover, if $A\in \mathfrak{n}_k$ and $B\in\mathfrak{n}_l$, $k,l\geq 0$, then their commutator $[A,B]\in\mathfrak{n}_{k+l}$ is defined on elements $X\in\mathfrak{n}_\minu$ inductively, according to the Jacobi identity. By this we mean that it should satisfy  
    $$[A,B]X=[A,BX]-[B,AX],$$
    which is sufficient enough to define $[A,B]$. 
    The Tanaka prolongation $\mathfrak{g}_T(\mathfrak{n}_\minu)$ is uniquely determined by the nilpotent Lie algebra $\mathfrak{n}_\minu$. It may happen that, given $\mathfrak{n}_\minu$, the sum in its Tanaka prolongation \eqref{gt1} is infinite. There are however $\mathfrak{n}_\minu$ for which the Tanaka prolongation is finite. In particular, there are known situations when the Tanaka prolongation
$$\mathfrak{g}=\mathfrak{g}_T(\mathfrak{g}_\minu)$$
of the $p$-step nilpotent part $$\mathfrak{g}_\minu=\mathfrak{g}_{\minu p}\oplus\dots\oplus\mathfrak{g}_{\minu 1}$$ is \emph{symmetric}, in the sense   
$$\mathfrak{g}_T(\mathfrak{g}_\minu)=\mathfrak{g}_{\minu p}\oplus\dots\oplus\mathfrak{g}_{\minu 1}\oplus\mathfrak{g}_0\oplus\mathfrak{g}_1\oplus\dots\oplus\mathfrak{g}_p,$$
with
$$\dim(\mathfrak{g}_{-k})=\dim(\mathfrak{g}_k), \quad k=1,2,\dots,p,$$
  and when the so defined Lie algebra $\mathfrak{g}_T(\mathfrak{g}_\minu)$ is \emph{simple}. In such case the Tanaka prolongation defines a \emph{gradation} of this simple Lie algebra, and the subalgebra $$\mathfrak{p}=\mathfrak{g}_0\oplus\mathfrak{g}_1\oplus\dots\oplus\mathfrak{g}_p,$$
  in such $\mathfrak{g}_T(\mathfrak{g}_\minu)$ is \emph{parabolic}.

The following theorem is due to Noboru Tanaka:
\begin{theorem}\label{tansym}
  Consider distribution structures $(M^N,H)$, with distributions $H$ defining a $p$-step filtration as in \eqref{pstep1}-\eqref{pstep2} and having the same constant symbol $\mathfrak{n}_\minu$. Then
  \begin{itemize}
  \item The most symmetric of all of these distribution structures is $(Nil, H={\mathcal D}^{-1})$, with $Nil$ being a nilpotent Lie group associated of the symbol algebra $\mathfrak{n}_\minu$, and with $H$ being the first component ${\mathcal D}^{-1}$ of the natural filtration on $Nil$ associated to the $p$-step grading in $\mathfrak{n}_\minu$.
  \item The Lie algebra of automorphisms $\mathfrak{aut}(H)$ of the flat model structure $(Nil, H={\mathcal D}^{-1})$ is isomorphic to the Tanaka prolongation $\mathfrak{g}_T(\mathfrak{n}_\minu)$ of the symbol algebra $\mathfrak{n}_\minu$,
$\mathfrak{aut}(H)\simeq \mathfrak{g}_T(\mathfrak{n}).$   
    \end{itemize}
\end{theorem}
Let us now return to Example \ref{ex1}.
\begin{example}
  Note that in Example \ref{ex1}, the vector fields $X_1,X_2,X_3,X_4$ and $X_5=\partial_5$, $X_6=\partial_6$, $X_7=\partial_7$ span a 2-step nilpotent Lie algebra
  $$\mathfrak{n}=\Span_\bbR(X_1,X_2,\dots, X_7)=\mathfrak{n}_{\minu 2}\oplus\mathfrak{n}_{\minu 1}.$$
  Here the graded components are
  $$\begin{aligned}
    \mathfrak{n}_{\minu 1}&=\Span_\bbR(X_1,X_2,X_3,X_4)\\
    \mathfrak{n}_{\minu 2}&=\Span_\bbR(X_5,X_6,X_7).\end{aligned}$$
  According to our discussion above, we can now built a flat model of rank 4 distributions $H$ in dimension 7 with symbol $\mathfrak{n}$. This model distribution can be identified with the original distribution $H$, as in \eqref{babyh} in Example \ref{ex1}. Using the commutation relations for the basis $(X_1,X_2,\dots,X_7)$ and the definition of the Tanaka prolongation discussed above, we find that the Tanaka prolongation for $\mathfrak{n}$ and, as a consequence, the Lie algebra of automorphisms of $H$ is,
  $$\mathfrak{aut}(H)=\mathfrak{g}_T(\mathfrak{n})=\mathfrak{n}_{\minu 2}\oplus\mathfrak{n}_{\minu 1}\oplus\mathfrak{n}_{0}\oplus\mathfrak{n}_1\oplus\mathfrak{n}_2,$$
  with the submodules $\mathfrak{n}_k$ of respective dimensions
  $$\begin{aligned}
   \dim( \mathfrak{n}_{\pm 2})=3\\
   \dim( \mathfrak{n}_{\pm 1})=4\\
    \dim(\mathfrak{n}_0)=7.
  \end{aligned}$$
  One can recognize that in this Tanaka prolongation of $\mathfrak{n}$ the homogeneity 0 component $\mathfrak{n}_0=\bbR\oplus 2\sua(2)$, and that the full Tanaka prolongation is $\mathfrak{g}_T(\mathfrak{n})\simeq\spa(1,2)$. This, via the Tanaka Theorem \ref{tansym}, confirms the claim that the local group of automorphisms of the distribution structure $(M,H)$ form Example \ref{ex1} is $\spg(1,2)$ as claimed before.
\end{example}
\subsection{Decorated distributions}
Distribution structures are perhaps the simplest geometric structures one can define on a smooth manifold. Note, for example that any smooth manifold is naturally equipped with an $m$-distribution structure $(M,H)$, where $m$ is the dimension of $M$ and the the distribution $m$-plane $H_p$ at each point $p\in M$ is the entire tangent space $\mathrm{T}_pM$ at $p$.

One obtains more exciting geometries when one \emph{decorates} a distribution structure $(M,H)$, or the distribution $H$, with the same kinds of \emph{geometric objects} at each point $p\in M$. Such objects can be, for example a \emph{metric} $g$ on $H$, or a \emph{skew symmetric} form $\omega$ on $H$, or more generally a tensor $\ten$, or families of tensors such as e.g. \emph{pencils} of tensors, on $M$. Then the distribution structure $(M,H)$ decorated in this way, is given by a \emph{triple} $(M,H,\ten)$, where $\ten$ is an appropriate object defined on $H$. Frequently some \emph{integrability conditions} for the decorating object field $\ten$ on $H$ are also required (see below). 

Given a distribution structure $(M,H,\ten)$ decorated by a \emph{tensor field} $\ten$ on $H$ one defines an (local) equivalence of two such decorated structures $(M,H_1,\ten_1)$ and $(M,H_2,\ten_2)$ on $M$, by saying that they are (locally) equivalent if they are (locally) equivalent as distribution structures and if the distribution structure equivalence diffeomorphism $\phi:M_1\to M_2$ transforms the tensor field $\ten_1$ from the first distribution $H_1$ to the corresponding tensor field $\ten_2$ on $H_2$, $\phi^*\ten_2=\ten_1$. Obviously, one can speak about the (local) group $G_\ten$ of automorphisms of a decorated distribution structure $(M,H,\ten)$, as well as about the corresponding Lie algebra $\mathfrak{g}_\ten$ of infinitesimal symmetries of $(M,H,\ten)$: this is generated by those vector fields $Y$ from the Lie algebra $\mathfrak{g}$ of infinitesimal symmetries of the distribution structure $(M,H)$, which additionally preserve $\ten$ on $H$. It is also obvious that in most of decorations the group $G_\ten$ will be a \emph{proper} subgroup of $G$, and the Lie algebra $\mathfrak{g}_\ten$ will be a \emph{proper} subalgebra of $\mathfrak{g}$. We further discuss more specific issues associated with the decorations of distribution structures on a \emph{particular class of decorations} $(M,H,J)$, which are termed \emph{Cauchy-Riemann structures}, or \emph{CR} structures, for short. 

\subsubsection{CR structures of type $(n,k)$} An \emph{almost} CR structure of \emph{CR dimension} $n$ and \emph{CR codimension} $k=(m-2n)$ on an $m$-dimensional manifold $M$ equipped with a bracket generating $(2n)$-distribution $H$ is a \emph{decoration} $(M,H,J)$ of the structure $(M,H)$ with a linear operator $J:H\to H$, such that $J^2=-\mathrm{id}_{|H}$. An almost CR structure of CR dimension $n$ and CR codimension $k$ is called a \emph{CR structure} of CR dimension $n$ and CR codimension $k$ if and only if the \emph{integrability conditions} are satisfied for the pair $(H,J)$, i.e. if and only if, for every two vector fields $X$ and $Y$ belonging to $H$, $X,Y\in H$, the following two conditions are satisfied:
\begin{itemize}
\item the difference of the commutators $[X,Y]$ and $[JX,JY]$ is in $H$, \begin{equation}\big([X,Y]-[JX,JY]\big)\in H;\label{ni1}\end{equation}
\item the \emph{vanishing Nijenhuis tensor condition} is satisfied, namely
  \begin{equation}J\big([X,Y]-[JX,JY]\big)=[JX,Y]+[X,JY].\label{ni2}\end{equation}
   \end{itemize}

For further use let us introduce a \emph{convenient terminology}: An (almost) CR structure $(M,H,J)$ of CR dimension $n$ and CR codimension $k$ will be called an (almost) \emph{CR structure of type} $(n,k)$. For these structures we obviously have that rank $\ell$ of $H$ is $\ell=2n$ and the dimension of the manifold $M$ is $m=2n+k$.

An almost CR structure of type $(n,k)$ can be also defined via the \emph{dual picture}, i.e. starting from the codistribution $H^\perp$ which is annihilated by $H$. For this, in the complexification $(\mathrm{T}^*M)^\bbC$ we need a subbundle $Z^*$ of \emph{complex} rank $(n+k)$ such that
$$(H^\perp)^\bbC\subset Z^*\subset(\mathrm{T}^*M)^\bbC\quad \mathrm{and}\quad  (H^\perp)^\bbC\dz Z^*\dz\bar{Z}^*=\wedge^{(2n+k)} (\mathrm{T}^*M)^\bbC;$$
here $\bar{Z}^*=\{(\mathrm{T}^*M)^\bbC\ni\mu\,\,\mathrm{s.t.}\,\,\bar{\mu}\in Z^*\}$. Now, $J$ is defined in $H$ by saying that it is a real operator in $H$ such that when it is complexified, it acts as $JZ=iZ$ on all vector fields $Z$ in the annihilator of $\bar{Z}^*$.

Since this is a bit complicated, let us return to our example \ref{ex1}.
\begin{example}\label{ex2} {\bf Continuation of Example \ref{ex1}}:
  In Example \ref{ex1} we have the distribution $H$ defined as the annihilator of $H^\perp$ spanned by the three forms $\lambda^1,\lambda^2,\lambda^3$ as given in \eqref{baby}.
  We define $Z^*=(H^\perp)^\bbC+W^*$
  with $W^*=\Span(\mu^1,\mu^2)$, where
  \begin{equation}
    \begin{aligned}
      \mu^1&=\der x^1+i\der x^4,\\
      \mu^2&=\der x^2-i\der x^3.
      \end{aligned}
    \end{equation}
  Here $i$ denotes the \emph{imaginary unit}, $i=\sqrt{-1}$, and $\Span$ is taken over the complex-valued functions on $M=\bbR^7$. This results in
  $$H^\perp\subset Z^*=\Span(\lambda^1,\lambda^2,\lambda^3,\mu^1,\mu^2),$$
  where again the $\Span$ is taken over the complex-valued functions in $\bbR^7$
  
  The \emph{flag} $(H^\perp)^\bbC\subset Z^*\subset (\mathrm{T}^*M)^\bbC$ defines $\bar{Z}^*=\Span(\lambda^1,\lambda^2,\lambda^3,\bar{\mu}^1,\bar{\mu}^2)$, with the `bar' operator on complex 1-forms denoting the usual complex conjugation, as for example in $\bar{\mu}^1=\der x^1-i\der x^4$. Furthermore, we have $$(H^\perp)^\bbC\dz Z^*\dz\bar{Z}^*=\Span(\der x^1\dz\der x^2\dz\der x^3\dz\der x^4\dz\der x^5\dz\der x^6\dz\der x^7)=\dz^7 (\mathrm{T}^*\bbR^7)^\bbC.$$
  Thus, according to the brief procedure above, we have an \emph{almost} CR structure of type $(2,3)$ in $M=\bbR^7$. To see how the \emph{complex structure operator in } $H$ looks like, note the following:
  
  The annihilator of $\bar{Z}^*$ is
  $$Z=\Span(Z_1,Z_2),$$
  with
  $$Z_1=X_1-i X_4\quad\mathrm{and}\quad Z_2=X_2+iX_3,$$
  where we have used the basis $(X_1,X_2,X_3,X_4)$ for $H$ with the vector fields $X_i$ as defined in \eqref{baby*}. We therefore have $JZ_1=iZ_1$ and $JZ_2=iZ_2$. Looking at the real and imaginary parts of these equations we find out the following action of the \emph{real} operator $J$ on the basis $(X_1,X_2,X_3,X_4)$ of $H$:
  $$JX_1=X_4, \quad JX_2=-X_3, \quad JX_3=X_2,\quad JX_4=-X_1.$$
  It is visible that $J^2=-id$ on $H$, so we really have an \emph{almost} CR structure $(\bbR^7,H,J)$ of type $(2,3)$ in $\bbR^7$.

  \emph{Interestingly} this \emph{almost} CR structure is actually an \emph{honest CR structure} of type $(2,3)$. One easily checks that the Nijenhuis integrability conditions \eqref{ni1}-\eqref{ni2} are satisfied.
\end{example}

Returning to the general case of an \emph{almost} CR structure of type $(n,k)$ we note the following usefulness of our formulation of a CR decoration of a distribution $H$ in terms of the flag $(H^\perp)^\bbC\subset Z^*\subset (\mathrm{T}^*M)^\bbC$. The Nijenhuis integrability conditions \eqref{ni1}-\eqref{ni2} of an almost CR structure defined on $M$ by such a flag are \emph{equivalent} to a single condition:\\
\centerline{\emph{differentials of all forms from $Z^*$ are in the ideal}} \
\centerline{\emph{generated by forms from $Z^*$.}}
The \emph{ideal} here is in the Grassmann algebra of all (complex-valued) skew symmetric forms on $M$, whose multiplication is the wedge product of forms. 

This leads to the following operational and easy to use way of checking the Nijenhuis conditions \eqref{ni1}-\eqref{ni2} for $(M,H,J)$:

Let $(\lambda^1,\lambda^2,\dots,\lambda^k)$ be a basis of 1-forms in $H^\perp$, and let $(\lambda^1,\lambda^2,\dots,\lambda^k,\mu^1,\mu^2,\dots,$ $\mu^n)$ be its extension to a basis of 1-forms in $Z^*$. Then the almost CR structure $(M,H,J)$ corresponding to the flag $(H^\perp)^\bbC\subset Z^*\subset (\mathrm{T}^*M)^\bbC$ satisfies the Nijenhuis integrability conditions \eqref{ni1}-\eqref{ni2} if and only if
\begin{equation}
  \begin{aligned}
    \der \lambda^i\dz\lambda^1\dz\lambda^2\dz\dots\dz\lambda^k\dz\mu^1\dz\mu^2\dz\dots\dz\mu^n=0, \quad&\mathrm{for\,\,all}\quad i=1,2,\dots, k,\\\mathrm{and}\quad\quad\quad\quad\quad\quad\quad&\\
    \der \mu^\alpha\dz\lambda^1\dz\lambda^2\dz\dots\dz\lambda^k\dz\mu^1\dz\mu^2\dz\dots\dz\mu^n=0, \quad&\mathrm{for\,\,all}\quad \alpha=1,2,\dots,n.
    \end{aligned}
  \end{equation}
In particular, using these conditions, one can easily check that the almost CR structure from Example \ref{ex2} is an integrable CR structure of type $(2,3)$ on $M=\bbR^7$.

Using the flag $(H^\perp)^\bbC\subset Z^*\subset (\mathrm{T}^*M)^\bbC$ formulation of the concept of an almost CR structure, we can also easily define the concept of (local) equivalence of almost CR structures and the concept of (local) symmetry group:

We say that two almost CR structures given by the respective two flags $(H_1^\perp)^\bbC\subset Z_1^*\subset (\mathrm{T}^*M)^\bbC$ and $(H_2^\perp)^\bbC\subset Z_2^*\subset (\mathrm{T}^*M)^\bbC$ on $M$ are (locally) equivalent, if and only if there exists a (local) diffeomorphism $\phi:M\to M$ transforming one flag to the other, i.e. such that 
$$\phi^*\big((H_2^\perp)^\bbC\subset Z_2^*\subset (\mathrm{T}^*M)^\bbC\big)=(H_1^\perp)^\bbC\subset Z_1^*\subset (\mathrm{T}^*M)^\bbC.$$
We further say that a (local) diffeomorphism $\phi:M\to M$ is a (local) \emph{CR automorphism} for an almost CR structure $(M,H,J)$ defined by $(H^\perp)^\bbC\subset Z^*\subset (\mathrm{T}^*M)^\bbC$ if and only if it satisfies
$$\phi^*\big((H^\perp)^\bbC\subset Z^*\subset (\mathrm{T}^*M)^\bbC\big)=(H^\perp)^\bbC\subset Z^*\subset (\mathrm{T}^*M)^\bbC.$$
Note that, in terms of the bases $(\lambda^1,\lambda^2,\dots,\lambda^k)$ in $H^\perp$ and $(\lambda^1,\lambda^2,\dots,\lambda^k,\mu^1,\mu^2,\dots,\mu^n)$ in $Z^*$, the (local) CR automorphism can be equivalently defined as a (local) diffeomorphism $\phi:M\to M$ satisfying: 
$$
  \begin{aligned}
    &\phi^* \big(\lambda^i\big)\dz\lambda^1\dz\lambda^2\dz\dots\dz\lambda^k=0, \quad&\mathrm{for\,\,all}\quad i=1,2,\dots k,\\&\quad\quad\mathrm{and}\\
    &\phi^*\big(\mu^\alpha\big)\dz\mu^1\dz\mu^2\dz\dots\dz\mu^n\dz\lambda^1\dz\lambda^2\dz\dots\dz\lambda^k=0, \quad&\mathrm{for\,\,all}\quad \alpha=1,2,\dots n.
  \end{aligned}
  $$
Consequently, the (local) \emph{group} $G_J$ \emph{of CR automorphisms} consists of all such $\phi$'s on $(M,H,J)$ with the composition of CR automorphisms as the group multiplication. If the group $G_J$ of CR automorphisms acts \emph{transitively} on $M$ then (locally) $M=G_J/P$, with isotropy $P$, and the CR structure $(M,H,J)$ is called \emph{homogeneous}.

The \emph{infinitesimal symmetries of an almost CR structure} (infinitesimal CR automorphisms) are \emph{real} vector fields $Y$ on $M$ which preserve the flag $(H^\perp)^\bbC\subset Z^*\subset (\mathrm{T}^*M)^\bbC$. Their most convenient equivalent definition is again given in terms of the bases $(\lambda^1,\lambda^2,\dots,\lambda^k)$ in $H^\perp$ and $(\lambda^1,\lambda^2,\dots,\lambda^k,\mu^1,\mu^2,\dots,\mu^n)$ in $Z^*$. We note that a vector field $Y$ on $M$ is an infinitesimal symmetry of an almost CR structure given on $M$ by $(H^\perp)^\bbC\subset Z^*\subset (\mathrm{T}^*M)^\bbC$ if and only if $Y$ satisfies the following PDEs with respect to the bases $(\lambda^1,\lambda^2,\dots,\lambda^k)$ in $H^\perp$ and $(\lambda^1,\lambda^2,\dots,\lambda^k,\mu^1,\mu^2,\dots,\mu^n)$ in $Z^*$:
\begin{equation}
  \begin{aligned}
&\big({\mathcal L}_Y\lambda^i\big)\dz\lambda^1\dz\lambda^2\dz\dots\dz\lambda^k=0, \quad\mathrm{for\,\,all}\quad i=1,2,\dots k,\\&\quad\quad\mathrm{and}\\
    &\big({\mathcal L}_Y\mu^\alpha\big)\dz\mu^1\dz\mu^2\dz\dots\dz\mu^n\dz\lambda^1\dz\lambda^2\dz\dots\dz\lambda^k=0, \quad\mathrm{for\,\,all}\quad \alpha=1,2,\dots n.
    \end{aligned}\label{sycr}
\end{equation}
Of course the \emph{Lie algebra} $\mathfrak{g}_J$ \emph{of infinitesimal symmetries} of $(M,H,J)$ is an $\bbR$-linear span of the above vector fields $Y$
with the Lie bracket of vector fields as the Lie algebra $\mathfrak{g}_J$ bracket.

It is worthwhile to note that the basis related \emph{flag} $(H^\perp)^\bbC\subset Z^*\subset (\mathrm{T}^*M)^\bbC$ definition of an (almost) CR structure $(M,H,J)$ is particularly useful when we want to \emph{compare}
the geometry of a structure of a naked distribution $(M,H)$ with that of a decorated distribution structure $(M,H,J)$. For example the Lie algebra of the infinitesimal symmetries of the naked $(M,H)$ structure consists
of real vector fields $Y$ on $M$ such that
\begin{equation}
  \begin{aligned}
    &\big({\mathcal L}_Y\lambda^i\big)\dz\lambda^1\dz\lambda^2\dz\dots\dz\lambda^k=0, \quad&\mathrm{for\,\,all}\quad i=1,2,\dots k,\label{syd}
  \end{aligned}
  \end{equation}
holds,
whereas the Lie algebra of infinitesimal symmetries of an (almost) CR structure $(M,H,J)$ decorating the structure $(M,H)$ will be its subalgebra which, in addition to \eqref{syd}, satisfies also the second part of integrability conditions from \eqref{sycr}, namely all conditions involving $({\mathcal L}_Y\mu^\alpha)$s. We will return to this (trivial) observation in Section \ref{acci}, where it will be crucial to motivate our main object of study in this paper.
\begin{example}\label{ex3} {\bf Continuation of Examples \ref{ex1} and \ref{ex2}}:
  In Example \ref{ex1} we have a distribution structure $(M,H)$ on $M=\bbR^7$ given in terms of the \emph{real} bundle $H^\perp=\Span(\lambda^1,\lambda^2,\lambda^3)$ as in \eqref{baby}. And in Example \ref{ex2} we decorated it with an almost CR structure $(M,H,J)$ by complexifying $H^\perp$ to $(H^\perp)^\bbC$ and by considering $Z^*=\Span_\bbC(\lambda^1,\lambda^2,\lambda^3,\mu^1,\mu^2)$. As we noticed already the almost CR structure $(M,H,J)$ defined by $(H^\perp)^\bbC\subset Z^*\subset (\mathrm{T}^*M)^\bbC$ is \emph{integrable}, which can be easily checked by seeing that
  $$\begin{aligned}
    \der \lambda^1\dz\lambda^1\dz\lambda^2\dz\lambda^3\dz\mu^1\dz\mu^2=0,\\
    \der \lambda^2\dz\lambda^1\dz\lambda^2\dz\lambda^3\dz\mu^1\dz\mu^2=0,\\
    \der \lambda^3\dz\lambda^1\dz\lambda^2\dz\lambda^3\dz\mu^1\dz\mu^2=0,\\
    \der \mu^1\dz\lambda^1\dz\lambda^2\dz\lambda^3\dz\mu^1\dz\mu^2=0,\\
    \der \mu^2\dz\lambda^1\dz\lambda^2\dz\lambda^3\dz\mu^1\dz\mu^2=0.
  \end{aligned}$$
  Furthermore, we know that the structure $(M,H)$ has the 21-dimensional group $\spg(1,2)$ as the group of its local symmetries, with the 21-dimensional Lie algebra $\mathfrak{g}$ of infinitesimal symmetries $Y_\nu$, $\nu=1,2,\dots,21$, as in Example \ref{ex1}. By solving the \emph{CR symmetry} equations (\ref{sycr}) for the CR structure $(M,H,J)$, namely by solving the equations
  \begin{equation}\begin{aligned}
    &\big({\mathcal L}_Y\lambda^1\big)\dz\lambda^1\dz\lambda^2\dz\lambda^3=0, \\
    &\big({\mathcal L}_Y\lambda^2\big)\dz\lambda^1\dz\lambda^2\dz\lambda^3=0, \\
    &\big({\mathcal L}_Y\lambda^3\big)\dz\lambda^1\dz\lambda^2\dz\lambda^3=0, \\
    &\big({\mathcal L}_Y\mu^1\big)\dz\mu^1\dz\mu^2\dz\lambda^1\dz\lambda^2\dz\lambda^3=0,\\
    &\big({\mathcal L}_Y\mu^2\big)\dz\mu^1\dz\mu^2\dz\lambda^1\dz\lambda^2\dz\lambda^3=0,
  \end{aligned}\label{sycr7}\end{equation}
  or by choosing only these symmetries $Y_\nu$, $\nu=1,2,\dots,21$, of $(M,H)$, which in addition to the first three equations in \eqref{sycr7}, also satisfy the last two, we \emph{find that the Lie algebra} $\mathfrak{g}_J$ \emph{of infinitesimal CR symmetries of} $(M,H,J)$ \emph{has dimension} $\dim(\mathfrak{g}_J)=12$, only. So the dimension of the Lie algebra $\mathfrak{g}$ of symmetries of the distribution structure  \emph{drops from 21 to the dimension 12} of the Lie algebra $\mathfrak{g}_J$ of infinitesimal symmetries of the decorated structure $(M,H,J)$. In particular, the infinitesimal distribution structure symmetry generator $Y_{10}$ from Example \ref{ex1} does \emph{not} satisfies the CR symmetry equations 
    $$\begin{aligned}
    &\big({\mathcal L}_Y\mu^1\big)\dz\mu^1\dz\mu^2\dz\lambda^1\dz\lambda^2\dz\lambda^3=0,\\
    &\big({\mathcal L}_Y\mu^2\big)\dz\mu^1\dz\mu^2\dz\lambda^1\dz\lambda^2\dz\lambda^3=0,
  \end{aligned}
  $$
  and thus it is \emph{not} a \emph{CR symmetry} generator. On the other hand the distribution structure symmetry $Y_{12}$ from Example \ref{ex1} satisfies all five conditions for being the CR symmetry, and generates a local \emph{CR automorphism}. So in this example we see explicitly \emph{the typical situation}, in which the symmetry algebra $ \mathfrak{g}_J$ of a decorated structure $(M,H,J)$, is a \emph{proper} subalgebra of the algebra $\mathfrak{g}$ of symmetries of the naked structure $(M,H)$.  
  \end{example}
\subsubsection{Embedding a CR structure of type $(n,k)$ in $\bbC^{(n+k)}$} Let us, from now on, assume that all our CR structures are \emph{real analytic}.

Consider then a CR structure $(M,H,J)$ of type $(n,k)$ defined on a $(2n)$-distribution $H$ via the flag $(H^\perp)^\bbC\subset Z^*\subset (\mathrm{T}^*M)^\bbC$. Let $(\lambda^1,\lambda^2,\dots,\lambda^k)$ and $(\lambda^1,\lambda^2,\dots,\lambda^k,\mu^1,\mu^2,$ $\dots,\mu^n)$ be the respective bases in $H^\perp$ and $Z^*$. A \emph{CR function} $z:M\to \bbC$ is a complex-valued differentiable function on $M$ such that
$$\der z\dz \lambda^1\dz\lambda^2\dz\dots\dz\lambda^k\dz\mu^1\dz\mu^2\dz\dots\dz\mu^n=0.$$
We say that CR functions  $z_1,z_2,\dots,z_j$ are \emph{functionally independent} in an open set $\mathcal U$ if and only if
$$\der z_1\dz\der z_2\dz\dots\dz\der z_j\neq 0\quad\quad\mathrm{in}\quad {\mathcal U}.$$
It is well known \cite{andreotti} that an \emph{analytic} CR structure of type $(n,k)$ always admits $(n+k)$ independent CR functions $z_1,z_2,\dots, z_{(n+k)}$. They provide a local \emph{embedding}
\begin{equation}
  M\ni p\stackrel{\iota}{\longrightarrow} \big(z_1(p),z_2(p),\dots, z_{(n+k)}(p)\big)\in\bbC^{(n+k)}.\label{emb}\end{equation}
Once a CR structure $(M,H,J)$ is embedded like that it provides a CR structure  $\big(\iota(M),\iota_*H,\iota^*J\big)$ of type $(n,k)$ embedded as a submanifold of real codimension $k$ in $\bbC^{(n+k)}$. As a real submanifold it also acquires a CR structure of type $(n,k)$ from the canonical complex structure $I$ (multiplication by an imaginary unit) in the ambient space $\bbC^{(n+k)}$. This is defined by noting that $\iota_*H=I\mathrm{T}(\iota(M))\cap \mathrm{T}(\iota(M))$ and that $\iota^*J=I_{|\iota_*H}$.  Thus both these CR structures on $\iota(M)$ are CR equivalent, and are equivalent to the original abstract CR structure $(M,H,J)$ on $M$. Therefore knowing $(n+k)$ independent CR functions of $(M,H,J)$ we have a nice model of an abstract CR structure: one embeds it by \eqref{emb} as a real submanifold of higher codimension and gets the CR structure on it from the ambient complex space $\bbC^N$.

A \emph{particular class} of embedded CR manifolds $M^{2n+k}\subset\bbC^{(n+k)}$ can be defined in terms of \emph{graphs} of $k$ \emph{real} functions $\Phi^i=\Phi^i(z_1,z_2,\dots,z_n,$ $\bar{z}_1,\bar{z}_2,\dots,\bar{z}_n)$, $i=1,2,\dots,k$ such that $\partial\Phi^1\dz\partial\Phi^2\dz\dots\dz\partial\Phi^k\neq 0$, 
where the linear differential operators $\partial$ and $\bar{\partial}$ act on real-valued differentiable functions $f=f(z,\bar{z})$ as \[\partial f=\sum_{\alpha=1}^n\frac{\partial f}{\partial z_\alpha}\der z_\alpha\,\,\mathrm{and}\,\,
\bar{\partial} f=\sum_{\alpha=1}^n\frac{\partial f}{\partial \bar{z}_\alpha}\der \bar{z}_\alpha.\]
To describe embedded CR manifolds, let us denote holomorphic coordinates in $\bbC^{n+k}$ by $(w,z)=(w_1,w_2,\dots,w_k,z_1,z_2,\dots,z_n)$. Then the CR structures from this particular class are defined by the following embeddings:
\begin{equation} M^{2n+k}\,=\,\{\, \bbC^{n+k}\ni (w,z)\,\, s.t.\,\, Im(w_i)-\Phi^i(z,\bar{z})=0\,\,\,\,\,\forall i=1,2,\dots k\,\}.\label{rigid}\end{equation}
In such case the distribution $H$ on $M^{2n+k}$ is defined as the annihilator of $k$ real 1-forms
$$\lambda^i=\der Re(w_i)+i\big(\bar{\partial}-\partial\big)\Phi^i,\,\,\, i=1,2,\dots,k,$$
Thus in such case we have \[H=\{\mathrm{T}M^{2n+k}\ni X\,\,\mathrm{s.t.}\,\,X\hook \lambda^i=0\,\,\,\forall i=1,2,\dots,k\}.\]
Writing $w_i$s and $z_i$s in terms of their real and imaginary parts,
$w_i=u^i+iv^i$, $z_\alpha=x^\alpha+i y^\alpha$, we then have that $\Phi^i=\Phi^i(x^1,x^2,\dots,x^n,y^1,y^2,\dots,y^n)$ and, in particular, that
\begin{equation}\lambda^i=\der u^i+\sum_{\alpha=1}^n\Big(\frac{\partial\Phi^i}{\partial x^\alpha}\der y^\alpha-\frac{\partial\Phi^i}{\partial y^\alpha}\der x^\alpha\Big),\quad\,\,\forall i=1,2,\dots,k.\label{lambda}\end{equation}
Now the CR manifold $M^{2n+k}$ can be conveniently parametrized by $(2n+k)$ real parameters $(u^i,x^\alpha,y^\alpha)$ in which the distribution $H$ is spanned as $$H=\Span(X_\alpha,Y_\alpha)$$
by the $2n$ real vector fields $X_\alpha$ and $Y_\alpha$ given by
$$X_\alpha=\partial_{x^\alpha}+\sum_{i=1}^k\frac{\partial \Phi^i}{\partial y^\alpha}\partial_{u^i},\quad\quad Y_\alpha=\partial_{y^\alpha}-\sum_{i=1}^k\frac{\partial \Phi^i}{\partial x^\alpha}\partial_{u^i},\quad \alpha=1,2,\dots, n.$$
This $H$ is invariant with respect to the complex structure $I$ from the ambient complex space $\bbC^{n+k}$. Indeed, it follows that $IX_\alpha=Y_\alpha$ and $IY_\alpha=-X_\alpha$ for all $\alpha=1,2,\dots, n$. It also follows that the bundle $Z$ is spanned as
$$Z=\Span(Z_\alpha)$$
by the $\alpha=1,2,\dots,n$ complex-valued vector fields
$$Z_\alpha=\tfrac12(X_\alpha-iY_\alpha)=\partial_{z_\alpha}+i\sum_{i=1}^k\frac{\partial\Phi^i}{\partial z_\alpha}\partial_{u^i},$$
and \emph{is} automatically \emph{integrable} since $[Z_\alpha,Z_\beta]=0$ for all $\alpha,\beta=1,2,\dots,n$.

The complex bundle $Z^*$ is spanned over the complex functions by all $\lambda^i$'s and $\mu^\alpha$'s where $\mu^\alpha=\der z^\alpha$, $\alpha=1,2,\dots,n$.


\begin{example}\label{ex4} {\bf Continuation of Example \ref{ex2}}: One easily  checks that the CR functions equation
  $$\der z\dz\lambda^1\dz\lambda^2\dz\lambda^3\dz\mu^1\dz\mu^2=0$$
  written in terms of the generators $(\lambda^1,\lambda^2,\lambda^3,\mu^1,\mu^2)$ of $Z^*$ for the CR structure $(M,H,J)$ of Example \ref{ex2} has the following \emph{five} independent solutions:
  $$
  \begin{aligned}
    &w_1=x^1+i x^4,\quad w_2=x^2-i x^3,\quad w_3=x^5+\tfrac{i}{2}\big((x^1)^2-(x^2)^2\big),\\&\quad\quad\quad\quad z_1=x^6+i x^1 x^3,\quad z_2=x^7+i(x^2-i x^3)x^2.
  \end{aligned}
  $$
Eliminating variables $x^1,x^2,x^3,x^4,x^5,x^6,x^7$ from $(w_1,w_2,w_3,z_1,z_2)$ we get a real codimension $k=3$ submanifold
  $$
  \iota(M)=\left\{\begin{aligned}\bbC^5&\ni(w_{1},w_{2},w_{3},z_{1},z_{2})\,\,s.t.\\
  &Im(w_1)=\tfrac12\big((Re(z_1))^2-(Re(z_2))^2\big)\\
  &Im(w_2)=-Re(z_1)Im(z_2)\\
  &Im(w_3)=Re(z_1)Re(z_2)\end{aligned}\right\}\subset \bbC^5,
  $$
  which gives a CR embedding of $(\iota(M),H,J)$ in $\bbC^5$.
  According to our formulas above, this $(\iota(M),H,J)$ CR structure, provides a distribution structure $(\iota(M),H)$ having $H^\perp=\Span(\lambda^1,\lambda^2,\lambda^3)$ given by the 1-forms $\lambda^1,\lambda^2,\lambda^3$ 
$$\begin{aligned}
    \lambda^1=\der u^1+x^1\der y^1-x^2\der y^2,\\
    \lambda^2=\der u^2-y^2\der y^1+x^1\der x^2,\\
    \lambda^1=\der u^3+x^2\der y^1+x^1\der y^2.\end{aligned}$$
    We leave to the reader to check that this distribution structure on $\bbR^7$ with coordinates $(x^1,y^1,x^2,y^2,u^1,u^2,u^3)$ is locally equivalent to our distribution structure $(M,H)$ from Example \ref{ex1}. We also leave to the reader to check that the CR structure on the embedded manifold $\iota(M)$ given by the flag $(H^\perp)^\bbC\subset Z^*\subset(\mathrm{T}^*M)^\bbC$ in which $(H^\perp)^\bbC$ is the complexification of the above real $H^\perp$, and the complex bundle $Z^*$ is $Z^*=(H^\perp)^\bbC+\Span(\mu^1,\mu^2)$ with $\mu^1=\der x^1+i \der y^1$ and $\mu^2=\der x^2+i\der y^2$, is locally CR equivalent to the CR structure $(M,H,J)$ from Example \ref{ex2}.  
  \end{example}

\subsection{Accidental decorations}\label{acci}
The CR decoration $(M,H,J)$ in Example \ref{ex2}, which was put on the distribution structure $(M,H)$ from Example \ref{ex1} was by no means canonical. Let us consider the following deformation of the CR structure $(M,H,J)$ defined, on the distribution structure $(M,H)$ from Example \ref{ex1}, in Example \ref{ex2}.
\begin{example}\label{ex5}
 Let us return to Example \ref{ex2} and consider three real \emph{constants} $a,b,c$ such that \begin{equation} a^2+b^2+c^2=1.\label{abc}\end{equation}
If $b^2+c^2\neq 0$, instead of $Z^*$ from example \ref{ex2}, we take $Z^*_{(a,b,c)}=(H^\perp)^\bbC+W^*_{(a,b,c)}$, with $W^*_{(a,b,c)}$ spanned by the 1-forms
$$\begin{aligned}
  \mu^1&=\der x^1-i a\der x^2-ic\der x^3+ib\der x^4,\\
  \mu^2&=ia\der x^1+\der x^2-ib\der x^3-ic\der x^4.
\end{aligned}$$
If $b=c=0$, we take $Z^*_{(a,0,0)}=(H^\perp)^\bbC+W^*_{(a,0.0)}$ with $W^*_{(a,0,0)}$ spanned by the 1-forms
$$\begin{aligned}
  \mu^1&=\der x^1-i a \der x^2,\\
  \mu^2&=\der x^3+i a\der x^4.\end{aligned}\quad\quad\quad a=\pm1,$$
Then it follows that \emph{every choice} of three constants $(a,b,c)$ as in \eqref{abc}, via the flag $(H^\perp)^\bbC\subset Z^*_{(a,b,c)}\subset (\mathrm{T}^*M)^\bbC$, decorates the distribution $H$ from Example \ref{ex1} with an \emph{integrable} CR structure $(M,H,J_{(a,b,c)})$. The one from Example \ref{ex2} corresponds to the choice $(a,b,c)=(0,1,0)$. 
This shows that a distribution structure $(M,H)$ may have \emph{many integrable} CR decorations.

We further note that \emph{all} CR structures $(M,H,J_{(a,b,c)})$ have the local group of CR automorphisms $G_{J_{(a,b,c)}}$ with symmetry dimension \emph{not} larger than 12. So we are in the \emph{typical} situation: the dimension 12 of the local symmetry group $G_{J_{(a,b,c)}}$ of the CR structure $(M,H,J_{(a,b,c)})$ is \emph{smaller} than the dimension 21 of the local symmetry group $G$ of the naked distribution structure $(M,H)$.  
\end{example}
In this paper we focus on quite different situations. Roughly, we aim to give examples of distribution structures $(M,H)$, which are such that they \emph{canonically define} nontrivial decorations on $H$. We want that the geometry of the distribution structure alone imposes, \emph{without any additional input}, some decoration. Such decorations we will call \emph{accidental}.

Leaving the precise formulation of the notion of an accidental structure in the most general situations to subsequent studies, here in this paper, we will focus on \emph{accidental CR structures} $(M,H,J)$ \emph{on homogeneous distribution structures} $(M,H)$.

Recall that a typical situation for the respective groups of symmetries $G$ and $G_J$ of a distribution structure $(M,H)$ and a CR structure $(M,H,J)$ is that $G_J\subsetneq G$. The notion of an accidental $(M,H,J)$ requires equality here: 

\begin{definition}\label{rrr}
A CR structure $(M,H,J)$ on a \emph{homogeneous} distribution structure $(M,H)$ is called \emph{accidental} if and only if the group $G_J$ of CR automorphisms of the CR structure $(M,H,J)$ is \emph{equal} to the group $G$ of symmetries of the distribution structure $(M,H)$.  
  \end{definition}

In case of $(M,H)$ with an accidental CR structure $(M,H,J)$ the addition of the complex structure $J$ to $(M,H)$ does not diminish the symmetry of $(M,H)$. Saying it differently, the complex structure $J$ is in a way compatible with the distribution structure, or is somehow cannonically defined by it.

Our two $E_6$ homogenous $CR$ structures from the introduction are examples of \emph{accidental CR structures}. In the next sections we present many more of them. Actually we provide a \emph{full list of accidental CR structures} $(M,H,J)$ for which the distribution $H$ satisfies
\begin{itemize}
\item $[H,H]+H=\mathrm{T}M$ and
\item the geometry of $(M,H,J)$ is a flat model for a parabolic geometry associated with the distribution $H$.
\end{itemize}

On the other hand, there are many \emph{non}accidental homogeneous CR structures. First, all the hypersurface type CR structures, i.e. real hypersurfaces in $\bbC^N$ which acquire their CR structure from the ambient complex space and have nondegenerate Levi form, are \emph{not} accidental: Their \emph{distribution structure} $(M,H)$ is a \emph{contact structure} and as such \emph{has infinite dimensional group} $G$ of local symmetries; in contrast \emph{the group} $G_J$ \emph{of local CR automorphisms} for them \emph{is} always \emph{finite dimensional}.  As our Example \ref{ex5} shows there is also plenty of nonaccidental homogeneous CR structures of higher codimension. 

We emphasize that the notion of an \emph{accidental structure} on a distribution is not solely reserved to the  integrable (or almost) CR structures. This however goes beyond the present work, and we move discussions of this issue to a subsequent paper.
\subsection{A note on a method for approaching the problem}\label{approach}
The restriction of our considerations to accidental CR structures $(M,H,J)$ on distribution structures $(M,H)$ being flat models for parabolic geometries, makes the task of finding a full list of them to be manageable. This is due to the following;

First, \emph{all parabolic geometries are classified}, so one way of finding the suspects, is to search within the available lists.

Second, the \emph{accidental} feature of the objects we search for, requires that the symbol algebra of the searched distribution structures is such that its $\mathfrak{n}_{\minu 1}$ part is \emph{naturally} equipped with an almost structure $J$. This is because this $J$ should naturally induce an almost complex structure $J^{Nil}$ on $H=\mathfrak{n}^{Nil}_{\minu 1}$, which in turn would define the flat model distribution $H$ on the group $Nil$. This immediately \emph{excludes} all nilpotent Lie algebras with \emph{odd} dimension of $\mathfrak{n}_{\minu 1}$.

Actually the situation is much better: restricting to the nilpotent $p$-step Lie algebras with \emph{even} dimension of $\mathfrak{n}_{\minu 1}$, if a $J$ in $\mathfrak{n}_{\minu 1}$ was \emph{not} natural, and we would consider a corresponding almost CR structure $(Nil,H=\mathfrak{n}^{Nil}_\minu,J^{Nil})$ on $Nil$, the symmetry of this almost CR structure would be smaller that the symmetry of $(Nil,H=\mathfrak{n}^{Nil}_\minu)$, since the Tanaka prolongation of $(\mathfrak{n}_{\minu},J)$ would not only preserve the strata $\mathfrak{n}_{k-}$ in $\mathfrak{n}_\minu$, but also $J$ in $\mathfrak{n}_{\minu 1}$. This would chop $\mathfrak{n}_0$  from the situation without $J$ in $\mathfrak{n}_\minu$. It would chop it to a smaller $\mathfrak{n}^J_0$, making the resulting new Tanaka prolongation $\mathfrak{g}^J_T(\mathfrak{n}_\minu)$ smaller than $\mathfrak{g}_T(\mathfrak{n}_\minu)$. For us the crucial information is that in such situation the new $\mathfrak{n}^J_0$ \emph{should preserve} $J$, in the natural adjoint action of $\mathfrak{n}^J_o$ in $\mathfrak{n}_{\minu 1}$ given by the Lie bracket in  $\mathfrak{g}^J_T(\mathfrak{n}_\minu)$. Therefore $\mathfrak{n}^J_0$ should naturally be a \emph{subalgebra of a unitary algebra} $\sua(\mathfrak{n}_{\minu 1})$ for $J$.

This last observation forces us to restrict to the nilpotent Lie algebras, which on top of producing flat models of parabolic geometries, should have, in their Tanaka prolongation, 
their {\em $\mathfrak{n}_0$ part as a subalgebra of a unitary algebra acting in $\mathfrak{n}_{\minu 1}$}. This selects quite a small subset of all $\mathfrak{n}$'s defining gradations in simple Lie algebras. The addition of the requirement that the searched $\mathfrak{n}$'s must be 2-step graded, finishes the job, and one gets the full list. Note, that by this algebraic approach, we are guaranteed that we found accidental \emph{almost} CR manifolds. Interestingly, all the found ones, the ones which appear in our list, are the true
CR manifolds, with corresponding $J$'s satisfying the integrability conditions \eqref{ni1}-\eqref{ni2}.

\section{Flat parabolic CR structures with symmetry algebras  $\mathfrak{a}_{\ell}$, $\mathfrak{d}_{\ell}$ and $\mathfrak{e}_6$}
\subsection{The case of $E_{II}$ and $E_{III}$ symmetry} This was discussed in the Introduction in Theorems \ref{e2} and \ref{e3}. Here we mention that the explicit formulae for the embedding of CR manifolds depend on the convenience or, in most cases, on the personal taste of the person who embedds them. In particular, if one guesses the embeddings of the CR manifold of Theorem \ref{e3} by analyzing the structure of the simple roots of an appropriate real form of $E_6$, a more natural embedding of the 24-dimensional CR manifold $M^{24}_{E_{III}}$ can be given:

In $\mathbb{C}^{8 + 8}$ with holomorphic coordinates $(z_i,w_i)$, $i=1,2,\dots,8$, consider 
\[
\aligned
\tilde{M}^{24}_{E_{III}}\,=\,\, \Big\{\mathrm{Re}\,w_1
&
\,=\,
\mathrm{Re}\,
\big(
\vert z_1\vert^2
+
\vert z_2\vert^2
+
\vert z_3\vert^2
+
\vert z_4\vert^2
\big),
\\
\mathrm{Re}\,w_2
&
\,=\,
\mathrm{Re}\,
\big(
\vert z_5\vert^2
+
\vert z_6\vert^2
+
\vert z_7\vert^2
+
\vert z_8\vert^2
\big),
\\
\mathrm{Im}\,w_3
&
\,=\,
\mathrm{Im}\,
\big(
z_1\,\overline{z}_7
+
z_2\,\overline{z}_8
+
z_5\,\overline{z}_3
+
z_6\,\overline{z}_4
\big),
\\
\mathrm{Re}\,w_4
&
\,=\,
\mathrm{Re}\,
\big(
z_1\,\overline{z}_7
+
z_2\,\overline{z}_8
+
z_5\,\overline{z}_3
+
z_6\,\overline{z}_4
\big),
\\
\mathrm{Im}\,w_5
&
\,=\,
\mathrm{Im}\,
\big(
z_1\,\overline{z}_6
-
z_3\,\overline{z}_8
+
z_5\,\overline{z}_2
-
z_7\,\overline{z}_4
\big),
\\
\mathrm{Re}\,w_6
&
\,=\,
\mathrm{Re}\,
\big(
z_1\,\overline{z}_6
-
z_3\,\overline{z}_8
+
z_5\,\overline{z}_2
-
z_7\,\overline{z}_4
\big),
\\
\mathrm{Im}\,w_7
&
\,=\,
\mathrm{Im}\,
\big(
z_2\,\overline{z}_6
+
z_3\,\overline{z}_7
-
z_5\,\overline{z}_1
-
z_8\,\overline{z}_4
\big),
\\
\mathrm{Re}\,w_8
&
\,=\,
\mathrm{Re}\,
\big(
z_2\,\overline{z}_6
+
z_3\,\overline{z}_7
-
z_5\,\overline{z}_1
-
z_8\,\overline{z}_4
\big)\Big\}\subset\bbC^{16}.
\endaligned
\]
Then this embedded 24-dimensional CR manifold of CR dimension $n=8$ and CR codimension $k=8$ is biholomorphically equivalent to the one from Theorem \ref{e3}, and as such has the exceptional simple Lie group $E_{III}$ as its group of CR automorphisms.
\subsection{The case of $SO(\ell-1,\ell+1)$ symmetry}
\label{subsection-SO-ell-1-ell-1}
Let $\ell\geq 4$ and $N(\ell)=\frac{(\ell-1)(\ell+2)}{2}$, and consider $\bbC^{\ell(\ell-1)/2}$ with holomorphic coordinates $(w,z)=(w^{ij},z^k)$. Here $i<j,k=1,2,\dots,\ell-1$.

Define 

$$\boxed{\begin{aligned}
  M^{N(\ell)}\,=\,\, \Big\{~&\bbC^{\ell(\ell-1)/2}\ni (w,z)\,\,\,\,\mathrm{s.t.}\,\,\,\, Im(w^{ij}-z^i\bar{z}^j)=0\,\,\quad\forall i<j=1,2,\dots,\ell-1\,\,\Big\}
\end{aligned}}.$$

We have the following theorem.
\begin{theorem}\label{so}
   Let $\ell\geq 4$. The set $M^{N(\ell)}\subset \bbC^{\ell(\ell-1)/2}$ is a real $N(\ell)$-dimensional embedded CR manifold, acquiring the CR structure of CR dimension $n=\ell-1$ and CR codimension $k=\frac{(\ell-1)(\ell-2)}{2}$ from the ambient complex space $\bbC^{\ell(\ell-1)/2}$. Its local group of CR automorphisms is isomorphic to the real simple Lie group $SO(\ell-1,\ell+1)$ with the Lie algebra having the Satake diagram \tikzset{/Dynkin diagram/fold style/.style={stealth-stealth,thin,
shorten <=1mm,shorten >=1mm}}\begin{dynkinDiagram}[edge length=.5cm]{D}{ooo.oooo}
\dynkinFold{6}{7}
  \end{dynkinDiagram} with $\ell$ nodes. It is locally CR equivalent to the flat model $SO(\ell-1,\ell+1)/P_1$ of an $N(\ell)$-dimensional parabolic geometry of type $(SO(\ell-1,\ell+1),P_1)$, where the real parabolic subgroup $P_1$ in $SO(\ell-1,\ell+1)$ is determined by the following crossing on the corresponding $\mathfrak{d}_\ell$ Satake diagram:
\tikzset{/Dynkin diagram/fold style/.style={stealth-stealth,thin,
shorten <=1mm,shorten >=1mm}}\begin{dynkinDiagram}[edge length=.5cm]{D}{ooo.oott}
\dynkinFold{6}{7}
 \end{dynkinDiagram}.
\end{theorem}

\begin{example}\label{exso} 
If $\ell=4$ we have a 9-dimensional CR manifold $M^9$, of \emph{CR dimension} $n=3$ and \emph{CR codimension} $k=3$, embedded in $\bbC^6$ with holomorphic coordinates $(w^{12},w^{13},w^{23},z^1,z^2,z^3)$, via $k=3$ real equations
  $$Im(w^{12}-z^1\bar{z}^2)=Im(w^{13}-z^1\bar{z}^3)=Im(w^{23}-z^2\bar{z}^3)=0.$$
  This CR structure has the orthogonal group $SO(5,3)$ as the full group of its local CR automorphisms.

  Our formulae are also valid for $\ell=3$. In this case the CR manifold $M^{N(3)}$ is 5-dimensional; it has CR dimension $n=2$ and CR codimension $k=1$. Therefore it is a CR manifold of \emph{hypersurface type} in $\bbC^3$. Indeed, with holomorphic coordinates $(w^{12},z^1,z^2)$ our formulae give the \emph{hyperquadric} CR structure $Im(w^{12})=\tfrac{1}{2i}(z^1\bar{z}^2-\bar{z}^1z^2)$ in $\bbC^3$. This has the \emph{Levi form} of signature $(+,-)$ and is locally equivalent to the \emph{Penrose's null twistors} CR manifold $PN=\{|Z^1|^2+|Z^2|^2-|Z^3|^2-|W|^2=0\}$ in $\bbC P^3$, with homogeneous coordinates $[Z^1,Z^2,Z^3,W]$. Restricting to $\ell\geq 4$ we excluded it from our Theorem, since this flat parabolic CR structure is already on the \emph{classical list of parabolic CR structures of hypersurface type} corresponding to the contact gradation in $\sua(p,q)$. This is due to the \emph{low dimensional isomorphism} between the Lie algebras $\soa(4,2)$ and $\sua(2,2)$.   
  \end{example}

\subsection{Explicit formulae for the CR symmetry generators}
We also calculated vector fields of infinitesimal 
CR automorphisms generating the symmetry algebra
for the $SO(\ell-1,\ell+1)$ symmetric CRs covered by Theorem \ref{so};
see Appendix~C for explicit formulae.
This was also done for all other homogeneous CR structures included in this paper, but the formulae are too long to be included here. 
They can be found in \cite{HMNN}.

\subsection{The case of $SO^*(2\ell)$, with $\ell=2m+1$, symmetry}
Let $m\geq 2$, $\ell=2m+1$ and $N(m)=m(2m+3)$, and consider $\bbC^{m(2m+1)}$ with holomorphic coordinates $(w_I,w,z,\zeta)=(w^{ij}_I,w^{k_1},z^{k_2},\zeta^{k_3})$. Here $I=1,2,3,4$; $i<j,k_1,k_2,k_3=1,2,\dots,m$.

Define
$$\boxed{\begin{aligned}
  M^{N(m)}\,=\,\, \Big\{
                        ~\bbC^{m(2m+1)}\ni &(w_I,w,z,\zeta)\,\,\,\,\mathrm{s.t.}\\
  & \left\{   
\begin{aligned} & Im(w^{ij}_1)=Im(z^i\bar{z}^j+\zeta^i\bar{\zeta}^j)\\
                         &  Im(w^{ij}_2)=Im(z^i\bar{\zeta}^j+\zeta^i\bar{z}^j)\\
                         &  Im(w^{ij}_3)=Re(z^i\bar{z}^j+\zeta^i\bar{\zeta}^j)\\
                         &  Im(w^{ij}_4)=Re(z^i\bar{\zeta}^j+\zeta^i\bar{z}^j)
               \end{aligned}  
  \right\} \,\,\,\forall i<j=1,2,\dots m,\\
  & \,\,\,\,\,\,Im(w^i)=|z^i|^2+|\zeta^i|^2\quad\quad\quad\,\,\,\forall i=1,2,\dots m~\Big\}
\end{aligned}}.$$

We have the following theorem.
\begin{theorem}\label{so*}
  Let $m\geq 2$. The subset $M^{N(m)}$ of $\bbC^{m(2m+1)}$ is a real $N(m)$-dimensional embedded CR manifold, acquiring the CR structure of CR dimension $n=2m$ and CR codimension $k=m(2m-1)$ from the ambient complex space $\bbC^{m(2m+1)}$. Its local group of CR automorphisms is isomorphic to the real simple Lie group $SO^*(4m+2)$ with the Lie algebra having the Satake diagram \tikzset{/Dynkin diagram/fold style/.style={stealth-stealth,thin,
shorten <=1mm,shorten >=1mm}}\begin{dynkinDiagram}[edge length=.5cm]{D}{*o*.o*oo}
\dynkinFold{6}{7}
  \end{dynkinDiagram} with $\ell=2m+1$ nodes. It is locally CR equivalent to the flat model $SO^*(4m+2)/P_2$ of an $N(m)$-dimensional parabolic geometry of type $(SO^*(4m+2),P_2)$, where the real parabolic subgroup $P_2$ in $SO^*(4m+2)$ is determined by the following crossing on the corresponding $\mathfrak{d}_\ell$ Satake diagram:
\tikzset{/Dynkin diagram/fold style/.style={stealth-stealth,thin,
shorten <=1mm,shorten >=1mm}}\begin{dynkinDiagram}[edge length=.5cm]{D}{*o*.o*tt}
\dynkinFold{6}{7}
 \end{dynkinDiagram}.
\end{theorem}

\begin{example}\label{exso*}
  If $m=2$, $l=3$, we have a 14-dimensional CR manifold of \emph{CR dimension} $n=4$ and \emph{CR codimension} $k=6$, CR embedded in $\bbC^{10}$ with holomorphic coordinates $(w^{12}_1,w^{12}_2,w^{12}_3,w^{12}_4,w^1,w^2,z^1,z^2,\zeta^1,\zeta^2)$, via the $k=6$ real equations
  $$\begin{aligned}
    &Im(w^{12}_1-z^1\bar{z}^2-\zeta^1\bar{\zeta}^2)= Im(w^{12}_2-z^1\bar{\zeta}^2-\zeta^1\bar{z}^2)=0\\
    &Im(w^{12}_3)-Re(z^1\bar{z}^2+\zeta^1\bar{\zeta}^2)=Im(w^{12}_4)-Re(z^1\bar{\zeta}^2+\zeta^1\bar{z}^2)=0\\
    &Im(w^1)-|z^1|^2-|\zeta^1|^2=Im(w^2)-|z^2|^2-|\zeta^2|^2=0.
  \end{aligned}$$

 This CR structure has $SO^*(10)$ as the full group of its local CR automorphisms.

    Again, the formulae can be also made valid for $m=1$. In this case we have a 5-dimensional CR manifold $M^5$ embedded in $\bbC^3$ with coordinates $(w,z,\zeta)$. The real manifold is obtained as a \emph{hypersurface} $Im(w)=z\bar{z}+\zeta\bar{\zeta}$ in $\bbC^3$. We thus have again an embedded hypersurface CR structure, the \emph{Heisenberg group}, of CR dimension $n=2$ and CR codimension $k=1$, but this one with the signature of its Levi form $(+,+)$. This is again a classical flat parabolic CR structure from the series $SU(p,q)$, for $p=3$, $q=1$. It has $SO^*(6)$ as the group of its local automorphisms. It is excluded from our Theorem in order not to double the classification of the parabolic CR structures, due to the low dimensional isomorphism $\soa^*(6)=\sua(3,1)$.
\end{example}
\subsection{Flat parabolic CR structures with symmetry algebra $\sua(p,q)$}
We end our survey of examples, with a branch of accidental homogeneous CR manifolds which generalize, to the accidental setting, the codimension one \emph{hyperquadrics} embedded in $\bbC^{n+1}$. Our accidental ones have symmetries of a $\sug(p,q)$ group, similar to the hypersurface ones, but they have \emph{higher codimension}. We left this discussion to the end of our survey, because the presentation of all of these higher codimension accidental CRs with $\sug(p,q)$ symmetry is quite complicated, due to many possible choices of parabolic subgroups $P$ in $\sug(p,q)$, which lead to an \emph{accidental} CR structure $J$ on $M=\sug(p,q)/P$.

Since we are always in the homogeneous situation, it is enough to indicate which choices of parabolic subalgebras $\mathfrak{p}$ in $\sua(p,q)$ lead to the gradation $\sua(p,q)=\mathfrak{g}_{-2}\oplus\mathfrak{g}_{-1} \oplus \mathfrak{g}_0\oplus \mathfrak{g}_1\oplus\mathfrak{g}_2$ having natural $J:\mathfrak{g}_{-1}\to\mathfrak{g}_{-1}$ such that $J^2=-id$. These can be described as follows:

The Lie algebra $\sua(p,q)$ has $\ell=p+q-1$ simple positive roots $\alpha_1,\alpha_2$,...,$\alpha_{\ell}$. For reasons which will be clear soon in the examples, for our discussion it is enough to only consider $p\leq q$, when $p\geq2$ and $q\geq 3$.


The following two cases should be considered. 
\begin{itemize}
  \item Either $2\leq p<q$, and then the \emph{Satake diagram} for the real Lie algebra $\sua(p,q)$ is
$$\begin{dynkinDiagram}[edge length=.4cm]{A}{ooo.o**.**o.ooo}
\invol{1}{12}\invol{2}{11}\invol{3}{10}\invol{4}{9}
  \end{dynkinDiagram};$$
    here the black nodes start after the first $p$ white nodes, and became again white starting at $(\ell-p+1)^\mathrm{st}=q^\mathrm{th}$ node,
    \item or $p=q$, and 
the \emph{Satake diagram} for the real Lie algebra $\sua(p,p)$ is
$$\begin{dynkinDiagram}[edge length=.4cm
  ]{A}{ooo.ooo.ooo}
\invol{1}{9}\invol{2}{8}\invol{3}{7}\invol{4}{6}
\end{dynkinDiagram}.$$\end{itemize}
In both pictures, the $\ell=p+q-1$ roots $\alpha_1,\alpha_2$, ..., $\alpha_\ell$ are symbolized by the white or black nodes in the diagram from left to right.

It follows that each parabolic subgroup $P_s$, with Lie algebra $\mathfrak{p}_s$, which leads to an accidental CR structure on the corresponding homogeneous manifold $M=\sug(p,q)/P_s$ is in one to one correspondence with the \emph{choice of a pair of roots} $c_s=\{\alpha_s,\alpha_{p+q-s}\}$. The choices $c_s$ of these pairs of roots lead to \emph{nonequivalent CR manifolds} $M=\sug(p,q)/P_s$
$$ \mathrm{for}\quad\mathrm{each}\quad 1\leq p< q\quad\mathrm{ and}\quad s=1,2,\dots, p,$$
and
$$ \mathrm{for}\quad\mathrm{each}\quad q=p\geq 2\quad\mathrm{ and}\quad s=1,2,\dots, p-1.$$
It further follows that \emph{for each choice of the integers} $p,q,s$ \emph{as above}, the \emph{CR manifold} $M=\sug(p,q)/P_s$ has
\begin{itemize}
  \item \emph{real dimension}
    $\dim M=s(2p+2q-3s)$, \item the \emph{CR dimension} $n=s(p+q-2s)$ and \item the \emph{CR codimension} $k=s^2$.\end{itemize}
In particular, the familiar codimension one \emph{hyperquadrics} embedded in $\bbC^{n+1}$, with the $\sug(p,q)$ group of CR automorphisms, corresponds to $s=1$ and the choice of the roots $c_1=\{\alpha_1,\alpha_{p+q-1}\}$.
As it is seen from here, if $s\geq 2$ these CR structures are \emph{not} of hypersurface type. In all cases $s\geq 2$ these CR structures are \emph{accidental}.

To state an appropriate theorem about the accidental CRs with $\sug(p,q)$ symmetry we thus need to choose three integers $p,q,s$ with their ranges as discussed above. It turns out, that to set up the formulas for the embedding of these CR manifolds in $\bbC^{n+k}$ it is convenient to pass from the integers $p,q,s$ to the new integers
$$t=p-s\quad \mathrm{and}\quad r=q-p.$$
If $p,q,s$ changes as we discussed above, we have the following ranges of $t,r$ and $s$:
$$r\geq 0, \quad t\geq 0, \quad s\geq 1\quad\mathrm{and}\quad (r,t)\neq (0,0).$$
Let $$N(t,r,s)=s(2r+4t+s)$$ and consider $\bbC^{s(r+2t+s)}$ with holomorphic coordinates $$(z,u,v,w)=(z_{a\mu},u_{bA},v_{cB},w_{de})$$ and with $1\leq a,b,c,d,e\leq s,\,\,$ with $1\leq \mu\leq r\,\,$ and $1\leq A,B\leq t$, provided that $r\neq 0$ and $t\neq 0$. If $r=0$ the $z$ variables are not present in $(z,u,v,w)$, and if $t=0$ the $u$ and $v$ variables are not present in $(z,u,v,w)$.  

Define 

$$\boxed{
  \begin{aligned}
    M^{N(t,r,s)}\,=\,\,
    &\left\{
     \begin{aligned}
      \bbC^{s(r+2t+s)}&\ni (z,u,v,w)\,\,\,\,\mathrm{s.t.}\\
   &\mathrm{Im}\,w_{ac}
\,=\,
\mathrm{Im}\,
\left\{
\aligned
z_{a1}\,\overline{z}_{c1}
&
+\cdots+
z_{ar}\,\overline{z}_{cr}
\\
+
u_{a1}\,\overline{v}_{c1}
&
+\cdots+
u_{at}\,\overline{v}_{ct}
\\
+
v_{a1}\,\overline{u}_{c1}
&
+\cdots+
v_{at}\,\overline{u}_{ct}
\endaligned
\right\}\quad\mathrm{for}\quad 1\leq a<c\leq s,\\
&\mathrm{Re}\,w_{ac}
\,=\,
\mathrm{Re}\,
\left\{
\aligned
z_{a1}\,\overline{z}_{c1}
&
+\cdots+
z_{ar}\,\overline{z}_{cr}
\\
+
u_{a1}\,\overline{v}_{c1}
&
+\cdots+
u_{at}\,\overline{v}_{ct}
\\
+
v_{a1}\,\overline{u}_{c1}
&
+\cdots+
v_{at}\,\overline{u}_{ct}
\endaligned
\right\} \quad\mathrm{for}\quad 1\leq c\leq a\leq s.
    \end{aligned}
    \right\}
\end{aligned}.}$$

We have the following theorem.
\begin{theorem}\label{su}
  Let $r\geq 0$, $t\geq 0$, $s\geq 1$ and $(r,t)\neq (0,0)$. The set $M^{N(t,r,s)}\subset \bbC^{s(r+2t+s)}$ is a real $N(t,r,s)$-dimensional embedded CR manifold, acquiring the CR structure of CR dimension $n=s(r+2t)$ and CR codimension $k=s^2$ from the ambient complex space $\bbC^{s(r+2t+s)}$. Its local group of CR automorphisms is isomorphic to the real simple Lie group $\sug(t+s,r+t+s)$. It is locally CR equivalent to the flat model $\sug(t+s,r+t+s)/P_s$ of an $N(t,r,s)$-dimensional parabolic geometry of type $(\sug(t+s,r+t+s),P_s)$, where the real parabolic subgroup $P_s$ in $\sug(t+s,r+t+s)$ is determined by crossing the roots $\alpha_s$ and $\alpha_{r+2t+s}$ on the Satake diagram of $\sug(t+s,r+t+s)$. 
\end{theorem}
\begin{example}\label{su23}
  To get an \emph{accidental} CR manifold with $\sug(p,q)$ symmetry we need to have $s\geq2$, since otherwise the CR manifold has codimension one, and is of a hypersurface type. Knowing that we need $s$ at least as large as 2, and using inequalities for $t$ and $r$, it is easy to see that the smallest possible $\sug(p,q)$ symmetry group of an accidental CR structure is $\sug(2,3)$. In such case $p=2$, $q=3$, and the corresponding choice of a parabolic $P_s$ in $\sug(2,3)$ leading to the accidental CR structure on $\sug(2,3)/P_s$ is, at the level of Lie algebra, given by the following crossings  in the $\sua(2,3)$ Satake diagram:
\begin{center}
  \includegraphics[scale=0.19]{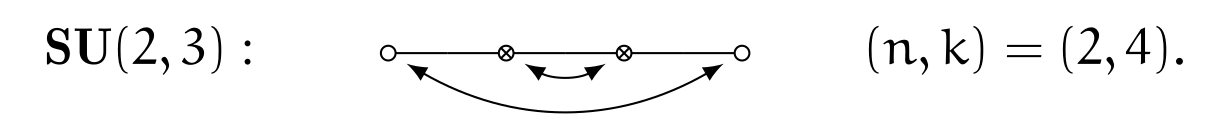}
\end{center}
  As indicated in the diagram, this choice of a parabolic $P_2$ in $\sug(2,3)$ corresponds to an accidental CR manifold $M^8=\sug(2,3)/P_s$. This $\sug(2,3)$ homogeneous CR manifold has dimension $N=8$, CR dimension $n=2$ and CR codimension $k=4$. It is embedded in $\bbC^6$ as follows:
  $$M^8=\left\{\begin{aligned}\bbC^6&\ni(z_{1},z_{2},w_{11},w_{12},w_{21},w_{22})\,\,s.t.\\
  &w_{12}-\bar{w}_{12}=z_{1}\bar{z}_{2}-\bar{z}_{1}z_{2}\\
  &w_{11}+\bar{w}_{11}=2z_{1}\bar{z}_{1}\\
  &w_{21}+\bar{w}_{21}=z_{1}\bar{z}_{2}+\bar{z}_{1}z_{2}\\
  &w_{22}+\bar{w}_{22}=2z_{2}\bar{z}_{2}\end{aligned}\right\},$$
  and has $\sug(2,3)$ group of CR automorphisms.
  \end{example}
\begin{example}\label{exsu33} 
  In case of the $\sug(3,3)$ CR symmetry we have $p=q=3$, i.e. $r=0$ and $t=3-s>0$, and if we want to have an accidental CR structure we need to take the only possibility $s=2$.

  The Satake Diagram for $\sua(3,3)$ is
  $$\begin{dynkinDiagram}[edge length=.4cm
  ]{A}{ooooo}
\invol{1}{5}\invol{2}{4}
  \end{dynkinDiagram},$$
  and the choice of a parabolic leading to the accidental CR structure is:
$$\begin{dynkinDiagram}[edge length=.4cm]{A}{ototo}\invol{1}{5}\invol{2}{4}
  \end{dynkinDiagram},\quad\mathrm{with}\quad (n,k)=(4,4).$$
  The embedding is given by:
  $$M^{12}=\left\{\begin{aligned}\bbC^{8}&\ni(u_{a},v_{b},w_{cd}), a,b,c,d=1,2, s.t.\\
  &w_{12}-\bar{w}_{12}=u_{1}\bar{v}_{2}+v_{1}\bar{u}_{2}-\bar{u}_{1}v_{2}-\bar{v}_{1}u_{2}\\
  &w_{11}+\bar{w}_{11}=2(u_{1}\bar{v}_{1}+v_{1}\bar{u}_{1})\\
  &w_{21}+\bar{w}_{21}=u_{2}\bar{v}_{1}+v_{2}\bar{u}_{1}+\bar{u}_{2}v_{1}+\bar{v}_{2}u_{1}\\
   &w_{22}+\bar{w}_{22}=2(u_{2}\bar{v}_{2}+v_{2}\bar{u}_{2})
  \end{aligned}\right\},$$
  and provides an \emph{accidental} CR structure  
  of CR dimension $n=4$ and CR codimension $k=4$ 
  in $\bbC^{8}$ with the CR automorphisms group $\sug(3,3)$.
\end{example}
\begin{example}\label{exsu24} 
  In case of the $\sug(2,4)$ CR symmetry we have $p=2$, $q=4$, i.e. $r=2$ and $t=2-s\geq 0$, and an accidental CR structure appears for $s=2$ only. In such case $t=0$.

  The Satake Diagram for $\sua(2,4)$ is
  $$\begin{dynkinDiagram}[edge length=.4cm
  ]{A}{oo*oo}
\invol{1}{5}\invol{2}{4}
  \end{dynkinDiagram},$$
  and the choice of a parabolic leading to the accidental CR structure is:
$$\begin{dynkinDiagram}[edge length=.4cm]{A}{ot*to}\invol{1}{5}\invol{2}{4}
  \end{dynkinDiagram},\quad\mathrm{with}\quad (n,k)=(4,4).$$
  The embedding is given by:
  $$M^{12}=\left\{\begin{aligned}\bbC^{8}&\ni(z_{ab},w_{cd}), a,b,c,d=1,2, s.t.\\
  &w_{12}-\bar{w}_{12}=z_{11}\bar{z}_{21}+z_{12}\bar{z}_{22}-\bar{z}_{11}z_{21}-\bar{z}_{12}z_{22}\\
  &w_{11}+\bar{w}_{11}=2(z_{11}\bar{z}_{11}+z_{12}\bar{z}_{12})\\
  &w_{21}+\bar{w}_{21}=z_{21}\bar{z}_{11}+z_{22}\bar{z}_{12}+\bar{z}_{21}z_{11}+\bar{z}_{22}z_{12}\\
   &w_{22}+\bar{w}_{22}=2(z_{21}\bar{z}_{21}+z_{22}\bar{z}_{22})
  \end{aligned}\right\}.$$
  This provides an \emph{accidental} CR structure  
  of CR dimension $n=4$ and CR codimension $k=4$ 
  in $\bbC^{8}$. It has $\sug(2,4)$ as its group of CR automorphisms.
\end{example}
\begin{example}\label{exsu34} 
  If $p=3$, $q=4$, and the group of CR automorphisms is $\sug(3,4)$ we have $r=4-3=1$ and $t=3-s\geq 0$, and we encounter the lowest dimensional situation when we have \emph{two} nonequivalent \emph{accidental} CR manifolds. This is because in this case there are two not equal to 1 possible values for $s$, namely $s=2$ and $s=3$. 
  The Satake Diagram for $\sua(3,4)$ with the choice of parabolic related to $s=2$ is
  \begin{center}
  \includegraphics[scale=0.23]{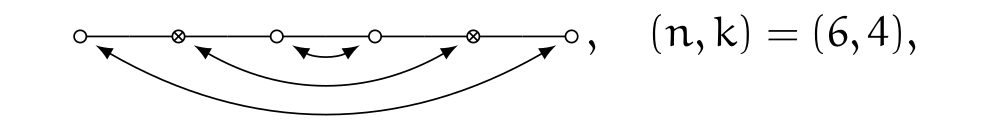}
\end{center}
and the Satake Diagram for $\sua(3,4)$ with the choice of parabolic related to $s=3$ is
\begin{center}
  \includegraphics[scale=0.23]{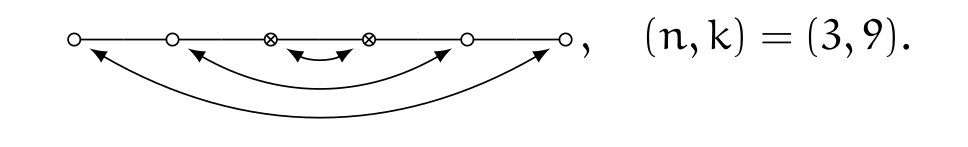}
\end{center}
\begin{itemize}
\item  In case of $s=2$, $r=1$, $t=1$, the embedding in $\bbC^{10}$ of the corresponding accidental CR structure of dimension  $N=16$,  CR dimension $n=6$ and CR codimension $k=4$ is given by:
  $$M^{16}=\left\{\begin{aligned}\bbC^{10}&\ni(z_{a},u_{b},v_{c},w_{de}), a,b,c,d,e=1,2, s.t.\\
  &\mathrm{Im}(w_{12})=\mathrm{Im}\big(z_{1}\bar{z}_{2}+u_{1}\bar{v}_{2}+v_{1}\bar{u}_{2}\big)\\
  &\mathrm{Re}(w_{11})=z_{1}\bar{z}_{1}+u_{1}\bar{v}_{1}+v_{1}\bar{u}_{1}\\
  &\mathrm{Re}(w_{21})=\mathrm{Re}\big(z_{2}\bar{z}_{1}+u_{2}\bar{v}_{1}+v_{2}\bar{u}_{1}\big)\\
   &\mathrm{Re}(w_{22})=z_{2}\bar{z}_{2}+u_{2}\bar{v}_{2}+v_{2}\bar{u}_{2}
  \end{aligned}\right\}.$$
  By construction this \emph{accidental} CR structure has $\sug(3,4)$ group of the CR automorphisms.
\item In case of $s=3$, $r=1$, $t=0$, the embedding in $\bbC^{12}$ of the corresponding accidental CR structure of dimension  $N=15$,  CR dimension $n=3$ and CR codimension $k=9$ is given by:
  $$M^{15}=\left\{\begin{aligned}\bbC^{12}&\ni(z_{a},w_{bc}), a,b,c=1,2,3, s.t.\\
  &w_{12}-\bar{w}_{12}=z_{1}\bar{z}_{2}-\bar{z}_{1}z_{2}\\
  &w_{13}-\bar{w}_{13}=z_{1}\bar{z}_{3}-\bar{z}_{1}z_{3}\\
   &w_{23}-\bar{w}_{23}=z_{2}\bar{z}_{3}-\bar{z}_{2}z_{3}\\
  &w_{11}+\bar{w}_{11}=2z_{1}\bar{z}_{1}\\
   &w_{21}+\bar{w}_{21}=z_{2}\bar{z}_{1}+\bar{z}_{2}z_{1}\\
  &w_{22}+\bar{w}_{22}=2z_{2}\bar{z}_{2}\\
  &w_{31}+\bar{w}_{31}=z_{3}\bar{z}_{1}+\bar{z}_{3}z_{1}\\
  &w_{32}+\bar{w}_{32}=z_{3}\bar{z}_{2}+\bar{z}_{3}z_{2}\\
   &w_{33}+\bar{w}_{33}=2z_{3}\bar{z}_{3}
  \end{aligned}\right\}.$$
  By construction this \emph{accidental} CR structure has $\sug(3,4)$ group of the CR automorphisms.
  \end{itemize}
\end{example}
\begin{example}\label{exsu44} 
  In case of $\sug(4,4)$ we have $p=q=4$, i.e. $r=0$ and $t=4-s>0$. This shows that $s$ can only assume \emph{three} values $s=1,2,3$. The Satake Diagram for $\sua(4,4)$ is
  $$\begin{dynkinDiagram}[edge length=.4cm
  ]{A}{ooooooo}
\invol{1}{7}\invol{2}{6}\invol{3}{5}
  \end{dynkinDiagram},$$
  and the three possible choices of parabolics leading to the 3 accidental CR structures are:
  \begin{itemize}
  \item[(i)]
    $\begin{dynkinDiagram}[edge length=.4cm]{A}{tooooot}
\invol{1}{7}\invol{2}{6}\invol{3}{5}
  \end{dynkinDiagram},$ with $(n,k)=(6,1)$,
  \item[(ii)]
    $\begin{dynkinDiagram}[edge length=.4cm]{A}{otoooto}
\invol{1}{7}\invol{2}{6}\invol{3}{5}
  \end{dynkinDiagram},$ with $(n,k)=(8,4)$,
  \item[(iii)]
    $\begin{dynkinDiagram}[edge length=.4cm]{A}{oototoo}
\invol{1}{7}\invol{2}{6}\invol{3}{5}
  \end{dynkinDiagram}$, with $(n,k)=(6,9)$.
    \end{itemize}
  In case (i) we have the familiar hypersurface type 13-dimensional hyperquadric CR manifold
  $$M^{13}=\left\{\begin{aligned}\bbC^{7}&\ni(u_{A},v_{B},w), A,B=1,2,3, s.t.\\
  &\mathrm{Re}\big(w\big)=2\mathrm{Re}\big(u_{1}\bar{v}_{1}+u_{2}\bar{v}_{2}+u_{3}\bar{v}_{3}\big)\end{aligned}\right\}$$
  in $\bbC^{7}$ with the CR automorphisms group $\sug(4,4)$.

  In case (ii) we have an \emph{accidental} 20-dimensional CR manifold
  $$M^{20}=\left\{
\begin{aligned}\bbC^{12}
&\,\ni\,
(u_{ab},v_{cd},w_{ef}), a,b,c,d,e,f=1,2, s.t.
\\
\mathrm{Im}\big(w_{12}\big)
&\,=\,\mathrm{Im}\big(u_{11}\bar{v}_{21}+u_{12}\bar{v}_{22}+v_{11}\bar{u}_{21}+v_{12}\bar{u}_{22}\big)
\\
\mathrm{Re}\big(w_{11}\big)
&\,=\,
2\mathrm{Re}\big(u_{11}\bar{v}_{11}+u_{12}\bar{v}_{12}\big)
\\
\mathrm{Re}\big(w_{21}\big)
&\,=\,
\mathrm{Re}\big(u_{21}\bar{v}_{11}+u_{22}\bar{v}_{12}+v_{21}\bar{u}_{11}+v_{22}\bar{u}_{12}\big)
\\
\mathrm{Re}\big(w_{22}\big)
&\,=\,
2\mathrm{Re}\big(u_{21}\bar{v}_{21}+u_{22}\bar{v}_{22}\big)
\end{aligned}
\right\}$$
  of CR dimension $n=8$ and CR codimension $k=4$ 
  in $\bbC^{12}$ with the CR automorphisms group $\sug(4,4)$.

   In case (iii) we have an \emph{accidental} 21-dimensional CR manifold
  $$M^{21}=\left\{
\begin{aligned}\bbC^{15}&\ni(u_{a},v_{b},w_{cd}), a,b,c,d=1,2,3, s.t.
\\
&w_{12}-\bar{w}_{12}
=
u_1\bar{v}_2+v_1\bar{u}_2-\bar{u}_1v_2-\bar{v}_1u_2
\\
&w_{13}-\bar{w}_{13}
=
u_1\bar{v}_3+v_1\bar{u}_3-\bar{u}_1v_3-\bar{v}_1u_3
\\
&w_{23}-\bar{w}_{23}
=
u_2\bar{v}_3+v_2\bar{u}_3-\bar{u}_2v_3-\bar{v}_2u_3
\\
&w_{11}+\bar{w}_{11}
=
2\big(u_1\bar{v}_1+v_1\bar{u}_1\big)
\\
&w_{21}+\bar{w}_{21}
=
u_2\bar{v}_1+v_2\bar{u}_1+\bar{u}_2v_1+\bar{v}_2u_1
\\
&w_{22}+\bar{w}_{22}
=
2\big(u_2\bar{v}_2+v_2\bar{u}_2\big)
\\
&w_{31}+\bar{w}_{31}
=
u_3\bar{v}_1+v_3\bar{u}_1+\bar{u}_3v_1+\bar{v}_3u_1
\\
&w_{32}+\bar{w}_{32}
=
u_3\bar{v}_2+v_3\bar{u}_2+\bar{u}_3v_2+\bar{v}_3u_2
\\
&w_{33}+\bar{w}_{33}
=
2\big(u_3\bar{v}_3+v_3\bar{u}_3\big)
\\
\end{aligned}
\right\}$$
  of CR dimension $n=6$ and CR codimension $k=9$ 
  in $\bbC^{15}$, again with the CR automorphisms group $\sug(4,4)$.
\end{example}
\begin{example}\label{exsu35} 
  In the case of $\sug(3,5)$ we have $p=3, q=5$, i.e. $r=2$ and $t=3-s\geq 0$. This shows that again $s$ can only assume \emph{three} values $s=1,2,3$. The Satake Diagram for $\sua(3,5)$ is
  $$\begin{dynkinDiagram}[edge length=.4cm
  ]{A}{ooo*ooo}
\invol{1}{7}\invol{2}{6}\invol{3}{5}
  \end{dynkinDiagram},$$
  and the three possible choices of parabolics leading to the 3 accidental CR structures are:
  \begin{itemize}
  \item[(i)]
    $\begin{dynkinDiagram}[edge length=.4cm]{A}{too*oot}
\invol{1}{7}\invol{2}{6}\invol{3}{5}
  \end{dynkinDiagram},$ with $(n,k)=(6,1)$,
  \item[(ii)]
    $\begin{dynkinDiagram}[edge length=.4cm]{A}{oto*oto}
\invol{1}{7}\invol{2}{6}\invol{3}{5}
  \end{dynkinDiagram},$ with $(n,k)=(8,4)$,
  \item[(iii)]
    $\begin{dynkinDiagram}[edge length=.4cm]{A}{oot*too}
\invol{1}{7}\invol{2}{6}\invol{3}{5}
  \end{dynkinDiagram}$, with $(n,k)=(6,9)$.
  \end{itemize}
 In case (i) we again have a hypersurface type 13-dimensional hyperquadric CR manifold
  $$M^{13}=\left\{\begin{aligned}\bbC^{7}&\ni(z_{1a},u_{1b},v_{1c},w_{11}), a,b,c=1,2, s.t.\\
  &w_{11}+\bar{w}_{11}=2\big(z_{11}\bar{z}_{11}+z_{12}\bar{z}_{12}+u_{11}\bar{v}_{11}+u_{12}\bar{v}_{12}+v_{11}\bar{u}_{11}+v_{12}\bar{u}_{12}\big)\end{aligned}\right\}$$
  in $\bbC^{7}$ with the CR automorphisms group $\sug(3,5)$.

  In case (ii) we have an \emph{accidental} 20-dimensional CR manifold
  $$M^{20}=\left\{\begin{aligned}\bbC^{12}
&\,\ni\,
(z_{ab},u_{c1}v_{d1},w_{ef}), a,b,c,d,e,f=1,2, s.t.
\\
\mathrm{Im}\big(w_{12}\big)
&\,=\,
\mathrm{Im}\big(z_{11}\bar{z}_{21}+z_{12}\bar{z}_{22}+u_{11}\bar{v}_{21}+v_{11}\bar{u}_{21}\big)
\\
w_{11}+\bar{w}_{11}
&\,=\,
2\big(z_{11}\bar{z}_{11}+z_{12}\bar{z}_{12}+u_{11}\bar{v}_{11}+v_{11}\bar{u}_{11}\big)
\\
\mathrm{Re}\big(w_{21}\big)
&\,=\,
\mathrm{Re}\big(z_{21}\bar{z}_{11}+z_{22}\bar{z}_{12}+u_{21}\bar{v}_{11}+v_{21}\bar{u}_{11}\big)
\\
w_{22}+\bar{w}_{22}
&\,=\,
2\big(z_{21}\bar{z}_{21}+z_{22}\bar{z}_{22}+u_{21}\bar{v}_{21}+v_{21}\bar{u}_{21}\big)
\end{aligned}\right\}$$
  of CR dimension $n=8$ and CR codimension $k=4$ 
  in $\bbC^{12}$ with the CR automorphisms group $\sug(3,5)$.

   In case (iii) we have an \emph{accidental} 21-dimensional CR manifold
  $$M^{21}=\left\{\begin{aligned}\bbC^{15}&\ni(z_{a\mu},w_{bc}), a,b,c=1,2,3,\,\mu=1,2\,\, s.t.\\
   &w_{12}-\bar{w}_{12}=z_{11}\bar{z}_{21}+z_{12}\bar{z}_{22}-\bar{z}_{11}z_{21}-\bar{z}_{12}z_{22}\\
   &w_{13}-\bar{w}_{13}=z_{11}\bar{z}_{31}+z_{12}\bar{z}_{32}-\bar{z}_{11}z_{31}-\bar{z}_{12}z_{32}\\
   &w_{23}-\bar{w}_{23}=z_{21}\bar{z}_{31}+z_{22}\bar{z}_{32}-\bar{z}_{21}z_{31}-\bar{z}_{22}z_{32}\\
   &w_{11}+\bar{w}_{11}=2(z_{11}\bar{z}_{11}+z_{12}\bar{z}_{12})\\
   &w_{21}+\bar{w}_{21}=z_{21}\bar{z}_{11}+z_{22}\bar{z}_{12}+\bar{z}_{21}z_{11}+\bar{z}_{22}z_{12}\\
   &w_{22}+\bar{w}_{22}=2(z_{21}\bar{z}_{21}+z_{22}\bar{z}_{22})\\
   &w_{31}+\bar{w}_{31}=z_{31}\bar{z}_{11}+z_{32}\bar{z}_{12}+\bar{z}_{31}z_{11}+\bar{z}_{32}z_{12}\\
   &w_{32}+\bar{w}_{32}=z_{31}\bar{z}_{21}+z_{32}\bar{z}_{22}+\bar{z}_{31}z_{21}+\bar{z}_{32}z_{22}\\
    &w_{33}+\bar{w}_{33}=2(z_{31}\bar{z}_{31}+z_{32}\bar{z}_{32})
\end{aligned}\right\}$$
  of CR dimension $n=6$ and CR codimension $k=9$ 
  in $\bbC^{15}$, again with the CR automorphisms group $\sug(3,5)$.
\end{example}
\section{Why examples of Sections 2 and 3 are flat parabolic geometries?}

\subsection{Proofs of Theorems \ref{e2} and \ref{e3}}
In the rest of this section we give justifications for Theorems \ref{e2}, \ref{e3}, \ref{so}, \ref{so*} and \ref{su}. Since the idea is the same for all of them, we will only concentrate on the proofs of Theorems \ref{e2} and \ref{e3} related to $\mathfrak{e}_6$. This is, anyhow, not so important, since the proofs of Theorems \ref{e3}, \ref{e2}, \ref{so}, \ref{so*} and \ref{su} follow directly from our Section 5 and (in full generality) from the Reference \cite{medori}.

The basic observation (a nontrivial one!), valid for all Theorems \ref{e3}, \ref{e2}, \ref{so}, \ref{so*} and \ref{su}, is that the entire CR geometry of CR manifolds $M^N$ appearing in them, is \emph{totally determined} by the mere geometry of the real distribution $H$. By this we mean that for \emph{all of the CR manifolds} from these Theorems, the complex structure $J$ in $H$ is an object totally determined by the pair $(M^N,H)$. Saying it yet differently: for all the CR manifolds from Theorems \ref{e3}, \ref{e2}, \ref{so} and \ref{so*} the local differential geometry of the CR structure $(M^N,H,J)$ is \emph{the same} as the local geometry of the structure $(M^N,H)$ of a real manifold of dimension $N$ and a rank $2n$ (real) distribution $H$ (with a proper \emph{symbol algebra}). This is the reason why we call the CR structures from Theorems \ref{e3}, \ref{e2}, \ref{so} and \ref{so*} as \emph{accidental} CR structures: these are structures of a vector distribution $H$ on $M^N$, and the $J$ in $H$ is given as a gift, or an \emph{accident}, from the geometric data $(M^N,H)$.

Of course the statement that the local differential geometry of the CR structures $(M^N,H,J)$ is \emph{the same} as the local geometry of the structures $(M^N,H)$ of a real manifold of dimension $N$ and a rank $2n$ (real) distribution $H$, for all the structures in Theorems \ref{e3}, \ref{e2}, \ref{so} and \ref{so*}, requires proofs. We will give them by inspecting all of these cases separately below. \\

\begin{roof}{{\bf Theorem}}{\ref{e2}}
  In this case, passing to the real variables $(u^i,v^i,x^i,y^i)$, we can write our CR structure $M^{24}_{E_{II}}$ in the form \eqref{rigid} with 8 defining functions $\Phi^i$ given by:
$$\begin{aligned}
    \Phi^1=\,\,&x^2x^3+x^1x^4+y^2y^3+y^1y^4\\
    \Phi^2=\,\,&x^2x^5+x^1x^6+y^2y^5+y^1y^6\\
    \Phi^3=\,\,&x^7y^1-x^5y^3+x^3y^5-x^1y^7\\
    \Phi^4=\,\,&x^8y^1+x^7y^2+x^6y^3+x^5y^4-x^4y^5-x^3y^6-x^2y^7-x^1y^8\\
    \Phi^5=\,\,&x^2x^7+x^3x^6-x^1x^8-x^4x^5+y^2y^7+y^3y^6-y^1y^8-y^4y^5\\
    \Phi^6=\,\,&x^8y^2-x^6y^4+x^4y^6-x^2y^8\\
    \Phi^7=\,\,&x^4x^7+x^3x^8+y^4y^7+y^3y^8\\
    \Phi^8=\,\,&x^6x^7+x^5x^8+y^6y^7+y^5y^8.
  \end{aligned}$$
  So now our CR manifold $M^{24}_{E_{II}}$ is parametrized by the real coordinates $(u^i,x^i,y^i)$, $i=1,2,\dots,8$, so that the forms $\lambda^i$ anihilating the rank 16 distribution $H$ on $M^{24}_{E_{II}}$, according to the formula \eqref{lambda}, are given by the 1-forms $\lambda^i$, $i=1,2,\dots,8$, of the Pfaffian system from Corollary \ref{co13}.
 Now we have the following Lemma\footnote{Actually, to shorten the expressions for the $\lambda$'s given here, and to have more symmetric formulas for the EDS later, we changed the original coordinates $u^i$ appearing in \eqref{lambda} to more suitable $u^i$'s appearing here, and slightly rescaled the $\lambda$'s in \eqref{lambda}.}.
  \begin{lemma}\label{e3l}
    The Lie algebra of infinitesimal symmetries of the Pfaffian system $[\lambda^1,\lambda^2,\dots,\lambda^8]$ on
    $M^{24}_{E_{II}}$, with forms $\lambda^i$ as in Corollary \ref{co13}, or what is the same, the Lie algebra of infinitesimal symmetries of the rank 16 distribution $H=[\lambda^1,\lambda^2,\dots,\lambda^8]^\perp$, is the real form of the simple exceptional Lie algebra $\mathfrak{e}_6$ with Satake diagram
\begin{dynkinDiagram}[edge length=.4cm]{E}{oooooo}
\invol{1}{6}\invol{3}{5}
 \end{dynkinDiagram}.
  \end{lemma}

  \begin{roof}{{\bf Lemma}}{\ref{e3l}} There are at least three ways of proving the Lemma:
    \begin{itemize}
    \item By brute force: solve the PDEs \eqref{syd} for the symmetries $Y$ of the distribution structure $(M,H)$ defined as the annihilator of the Pfaffian system  $[\lambda^1,\lambda^2,\dots,\lambda^8]$ given in the Lemma. This can be done e.g. by I. Anderson's \emph{Differential geometry package} of Maple. This package can also classify the obtained Lie algebra of symmetries, showing that the algebra is of type $\mathfrak{e}_6$. But for this one needs quite a powerful computer and the commercial software Maple.
    \item By the \emph{Cartan equivalence method} applied to forms $(\lambda^i)$ given on a manifold $M^{24}_{E_{II}}$ up to the transformations $\lambda^i\mapsto a^i{}_j\lambda^j$, where the real $8\times 8$ matrices $(a^i{}_j)$ belong to $GL(8,\bbR)$. This is very tedious, and as a result gives $78$ linearly independent 1-forms on a certain 78-dimensional manifold $\mathcal G$ which satisfies the EDS of the Maurer-Cartan forms on the appropriate real form of the simple exceptional Lie group $E_6$.
      \item By the \emph{Tanaka prolongation method}, which we will follow in this exposition.  
    \end{itemize}

    \noindent
The proof of the Lemma is based on the Tanaka's Theorem \ref{tansym}. To see this we extend the eight 1-forms $\lambda^i$ generating the Pfaffian system from the Lemma to a coframe $(\lambda^A)$, $A=1,2\dots 24$, on $M^{24}_{E_{III}}$ by setting
$$\lambda^{i+8}=\der x^i,\quad \quad \lambda^{i+16}=\der y^i,\quad i=1,2,\dots, 8.$$
We calculate all the exterior derivatives $\der\lambda^A$ obtaining:
 \begin{equation}\begin{aligned}
    \der\lambda^1=\,\,&\lambda^9\dz \lambda^{20}+\lambda^{10}\dz \lambda^{19}+\lambda^{11}\dz \lambda^{18}+\lambda^{12}\dz \lambda^{17}\\
    \der\lambda^2=\,\,&\lambda^9\dz \lambda^{22}+\lambda^{10}\dz \lambda^{21}+\lambda^{13}\dz \lambda^{18}+\lambda^{14}\dz \lambda^{17}\\
    \der\lambda^3=\,\,&\lambda^9\dz \lambda^{15}-\lambda^{11}\dz \lambda^{13}+\lambda^{17}\dz \lambda^{23}-\lambda^{19}\dz \lambda^{21}\\
    \der\lambda^4=\,\,&\lambda^{9}\dz \lambda^{16}+\lambda^{10}\dz \lambda^{15}+\lambda^{11}\dz \lambda^{14}+\lambda^{12}\dz \lambda^{13}+\lambda^{17}\dz \lambda^{24}+\lambda^{18}\dz \lambda^{23}+\\&\lambda^{19}\dz \lambda^{22}+\lambda^{20}\dz \lambda^{21}\\
    \der\lambda^5=\,\,&-\lambda^9\dz \lambda^{24}+\lambda^{10}\dz \lambda^{23}+\lambda^{11}\dz \lambda^{22}-\lambda^{12}\dz \lambda^{21}-\lambda^{13}\dz \lambda^{20}+\lambda^{14}\dz \lambda^{19}+\\&\lambda^{15}\dz \lambda^{18}-\lambda^{16}\dz \lambda^{17}\\
    \der\lambda^6=\,\,&\lambda^{10}\dz \lambda^{16}-\lambda^{12}\dz \lambda^{14}+\lambda^{18}\dz \lambda^{24}-\lambda^{20}\dz \lambda^{22}\\
    \der\lambda^7=\,\,&\lambda^{11}\dz \lambda^{24}+\lambda^{12}\dz \lambda^{23}+\lambda^{15}\dz \lambda^{20}+\lambda^{16}\dz \lambda^{19}\\
    \der\lambda^8=\,\,&\lambda^{13}\dz \lambda^{24}+\lambda^{14}\dz \lambda^{23}+\lambda^{15}\dz \lambda^{22}+\lambda^{16}\dz \lambda^{21}\\
    \der\lambda^\mu=\,\,&0,\quad\forall \mu=9,10,\dots,24.
  \end{aligned}
  \label{lg}\end{equation}
We thus have \begin{equation}\der\lambda^A=-\tfrac12c^A{}_{BD}\lambda^B\dz\lambda^D,\label{mc}\end{equation}
with \emph{all} the coeffcients $c^A{}_{BD}=-c^A{}_{DB}$ being \emph{constants}. Thus our 24-dimensional manifold $M^{24}_{E_{II}}$ with the coframe $(\lambda^A)$ can be locally consider to be a Lie group, say $\mathcal N$, for which the coframe $(\lambda^A)$ is a coframe of Maurer-Cartan forms. Looking at the structure constants of the Lie algebra $\mathfrak{n}$ of this group, which can be read off \eqref{lg} via \eqref{mc}, we see that this Lie group is \emph{nilpotent}. Indeed, taking the vector fields $X_A$, $A=1,2,\dots,24$ on $\mathcal N$ dual to the coframe 1-forms $\lambda^B$, $X_A\hook\lambda^B=\delta^B_A$, we see that they form a 2-step nilpotent Lie algebra
\begin{equation}\mathfrak{n}=\mathfrak{n}_{\minu 2}\oplus\mathfrak{n}_{\minu 1}\label{n1}\end{equation}
with
\begin{equation}
  \mathfrak{n}_{\minu 1}=\Span_\bbR(X_9,X_{10},\dots,X_{24}),\quad\quad \mathfrak{n}_{\minu 2}=\Span_\bbR(X_1,X_2,\dots,X_8).\label{n2}\end{equation}
We used here the commutation relations of vector fields $X_A$ obtained from \eqref{lg}-\eqref{mc} via the formula $[X_A,X_B]=c^E{}_{AB}X_E$. Now note that our distribution $H$ from the Lemma \ref{e3l} is precisely the distrubution spanned, over the functions on $\mathcal N$, by the left invariant vector fields $X_\mu$, $\mu=9,10,\dots,24$, which span the $\mathfrak{n}_{\minu 1}$ part of the 2-step nilpotent Lie algebra $\mathfrak{n}$.

So, what is the Lie algebra of symmetries $\mathfrak{aut}(H)$ of our distribution $H$ from the Lemma?

Via the Tanaka theory this is a Lie algebra \emph{isomorphic} to the Tanaka prolongation $\mathfrak{g}_T(\mathfrak{n})$ of the nilpotent Lie algebra $\mathfrak{n}=\Span_\bbR(X_A)$. So to determine the Lie algebra of automorphisms of our rank 16 distribution $H$ it is enough to calculate $\mathfrak{g}_T(\mathfrak{n})$ for $\mathfrak{n}$ in \eqref{n1}-\eqref{n2}.  

Calculating the Tanaka prolongation $\mathfrak{g}_T(\mathfrak{n}_{\minu})$ is an algorithmic inductive process involving only linear algebra applied to $\mathfrak{n}_\minu$ and its successors $\mathfrak{n}_j$, $j\geq 0$. One first calculates $\mathfrak{n}_0$, then $\mathfrak{n}_1$, etc. Here, for brevity we will only show in details how  $\mathfrak{n}_0$ for our $\mathfrak{n}$, as in \eqref{n1}-\eqref{n2}, is calculated

The elements of $\mathfrak{n}_0$ are derivations $A$ of $\mathfrak{n}$ preserving the strata $\mathfrak{n}_{\minu 1}$ and $\mathfrak{n}_{\minu 2}$ in $\mathfrak{n}$. Thus the matrix elements $(A_B{}^C)$ of a linear map $A:\mathfrak{n}\to \mathfrak{n}$ belonging to $\mathfrak{n}_0$ must satisfy, in the basis $(X_A)$, the following equations:
\begin{enumerate}
\item $A_i{}^\mu=0$ and $A_\mu{}^i=0$, for $i=1,2,\dots 8$, $\mu=9,10,\dots, 24$, (preservation of the strata)
  \item $c^E{}_{BD}A_E{}^F-c^F{}_{BE}A_D{}^E+c^F{}_{DE}A_B{}^E=0$, $B,D,F=1,2,\dots, 24$, (derivation property).
  \end{enumerate}
These \emph{linear} equations for the matrix entries $A_B{}^C$, with $c^B{}_{DE}$ given by \eqref{n1}-\eqref{n2}, have a 30-dimensional space of solutions. Explicitly
\begin{equation}
  (A_B{}^C)=\bma A_i{}^j&|&0\\--&&--\\0&|&A_\mu{}^\nu\ema,\label{abc-bis}\end{equation}
where
$$\scriptscriptstyle{(A_i{}^j)=}\bma \scriptscriptstyle{a_{27}+a_{28}+2 a_{30}}&\scriptscriptstyle{a_{7}}&\scriptscriptstyle{a_{8}}&\scriptscriptstyle{2a_{9}}&\scriptscriptstyle{2a_{10}}&\scriptscriptstyle{a_{11}}&\scriptscriptstyle{a_{12}}&\scriptscriptstyle{0}\\
\scriptscriptstyle{a_{1}}&\scriptscriptstyle{-a_{27}+a_{29}+a_{30}}&\scriptscriptstyle{a_{17}}&\scriptscriptstyle{2a_{18}}&\scriptscriptstyle{2a_{19}}&\scriptscriptstyle{a_{20}}&\scriptscriptstyle{0}&\scriptscriptstyle{a_{12}}\\
\scriptscriptstyle{a_{2}}&\scriptscriptstyle{a_{13}}&\scriptscriptstyle{-a_{28}+a_{29}+a_{30}}&\scriptscriptstyle{2a_{23}}&\scriptscriptstyle{2a_{24}}&\scriptscriptstyle{0}&\scriptscriptstyle{a_{20}}&\scriptscriptstyle{-a_{11}}\\
\scriptscriptstyle{a_{3}}&\scriptscriptstyle{a_{14}}&\scriptscriptstyle{a_{21}}&\scriptscriptstyle{a_{30}}&\scriptscriptstyle{-a_{25}-a_{26}}&\scriptscriptstyle{a_{23}}&\scriptscriptstyle{-a_{18}}&\scriptscriptstyle{a_9}\\
\scriptscriptstyle{a_{4}}&\scriptscriptstyle{a_{15}}&\scriptscriptstyle{a_{22}}&\scriptscriptstyle{a_{25}+a_{26}}&\scriptscriptstyle{a_{30}}&\scriptscriptstyle{a_{24}}&\scriptscriptstyle{-a_{19}}&\scriptscriptstyle{a_{10}}\\
\scriptscriptstyle{a_{5}}&\scriptscriptstyle{a_{16}}&\scriptscriptstyle{0}&\scriptscriptstyle{2a_{21}}&\scriptscriptstyle{2a_{22}}&\scriptscriptstyle{a_{28}-a_{29}+a_{30}}&\scriptscriptstyle{a_{17}}&\scriptscriptstyle{-a_{8}}\\
\scriptscriptstyle{a_{6}}&\scriptscriptstyle{0}&\scriptscriptstyle{a_{16}}&\scriptscriptstyle{-2a_{14}}&\scriptscriptstyle{-2a_{15}}&\scriptscriptstyle{a_{13}}&\scriptscriptstyle{a_{27}-a_{29}+a_{30}}&\scriptscriptstyle{a_{7}}\\
\scriptscriptstyle{0}&\scriptscriptstyle{a_{6}}&\scriptscriptstyle{-a_{5}}&\scriptscriptstyle{2a_{3}}&\scriptscriptstyle{2a_{4}}&\scriptscriptstyle{-a_{2}}&\scriptscriptstyle{a_{1}}&\scriptscriptstyle{-a_{27}-a_{28}}
\ema,$$
and the real matrix $(A_\mu{}^\nu)$ is:
\begin{equation}
  (A_\mu{}^\nu)=\sum_{k=1}^{30} a_k E_k,
  \label{fa1}
\end{equation}
with $16\times 16$ matrices $E_k$, $k=1,2,\dots,30$ given in the 
Appendix~A.
Note that there are 30 real matrix coefficients $a_k$, $k=1,2,\dots,30$, in the matrix $(A_B{}^C)$.

The above matrices $(A_A{}^B)$, are closed with respect to the commutator $([A,A']_B{}^C)$ $=(A_B{}^DA'_D{}^C-A'_B{}^DA_D{}^C)$, as they should be, and form the $\mathfrak{n}_0$ of the Tanaka prolongation $\mathfrak{g}_T(\mathfrak{n})$.

What is this Lie algebra? By looking for a symmetric tensor $g_{ij}=g_{ji}$ invariant, $A_i{}^kg_{kj}+A_j{}^kg_{ik}=\tfrac18\mathrm{Tr}(A)g_{ij}$, under the adjoint action  of the matrix $(A_i{}^j)$ in $\mathfrak{n}_{\minu 2}$, we find that $(g_{ij})$ is a multiple of the numerical matrix
$$(g_{ij})=\bma
0&0&0&0&0&0&0&1\\
0&0&0&0&0&0&-1&0\\
0&0&0&0&0&1&0&0\\
0&0&0&-\tfrac12&0&0&0&0\\
0&0&0&0&-\tfrac12&0&0&0\\
0&0&1&0&0&0&0&0\\
0&-1&0&0&0&0&0&0\\
1&0&0&0&0&0&0&0
\ema.$$
This shows that $A$ acts in $\mathfrak{n}_{\minu 2}$ as the Lie algebra $\mathfrak{cso}(3,5)$. Thus, the Lie algebra $\mathfrak{n}_0$ is
$$\mathfrak{n}_0=\bbR\oplus\mathfrak{cso}(3,5)=\bbR\oplus\bbR\oplus \soa(3,5)=2\bbR\oplus{\raisebox{-1.5\depth}{\begin{turn}{120}\tikzset{/Dynkin diagram/fold style/.style={stealth-stealth,thin,
shorten <=1mm,shorten >=1mm}}
\begin{dynkinDiagram}[edge length=.4cm
  ]{D}{oooo}\dynkinFold{1}{4}
\end{dynkinDiagram}\end{turn}}}
.$$
Further calculation shows that
$\mathfrak{n}_1$ is a Lie algebra of dimension $16=\dim(\mathfrak{n}_{\minu 1})$, that $\dim(\mathfrak{n}_2)=8=\dim(\mathfrak{n}_{\minu 2})$, and that $\mathfrak{n}_k=\{0\}$ for all $k>2$. This shows that the Lie algebra of automorphisms of the distribution $H$ from our Lemma, $\mathfrak{aut}(H)$, being the Tanaka prolongation $\mathfrak{g}_T(\mathfrak{n})=\mathfrak{aut}(H)$, has a \emph{symmetric} gradation
$$\mathfrak{aut}(H)=\mathfrak{n}_{\minu 2}\oplus\mathfrak{n}_{\minu 1}\oplus\mathfrak{n}_0\oplus\mathfrak{n}_1\oplus\mathfrak{n}_2,$$
with its dimension $$78=8+16+30+16+8.$$
This is typical of a \emph{parabolic gradation} in a \emph{simple} Lie algebra. In our case it turns out that this Lie algebra is \begin{dynkinDiagram}[edge length=.4cm
  ]{E}{toooot}
\invol{1}{6}\invol{3}{5}
 \end{dynkinDiagram} with the choice of a parabolic as indicated by the crossings\footnote{Note that if you remove the crossed nodes, you will remain with the Dynkin diagram for $\soa(3,5)$, which together with the crossed nodes, counting as $2\bbR$, gives the calculated $\mathfrak{n}_0$.}.  This finishes the proof of the Lemma. 
  \end{roof}\\
  
  Returning to the proof of Theorem \ref{e2}, we see from the Lemma that our CR structure on $M^{24}_{E_{II}}$ has locally $E_{II}$ symmetric Levi distribution $H$. But the distribution $H$ defines $\mathfrak{n}_0$ which naturally acts in $H$ at every point. We emphasize, that $H$ \emph{itself} defines $\mathfrak{n}_0$ and its action. One therefore is tempted to see if there is any object, say a tensor $T$, in $H$ which is invariant with repsect to this action. So now, we want to see if there exists a rank $\binom{1}{1}$ tensor $J$ in $H$ preserved by the action of $\mathfrak{n}_0$ in $H$, which is the same as asking about the existence of a $16\times 16$ matrix $(T_\mu{}^\nu)$ such that
  \begin{equation}
    T_\mu{}^\rho A_\rho{}^\nu=A_\mu{}^\rho T_\rho{}^\nu\,\,\,\forall \mu,\nu=9,10,\dots, 24\label{inJ}.
  \end{equation}
  Here the matrix $(A\mu{}^\nu)$ is as in \eqref{abc-bis}.

  We solved the equations \eqref{inJ} for an $\mathfrak{n}_0$ invariant tensor $T$ in $H$. It turns out that there is a 2-parameter family of such invariant $T$s, parametrized by real numbers say $\alpha$, $\beta$,
  $$T=\alpha \id_H+\beta J,$$
  with
  \begin{equation}
    J=\bma 0&|& -\mathrm{id}_{8\times 8}\\--&&--\\\mathrm{id}_{8\times 8}&|&0\ema=E_{25}-E_{26}.\label{jot}\end{equation}
  Among such invariant $T$'s there is a unique (up to a sign) $T$ such that  $T^2=-\id_H$. This happens if and only if $\alpha=0$, $\beta=\pm1$. 
 
  So, although we did not assumed any complex structure $J$ in $H$, there is a \emph{prefered one} in there! It is defined uniquely up to a sign by the requirements that $J$ in $H$ satisfies $J^2=-\mathrm{id}_H$ and that it commutes with the action of $\mathfrak{n}_0$. And because $\mathfrak{n}_0$ is also fully defined by $H$, this $J$ is defined by the real distribution $H$ alone.

  If such $J$ were not $\mathfrak{n}_0$ invariant, then when calculating the Tanaka prolongation of the composed structure $(H,J)$, one should not only require that $\mathfrak{n}_0$ consisted of the strata preserving derivations of $\mathfrak{n}$, but that $\mathfrak{n}_0$ consists of those strata preserving derivations in $\mathfrak{n}$ which in addition preserve $J$ in $\mathfrak{n}_{\minu 1}$. In our case, since $J$ commutes with the $\mathfrak{n}_0$, the two Lie algebras `the strata preserving derivations of $\mathfrak{n}$' and `the strata preserving derivations of $\mathfrak{n}$ which also preserve $J$ in $\mathfrak{n}_{\minu 1}$' are the same. Thus the entire Tanaka prolongation of $\mathfrak{n}$ related to the structure of the distribution $H$ alone, and the Tanaka prolongation of the structure of the distribution $H$ with $J$, because it is $H$-defined, are also the same. As a consequence, the structure $(M^{24}_{E_{II}}, H,J)$ has the same Tanaka prolongation as $(M^{24}_{E_{II}}, H)$, and in turn the entire {\bf CR} structure $(M^{24}_{E_{II}},H,J)$ is (locally) $E_{II}$ symmetric.

  Finally one uses the $(M^{24}_{E_{II}},H,J)$ obtained from $(M^{24}_{E_{II}},H)$, checks that the corresponding $(H^\perp)^\bbC\subset Z^*\subset(\mathrm{T}^*M)^\bbC$ is integrable,  and embeds it in $\bbC^{16}$. It follows that this embedding is CR equivalent to the one given in Theorem \ref{e2}.

  This finishes the proof of this theorem.
\end{roof}
\vspace{1cm}

We now pass to the proof of Theorem \ref{e3}. Since it is almost the same as the proof of Theorem \ref{e2} we only give formulae that are different in this theorem when compared to the formulae in Theorem \ref{e2}.

\begin{roof}{{\bf Theorem}}{\ref{e3}}
  This time, the real distribution $H$ in $M^{24}_{E_{III}}$ is given as the annihilator of the Pfaffian forms $[\lambda^1,\dots,\lambda^8]$, which in the real coordinates $(u^i,x^i,y^i)$ are given in Corollary \ref{co14}.
  Again introducing $\lambda^{i+8}=\der x^i$ and $\lambda^{i+16}=\der y^i$, $i=1,\dots,8$, we easily see that these forms satisfy the following EDS:
  $$\begin{aligned}
    \der\lambda^1=\,\,&\lambda^9\dz \lambda^{24}+\lambda^{10}\dz \lambda^{20}+\lambda^{11}\dz \lambda^{23}+\lambda^{12}\dz \lambda^{18}+\lambda^{13}\dz \lambda^{22}+\lambda^{14}\dz \lambda^{21}+\\&\lambda^{15}\dz \lambda^{19}+\lambda^{16}\dz \lambda^{17}\\
    \der\lambda^2=\,\,&\lambda^9\dz \lambda^{20}+\lambda^{24}\dz \lambda^{10}+\lambda^{22}\dz \lambda^{11}+\lambda^{12}\dz \lambda^{17}+\lambda^{13}\dz \lambda^{23}+\lambda^{19}\dz \lambda^{14}+\\&\lambda^{15}\dz \lambda^{21}+\lambda^{18} \dz \lambda^{16}\\
    \der\lambda^3=\,\,&\lambda^9\dz \lambda^{23}+\lambda^{10}\dz \lambda^{22}+\lambda^{24}\dz \lambda^{11}+\lambda^{21}\dz \lambda^{12}+\lambda^{20}\dz \lambda^{13}+\lambda^{14}\dz \lambda^{18}+\\&\lambda^{15}\dz \lambda^{17}+\lambda^{19}\dz \lambda^{16}\\
    \der\lambda^4=\,\,&\lambda^{10}\dz \lambda^9+\lambda^{13}\dz \lambda^{11}+\lambda^{16}\dz \lambda^{12}+\lambda^{15}\dz \lambda^{14}+\lambda^{18}\dz \lambda^{17}+\lambda^{21}\dz \lambda^{19}+\\&\lambda^{24}\dz \lambda^{20}+\lambda^{23}\dz \lambda^{22}\\
    \der\lambda^5=\,\,&\lambda^9\dz \lambda^{22}+\lambda^{23}\dz \lambda^{10}+\lambda^{11}\dz \lambda^{20}+\lambda^{12}\dz \lambda^{19}+\lambda^{24}\dz \lambda^{13}+\lambda^{14}\dz \lambda^{17}+\\&\lambda^{18}\dz \lambda^{15}+\lambda^{21}\dz \lambda^{16}\\
    \der\lambda^6=\,\,&\lambda^{13}\dz \lambda^9+\lambda^{11}\dz \lambda^{10}+\lambda^{12}\dz \lambda^{15}+\lambda^{16}\dz \lambda^{14}+\lambda^{21}\dz \lambda^{17}+\lambda^{19}\dz \lambda^{18}+\\&\lambda^{20}\dz \lambda^{23}+\lambda^{24}\dz \lambda^{22}\\
    \der\lambda^7=\,\,&\lambda^{11}\dz \lambda^9+\lambda^{10}\dz \lambda^{13}+\lambda^{14}\dz \lambda^{12}+\lambda^{16}\dz \lambda^{15}+\lambda^{19}\dz \lambda^{17}+\lambda^{18}\dz \lambda^{21}+\\&\lambda^{22}\dz \lambda^{20}+\lambda^{24}\dz \lambda^{23}\\
    \der\lambda^8=\,\,&\lambda^9\dz \lambda^{17}+\lambda^{10}\dz \lambda^{18}+\lambda^{11}\dz \lambda^{19}+\lambda^{12}\dz \lambda^{20}+\lambda^{13}\dz \lambda^{21}+\lambda^{14}\dz \lambda^{22}+\\&\lambda^{15}\dz \lambda^{23}+\lambda^{16}\dz \lambda^{24}\\
    \der\lambda^\mu=\,\,&0,\quad\forall \mu=9,10,\dots,24.
  \end{aligned}
  $$
  The constancy of the coefficients $C^A{}_{BD}=-C^A{}_{DB}$ in $\der\lambda^A=-\tfrac12 C^{}_{BD}\lambda^B\dz\lambda^D$ above, and their algebraic structure, again show that these are the Maurer-Cartan forms on a nilpotent Lie group, say $\mathcal N$, with the Lie algebra $\mathfrak{n}$ being 2-step nilpotent
  $$\mathfrak{n}=\mathfrak{n}_{\minu 2}\oplus\mathfrak{n}_{\minu 1},$$
  with
  $$\mathfrak{n}_{\minu 1}=\Span_\bbR(X_9,\dots ,X_{24}), \quad \mathfrak{n}_{\minu 2}=\Span_\bbR(X_1,\dots, X_8),$$
  and $X_A\hook \lambda^B=\delta^B_A$.

  The Tanaka prolongation $\mathfrak{g}_T(\mathfrak{n})$ of this nilpotent Lie algebra $\mathfrak{n}$ has $\mathfrak{n}_0$ consisting of matrices (with commutator in matrices) of the form
  $$
  (A_B{}^C)=\bma A_i{}^j&|&0\\--&&--\\0&|&A_\mu{}^\nu\ema,$$
  where
$$(A_i{}^j)=\bma a_{30}&a_{22}&a_{28}&a_{13}&a_{27}&a_{25}&a_{18}&a_7\\
-a_{22}&a_{30}&-a_{21}&a_{9}&-a_{20}&-a_{19}&a_{14}&a_3\\
-a_{28}&a_{21}&a_{30}&a_{12}&a_{26}&a_{24}&a_{17}&a_6\\
-a_{13}&-a_{9}&-a_{12}&a_{30}&-a_{11}&-a_{10}&-a_{8}&-a_1\\
-a_{27}&a_{20}&-a_{26}&a_{11}&a_{30}&a_{23}&a_{16}&a_5\\
-a_{25}&a_{19}&-a_{24}&a_{10}&-a_{23}&a_{30}&a_{15}&-a_4\\
-a_{18}&-a_{14}&-a_{17}&a_{8}&-a_{16}&-a_{15}&a_{30}&-a_2\\
a_{7}&a_{3}&a_{6}&-a_{1}&a_{5}&-a_{4}&-a_{2}&a_{30}
\ema,$$
and the real matrix $(A_\mu{}^\nu)$ is:
\begin{equation}
  (A_\mu{}^\nu)=\sum_{k=1}^{30} \tfrac12 a_k \tilde{E}_k,
  \label{fa2}
\end{equation}
with $16\times 16$ matrices $\tilde{E}_k$, $k=1,2,\dots,30$ given in the Appendix~B.
Note that again there are 30 real matrix coeffcients $a_k$, $k=1,2,\dots,30$, in the matrix $(A_B{}^C)$.

The above matrices $(A_A{}^B)$, are closed with respect to the commutator $([A,A']_B{}^C)$ $=(A_B{}^DA'_D{}^C-A'_B{}^DA_D{}^C)$, as they should be, and form the $\mathfrak{n}_0$ of the Tanaka prolongation $\mathfrak{g}_T(\mathfrak{n})$.

Asking again the question of `what is this Lie algebra?', by the same argument of looking for a symmetric tensor $g_{ij}=g_{ji}$ invariant under the adjoint action  of the matrix $(A_i{}^j)$ in $\mathfrak{n}_{\minu 2}$, and finding that $(g_{ij})$ is a multiple of the diagonal matrix of the form $\mathrm{diag}(-1,-1,-1,-1,-1,-1,-1,1)$, one sees that $A$ acts in $\mathfrak{n}_{\minu 2}$ as the Lie algebra $\mathfrak{cso}(1,7)$. Thus, the Lie algebra $\mathfrak{n}_0$ is 
$$\mathfrak{n}_0=\bbR\oplus\mathfrak{cso}(1,7)=\bbR\oplus\bbR\oplus \soa(1,7)=2\bbR\oplus{\raisebox{-1.2\depth}{\begin{turn}{120}
\begin{dynkinDiagram}[edge length=.4cm
  ]{D}{**o*}
\end{dynkinDiagram}
\end{turn}}}
.$$
We then calculated $\mathfrak{n}_1$ and $\mathfrak{n}_2$ obtaining, in particular, that
$\mathfrak{n}_1$ is a Lie algebra of dimension $16=\dim(\mathfrak{n}_{\minu 1})$, that $\dim(\mathfrak{n}_2)=8=\dim(\mathfrak{n}_{\minu 2})$, and that $\mathfrak{n}_k=\{0\}$ for all $k>2$. This, as in the case of Theorem \ref{e3}, shows that the Lie algebra of automorphisms of the distribution $H$ from our Lemma, $\mathfrak{aut}(H)$, being the Tanaka prolongation $\mathfrak{g}_T(\mathfrak{n})=\mathfrak{aut}(H)$, has a \emph{symmetric} gradation
$$\mathfrak{aut}(H)=\mathfrak{n}_{\minu 2}\oplus\mathfrak{n}_{\minu 1}\oplus\mathfrak{n}_0\oplus\mathfrak{n}_1\oplus\mathfrak{n}_2,$$
with its dimension $$78=8+16+30+16+8.$$
This leads to the conclusion that this Lie algebra is \begin{dynkinDiagram}[edge length=.4cm
  ]{E}{to***t}
\invol{1}{6}
 \end{dynkinDiagram} with the choice of a parabolic as indicated by the crossings\footnote{Note that if you remove the crossed nodes, you will remain with the Dynkin diagram for $\soa(1,7)$, which together with the crossed nodes, counting as $2\bbR$, gives the calculated $\mathfrak{n}_0$.}. This proves that the local group of symmetries of the \emph{distribution} $H$ from Theorem \ref{e3} is $E_{III}$ with the Lie algebra  \begin{dynkinDiagram}[edge length=.4cm
  ]{E}{to***t}
\invol{1}{6}
 \end{dynkinDiagram}.    
  
Now: `what about the symmetries of the CR structure on $M^{24}_{E_{III}}$?'. The point is that again the (Levi) distribution $H$ on its own, by its mere algebraic structure, defines $J$ in it. Again there exists a \emph{unique} (up to a sign) rank $\binom{1}{1}$ tensor $J$ in $H$ preserved by the action of $\mathfrak{n}_0$ in $H$ and squaring to `$-\id_H$'. In the basis $(X_9,\dots, X_{24})$ in $H$ it is given by
$$J=\bma 0&|& -\mathrm{id}_{8\times 8}\\--&&--\\\mathrm{id}_{8\times 8}&|&0\ema=\tilde{E}_{29}.$$
Thus, again, the structure $(M^{24}_{E_{III}}, H,J)$ has the same Tanaka prolongation as $(M^{24}_{E_{III}}, H)$, and in turn the entire structure $(M^{24}_{E_{III}},H,J)$ is (locally) $E_{III}$ symmetric.

Also, as in the previous case, one uses the $(M^{24}_{E_{III}},H,J)$ obtained from $(M^{24}_{E_{III}},H)$ and embedds it in $\bbC^{16}$. It follows that this embedding is equivalent to the one given in Theorem \ref{e3}.

  This finishes the proof of this theorem.
\end{roof}

\section{Classification of accidental CR graded simple Lie algebras}\label{sec5}
Following the discussion in Subsection \ref{approach} about our approach to search for accidental CR structures, we start out by classifying the corresponding graded simple Lie algebras. 

Recall that a real graded Lie algebra (abbreviated as GLA) is a real Lie algebra $\fg$ with a direct sum decomposition $\fg=\bigoplus_{p\in \bz} \fg_p$ such that each $\fg_p$ is a finite dimensional real vector subspace of $\fg$ and $[\fg_p, \fg_q]\subset \fg_{p+q}$ for all integers $p, q$. 

Note that $\fg_0$ is a subalgebra of $\fg$ and that by restriction of the adjoint representation, there is a natural representation of $\fg_0$ on $\fg_p$ for any integer $p$. 

We will denote by $\fn_-(\fg)$ the nilpotent subalgebra $\oplus_{p<0}\fg_p$ of $\fg$. The maximum $d$ of the set of integers $p$ for which $\fg_{-p}\neq 0$ is called the {\bf depth} of the GLA. A GLA $\fg=\bigoplus_{p\in \bz}\fg_p$ is called 
\emph{fundamental} if $\fg_{-1}$ generates $\fn_-(\fg)=\bigoplus_{p<0}\fg_p$.

A real {\bf simple} Lie algebra $\fg$ belongs to two disjoint families (see \cite{Yam}):
\begin{enumerate}
\item {\bf Complex type:} $\fg$ is a complex simple Lie algebra $\t \fg$ regarded as real. 
In this case, the complexification $\fg^\bc$ is only semisimple as a complex Lie algebra. 
\item {\bf Real type:} $\fg$ is a real form of a complex Lie algebra. In this case, $\fg^\bc$ is simple as a complex Lie algebra. 
\end{enumerate}

In this paper, we are mainly interested in the second case of real type. 
Note that for the Killing form $\kappa$ of a simple GLA $\fg$, we have $\kappa(\fg_p, \fg_q)\neq 0\iff p+q=0$. 

\begin{definition}\label{our def} Let 
$$
\fg = \fg_{-d} \oplus \cdots \oplus \fg_{-1} \oplus \fg_0\oplus \fg_1 \oplus \cdots \oplus \fg_d
$$
be a simple graded Lie algebra of real type. 
$\fg$ is said to have an \emph{accidental CR structure} if 
\begin{enumerate}
\item $\fg$ is fundamental,
\item $\fg$ is the Tanaka prolongation of $\fn_-$,
\item 
$
\fg_{-1}$  has an almost complex structure $J$ compatible with the $\fg_0$-action, 
that is, there exists an $\br$-linear transformation $J: \fg_{-1} \to \fg_{-1}$ such that 
$$J^2 = -\on{Id}_{\fg_{-1}}, \quad\text{and}\quad \ad_{H}\circ J = J\circ \ad_H,\ \forall H\in \fg_0.
$$
\end{enumerate}
\end{definition}

\begin{remark} The pattern of this definition should be adoptable to define other \emph{accidental structures} as discussed above in Subsection \ref{approach}. 
\end{remark}

The authors of this paper started with this definition and looked to classify such accidental CR graded simple Lie algebras. Through Maple calculations of many examples, they came to the realization that such accidental CR structures exist  when the grading roots for the gradation come in pairs in the Satake diagram. 

After the realization that their list of such candidate simple Lie algebras are exactly the same as those in \cite[Section 4]{medori}, the present authors were drawn to the fundamental works of Medori and Nacinovich \cites{MN1, medori}. 

The first paper \cite{MN1} systematically studied Levi-Tanaka algebras following the method of N. Tanaka. In the second paper \cite{medori}, Medori and Nacinovich classified all semisimple Levi-Tanaka algebras, of both the complex and the real types. 

For the reader's convenience, we recall their basic definitions and compare their results with our more naive version. For the application to CR geometry, \cites{MN1, medori} put the integrability conditions in \er{ni1} and \er{ni2} of $J$ in the forefront from the beginning, as  should be.

\begin{definition}[\cite{medori}*{p.~287}]\label{their def} Let $\fg=\bigoplus_{p\in \bz} \fg_p$ be a finite dimensional real graded Lie algebra. 
$\fg$ is called a \emph{Levi-Tanaka algebra} if 
\begin{enumerate}
\item $\fg$ is fundamental;
\item there is a \emph{partial complex structure} on $\fg$, that is, an $\br$-linear map $J: \fg_{-1}\to \fg_{-1}$ which satisfies
\begin{equation}
\begin{cases}
J^2 = -Id_{\fg_{-1}},\\
[JX, JY] = [X, Y],\quad \forall X, Y\in \fg_{-1}.\label{MNint}
\end{cases}
 \end{equation}
\item the adjoint representation gives an isomorphism between $\fg_0$ and the algebra of $0$-degree derivations of $\fn_-(\fg)$ whose restriction to $\fg_{-1}$ commutes with $J$;
\end{enumerate}
\end{definition}

\begin{definition}[\cite{medori}*{Section 4}]\label{sLT} A \emph{simple Levi-Tanaka Lie algebra of the real type} is a Levi-Tanaka Lie algebra that is also a simple Lie algebra of the real type. 
\end{definition}

If we compare our definition \ref{our def} and their definitions \ref{their def} and \ref{sLT}, we see that we didn't require the partial integrability condition \er{MNint}. Also a Levi-Tanaka algebra \ref{their def} is stronger in condition (3) requiring that $\fg_0$ is equal to the unitary algebra $\fs\fu(\fg_{-1})$ for $J$, while accidental CR structure \ref{our def} only requires that $\fg_0$ is a subalgebra of $\fs\fu(\fg_{-1})$. On the other hand, for accidental CR structure \ref{our def}, condition (2) explicitly requires that the simple Lie algebra is the Tanaka prolongation of $\mathfrak{n}_\minu$.

The gradations on a simple Lie algebra of the real type in the second case can be given in terms of the restricted roots or using the Satake diagram \cite{Yam}, with the latter giving more information.
Recall that the nodes on the Satake diagram are black when the corresponding simple roots are compact (or purely imaginary), or otherwise white. Some pairs of white nodes are joined by curved arrows if they are conjugate to each other, modulo the compact roots. 
The other white ones stand alone. 
Using the Satake diagram, a gradation on $\fg$ is given by a subset $\cb_{-1}$ of simple roots represented by white nodes, where one may choose some standing-alone white nodes or pairs of white nodes joined by curved arrows.

\begin{theorem}[\cite{medori}*{Thm~4.1}]\label{their thm} Let $\fg$ be a simple graded Lie algebra of the real type. 
Then $\fg$ admits the structure of a  Levi-Tanaka algebra if and only if its Satake diagram satisfies that the set of grading roots is nonempty and consists of a disjoint union of pairs of white roots joined by a curved arrow. 
\end{theorem}

Then \cite{medori}*{Section 4} lists all the possible simple Levi-Tanaka algebras of the real type as follows.  
\begin{enumerate}
\item $\fs\fu(p, q), 1\leq p<q, p+q=\ell + 1$ graded by 
$$\cb_{-1} = \{\a_{i_1},\dots,\a_{i_\nu},\a_{\ell-i_v+1},\dots,\a_{\ell-i_1+1}\}$$
where $1\leq i_1<\dots<i_v\leq p,$ with depth $2\nu$, 
\item $\fs\fu(p, p), p\geq 2, 2p=\ell+1$ graded by 
$$\cb_{-1}=\{\a_{i_1},\dots,\a_{i_\nu},\a_{\ell-i_v+1},\dots,\a_{\ell-i_1+1}\}$$ 
where $1\leq i_1<\dots<i_v\leq p-1,$ 
with depth $2\nu$,
\item $\fs\fo(\ell-1, \ell+1),\ell\geq 4$ graded by $\cb_{-1}=\{\a_{\ell-1}, \a_\ell\}$, with depth 2,
\item $\fs\fo^*(2\ell),\ell=2m+1,m\geq 2$ graded by $\cb_{-1}=\{\a_{\ell-1},\a_{\ell}\}$, with depth 2,
\item $E_{II}$ graded by $\cb_{-1} = \{\a_1,\a_6\}, \{\a_3,\a_5\}\text{ or }\{\a_1,\a_3,\a_5,\a_6\}$, with depths 2, 4, 6,
\item $E_{III}$ graded by $\cb_{-1} = \{\a_1,\a_6\}$ with depth 2. 
\end{enumerate}

Adapting their proof to our situation, we can also get the following result. 
\begin{theorem}\label{accr}
A simple Lie algebra of real type admits an accidental CR structure iff it admits a Levi-Tanaka structure other than the case of 
\begin{equation}\label{hyp type}
\begin{split}
&\fs\fu(p, q), 1\leq p<q, p+q=\ell + 1 \text{ graded by }\cb_{-1} = \{\a_1,\a_{\ell}\},\\
&\fs\fu(p, p), 2\leq p, 2p=\ell + 1 \text{ graded by }\cb_{-1} = \{\a_1,\a_{\ell}\}.
\end{split}
\end{equation}
Therefore the $J$ in the accidental CR structure can be made integrable, and up to a sign, such integrable $J$ is unique. 
\end{theorem}

\begin{proof} The criterion for a simple graded Lie algebra $\fg$ to be the prolongation of $\fn_-$ is given in \cite{Yam}*{Theorem 5.3}. Then all the simple Levi-Tanaka Lie algebras have accidental structures except those listed in \er{hyp type}, that is, the hypersurface type CR structure discussed in the introduction. 

Now in the other direction, we can show that an accidental CR structure always implies a Levi-Civita structure. 

The main technical advantage of the proof in \cite{medori} over our original more computational approach is their theorem 2.4. It states that the partial complex structure $J$ is induced by a unique element $\tilde J\in \fg_0$. This part does not require the integrability condition. All it requires is that the $J: \fg_{-1}\to \fg_{-1}$ defines a degree 0 derivation of $\fn_-$. In the Levi-Tanaka cases, this derivation is $J$ itself so it belongs to $\fg_0$. In our accidental CR structure case, condition (2) would imply that such a derivation belongs to the degree 0 part of the Tanaka prolongation and so in our $\fg_0$. 

Now the integrability condition is some extra condition we can require on our accidental CR structure, and up to a sign, it is unique as Medori and Nacinovich showed. This also holds for our accidental CR structures if we require $J$ to be integrable. 
\end{proof}

\begin{remark} Medori and Nacinovich \cite{medori} also classifies simple Levi-Tanaka Lie algebra of complex type. There, the gradation root sets are more flexible, and the integrability condition comes in much more prominently and is the main factor to decide the result. 
\end{remark}

We summarize the accidental CR structures with depth 2 that we studied in this paper in the following theorem. We also find among them the nonrigid ones. 

\begin{theorem}
\label{summm} 
\begin{enumerate}
\item The list of simple Lie algebras of real type with an accidental CR structure of depth 2 is as follows. 
\begin{enumerate}
\item $E_{II}$ graded by $\cb_{-1} = \{\a_1,\a_6\}$ corresponding to Theorem \ref{e2},
\item $E_{III}$ graded by $\cb_{-1} = \{\a_1,\a_6\}$ corresponding to Theorem \ref{e3},
\item $\fs\fo(\ell-1, \ell+1),\ell\geq 4$ graded by $\cb_{-1}=\{\a_{\ell-1}, \a_\ell\}$ corresponding to Theorem \ref{so},
\item $\fs\fo^*(2\ell),\ell=2m+1,m\geq 2$ graded by $\cb_{-1}=\{\a_{\ell-1},\a_{\ell}\}$ corresponding to Theorem \ref{so*},
\item $\fs\fu(p, q), 1\leq p\leq q, p+q=\ell + 1$ graded by 
$\cb_{-1} = \{\a_{s},\a_{\ell-s+1}\},$
where $2\leq s\leq p$ if $p<q$, and $2\leq s<p$ if $p=q$, corresponding to Theorem \ref{su}.  
\end{enumerate}
\item Among such accidental CR structures with depth 2, the only nonrigid ones are 
$\fs\fu(p, q), p+q=\ell + 1$ graded by 
$\cb_{-1} = \{\a_{2},\a_{\ell-1}\},$
where $2\leq p<q$ or $3\leq p=q$. 
\end{enumerate}
\end{theorem}

\begin{proof} Part (1) follows from Theorem \ref{their thm}, the list after it, and Theorem \ref{accr}. 

Recall that a parabolic geometry modeled on a graded Lie algebra $\fg$ with negative part $\fn_-$ is nonrigid if the Lie algebra cohomology $H^2(\fn_-, \fg)$ has nonzero components with nonnegative weights. K. Yamaguchi in \cite{Yam}*{Prop. 5.5} listed all such simple graded Lie algebras. Comparing that list with Part (1), we see that the only common cases are Yamaguchi's case (I.8) for $\ell=4$ and (I.10) for $\ell\geq 5$, that is $A_{\ell}$ graded by $\{\alpha_2, \alpha_{\ell-1}\}$ for $\ell\geq 4$. This is our Part (2). 
\end{proof}

\section{Appendix~A}
\subsection{Basis for $\bbR\oplus\mathfrak{cso}(3,5)$ in 16-dim representation}
First we introduce $16\times 16$ real matrices $F^\mu{}_\nu=(F^\mu{}_{\nu\alpha}{}^\beta)$, $\mu,\nu=1,2,\dots, 16$, $\alpha,\beta=1,2,\dots,16$, with matrix elements $F^\mu{}_{\nu\alpha}{}^\beta=\delta^\mu_\alpha\delta^\beta_\nu$. Then, the matrices spanning $\bbR\oplus\mathfrak{cso}(3,5)$ and standing in formula \eqref{fa1} are:
\[
\aligned
E_1=&F^5{}_3+F^6{}_4+F^{13}{}_{11}+F^{14}{}_{12},\\
E_2=&-F^5{}_{10}+F^7{}_{12}+F^{13}{}_{2}-F^{15}{}_{4},\\
E_3=&F^5{}_9+F^6{}_{10}+F^{7}{}_{11}+F^{8}{}_{12}-F^{13}{}_1-F^{14}{}_2-F^{15}{}_{3}-F^{16}{}_4,\\
E_4=&-F^5{}_1+F^6{}_{2}+F^{7}{}_{3}-F^{8}{}_{4}-F^{13}{}_9+F^{14}{}_{10}+F^{15}{}_{11}-F^{16}{}_{12},
\\
E_5=&-F^6{}_9+F^8{}_{11}+F^{14}{}_{1}-F^{16}{}_{3},
\endaligned
\]
\[
\aligned
E_6=&F^7{}_1+F^8{}_2+F^{15}{}_{9}+F^{16}{}_{10},\\
E_7=&F^3{}_5+F^4{}_6+F^{11}{}_{13}+F^{12}{}_{14},\\
E_8=&F^2{}_{13}-F^4{}_{15}-F^{10}{}_{5}+F^{12}{}_{7},\\
E_9=&-F^1{}_{13}-F^2{}_{14}-F^{3}{}_{15}-F^{4}{}_{16}+F^{9}{}_5+F^{10}{}_{6}+F^{11}{}_{7}+F^{12}{}_{8},\\
E_{10}=&-F^1{}_5+F^2{}_{6}+F^{3}{}_{7}-F^{4}{}_{8}-F^{9}{}_{13}+F^{10}{}_{14}+F^{11}{}_{15}-F^{12}{}_{16},
\endaligned
\]
\[
\aligned
E_{11}=&F^1{}_{14}-F^3{}_{16}-F^{9}{}_{6}+F^{11}{}_{8},\\
E_{12}=&F^1{}_7+F^2{}_8+F^{9}{}_{15}+F^{10}{}_{16},\\
E_{13}=&F^3{}_{10}+F^7{}_{14}-F^{11}{}_{2}-F^{15}{}_{6},\\
E_{14}=&-F^3{}_9-F^4{}_{10}+F^{7}{}_{13}+F^{8}{}_{14}+F^{11}{}_{1}+F^{12}{}_{2}-F^{15}{}_{5}-F^{16}{}_{6},\\
E_{15}=&F^3{}_1-F^4{}_{2}+F^{7}{}_{5}-F^{8}{}_{6}+F^{11}{}_{9}-F^{12}{}_{10}+F^{15}{}_{13}-F^{16}{}_{14},
\endaligned
\]
\[
\aligned
E_{16}=&F^4{}_{9}+F^8{}_{13}-F^{12}{}_{1}-F^{16}{}_{5},\\
E_{17}=&-F^2{}_{11}-F^6{}_{15}+F^{10}{}_{3}+F^{14}{}_{7},\\
E_{18}=&F^1{}_{11}+F^2{}_{12}-F^{5}{}_{15}-F^{6}{}_{16}-F^{9}{}_{3}-F^{10}{}_{4}+F^{13}{}_{7}+F^{14}{}_{8},\\
E_{19}=&F^1{}_3-F^2{}_{4}+F^{5}{}_{7}-F^{6}{}_{8}+F^{9}{}_{11}-F^{10}{}_{12}+F^{13}{}_{15}-F^{14}{}_{16},\\
E_{20}=&-F^1{}_{12}-F^5{}_{16}+F^{9}{}_{4}+F^{13}{}_{8},
\endaligned
\]
\[
\aligned
E_{21}=&F^2{}_{1}-F^4{}_{3}-F^{6}{}_{5}+F^{8}{}_{7}+F^{10}{}_{9}-F^{12}{}_{11}-F^{14}{}_{13}+F^{16}{}_{15},\\
E_{22}=&F^2{}_9+F^4{}_{11}+F^{6}{}_{13}+F^{8}{}_{15}-F^{10}{}_{1}-F^{12}{}_{3}-F^{14}{}_{5}-F^{16}{}_{7},\\
E_{23}=&F^1{}_{2}-F^3{}_{4}-F^{5}{}_{6}+F^{7}{}_{8}+F^{9}{}_{10}-F^{11}{}_{12}-F^{13}{}_{14}+F^{15}{}_{16},\\
E_{24}=&-F^1{}_{10}-F^3{}_{12}-F^{5}{}_{14}-F^{7}{}_{16}+F^{9}{}_{2}+F^{11}{}_{4}+F^{13}{}_{6}+F^{15}{}_{8},\\
E_{25}=&-F^1{}_{9}-F^4{}_{12}-F^{6}{}_{14}-F^{7}{}_{15}+F^{9}{}_{1}+F^{12}{}_{4}+F^{14}{}_{6}+F^{15}{}_{7},
\endaligned
\]
\[
\aligned
E_{26}=&F^2{}_{10}+F^3{}_{11}+F^{5}{}_{13}+F^{8}{}_{16}-F^{10}{}_{2}-F^{11}{}_{3}-F^{13}{}_{5}-F^{16}{}_{8},\\
E_{27}=&F^3{}_{3}+F^4{}_{4}-F^{5}{}_{5}-F^{6}{}_{6}+F^{11}{}_{11}+F^{12}{}_{12}-F^{13}{}_{13}-F^{14}{}_{14},\\
E_{28}=&F^2{}_2+F^4{}_{4}-F^{5}{}_{5}-F^{7}{}_{7}+F^{10}{}_{10}+F^{12}{}_{12}-F^{13}{}_{13}-F^{15}{}_{15},\\
E_{29}=&F^1{}_{1}-F^4{}_{4}+F^{5}{}_{5}-F^{8}{}_{8}+F^{9}{}_{9}-F^{12}{}_{12}+F^{13}{}_{13}-F^{16}{}_{16},\\
E_{30}=&F^1{}_1+F^2{}_{2}+F^{3}{}_{3}+F^{4}{}_{4}+F^{9}{}_{9}+F^{10}{}_{10}+F^{11}{}_{11}+F^{12}{}_{12}.
\endaligned
\]

\section{Appendix~B}
\subsection{Basis for $\bbR\oplus\mathfrak{cso}(1,7)$ in 16-dim representation}
The matrices spanning $\bbR\oplus\mathfrak{cso}(1,7)$ and standing in formula \eqref{fa2} are:
\[
\aligned
\tilde{E}_1=&-(F^1{}_{10}+F^{10}{}_1)+(F^2{}_{9}+F^{9}{}_2)-(F^3{}_{13}+F^{13}{}_3)-(F^4{}_{16}+F^{16}{}_4)+\\&(F^5{}_{11}+F^{11}{}_5)-(F^6{}_{15}+F^{15}{}_6)+(F^7{}_{14}+F^{14}{}_7)+(F^8{}_{12}+F^{12}{}_8),\\
\tilde{E}_2=&-(F^1{}_{11}+F^{11}{}_1)+(F^2{}_{13}+F^{13}{}_2)+(F^3{}_{9}+F^{9}{}_3)-(F^4{}_{14}+F^{14}{}_4)-\\&(F^5{}_{10}+F^{10}{}_5)+(F^6{}_{12}+F^{12}{}_6)-(F^7{}_{16}+F^{16}{}_7)+(F^8{}_{15}+F^{15}{}_8),\\
\tilde{E}_3=&(F^1{}_{4}+F^{4}{}_1)-(F^2{}_{8}+F^{8}{}_2)-(F^3{}_{6}+F^{6}{}_3)+(F^5{}_{7}+F^{7}{}_5)+\\&(F^9{}_{12}+F^{12}{}_9)-(F^{10}{}_{16}+F^{16}{}_{10})-(F^{11}{}_{14}+F^{14}{}_{11})+(F^{13}{}_{15}+F^{15}{}_{13}),
\endaligned
\]
\[
\aligned
\tilde{E}_4=&-(F^1{}_{13}+F^{13}{}_1)-(F^2{}_{11}+F^{11}{}_2)+(F^3{}_{10}+F^{10}{}_3)+(F^4{}_{15}+F^{15}{}_4)+\\&(F^5{}_{9}+F^{9}{}_5)-(F^6{}_{16}+F^{16}{}_6)-(F^7{}_{12}+F^{12}{}_7)+(F^8{}_{14}+F^{14}{}_8),\\
\tilde{E}_5=&(F^1{}_{6}+F^{6}{}_1)-(F^2{}_{7}+F^{7}{}_2)+(F^3{}_{4}+F^{4}{}_3)-(F^5{}_{8}+F^{8}{}_5)+\\&(F^9{}_{14}+F^{14}{}_9)-(F^{10}{}_{15}+F^{15}{}_{10})+(F^{11}{}_{12}+F^{12}{}_{11})-(F^{13}{}_{16}+F^{16}{}_{13}),
\\
\tilde{E}_6=&(F^1{}_{7}+F^{7}{}_1)+(F^2{}_{6}+F^{6}{}_2)-(F^3{}_{8}+F^{8}{}_3)-(F^4{}_{5}+F^{5}{}_4)+\\&(F^9{}_{15}+F^{15}{}_9)+(F^{10}{}_{14}+F^{14}{}_{10})-(F^{11}{}_{16}+F^{16}{}_{11})-(F^{12}{}_{13}+F^{13}{}_{12}),
\\
\endaligned
\]
\[
\aligned
\tilde{E}_7=&(F^1{}_{8}+F^{8}{}_1)+(F^2{}_{4}+F^{4}{}_2)+(F^3{}_{7}+F^{7}{}_3)+(F^5{}_{6}+F^{6}{}_5)+\\&(F^9{}_{16}+F^{16}{}_9)+(F^{10}{}_{12}+F^{12}{}_{10})+(F^{11}{}_{15}+F^{15}{}_{11})+(F^{13}{}_{14}+F^{14}{}_{13}),
\\
\tilde{E}_8=&-(F^1{}_{5}-F^{5}{}_1)-(F^2{}_{3}-F^{3}{}_2)-(F^4{}_{7}-F^{7}{}_4)+(F^6{}_{8}-F^{8}{}_6)-\\&(F^9{}_{13}-F^{13}{}_9)-(F^{10}{}_{11}-F^{11}{}_{10})-(F^{12}{}_{15}-F^{15}{}_{12})+(F^{14}{}_{16}-F^{16}{}_{14}),
\\
\tilde{E}_9=&(F^1{}_{16}-F^{16}{}_1)+(F^2{}_{12}-F^{12}{}_2)-(F^3{}_{15}-F^{15}{}_3)+(F^4{}_{10}-F^{10}{}_4)-\\&(F^5{}_{14}-F^{14}{}_5)-(F^6{}_{13}-F^{13}{}_6)-(F^7{}_{11}-F^{11}{}_7)+(F^8{}_{9}-F^{9}{}_8),
\endaligned
\]
\[
\aligned
\tilde{E}_{10}=&(F^1{}_{3}-F^{3}{}_1)-(F^2{}_{5}-F^{5}{}_2)-(F^4{}_{6}-F^{6}{}_4)-(F^7{}_{8}-F^{8}{}_7)+\\&(F^9{}_{11}-F^{11}{}_9)-(F^{10}{}_{13}-F^{13}{}_{10})-(F^{12}{}_{14}-F^{14}{}_{12})-(F^{15}{}_{16}-F^{16}{}_{15}),
\\
\tilde{E}_{11}=&(F^1{}_{15}-F^{15}{}_1)+(F^2{}_{14}-F^{14}{}_2)+(F^3{}_{16}-F^{16}{}_3)+(F^4{}_{13}-F^{13}{}_4)+\\&(F^5{}_{12}-F^{12}{}_5)+(F^6{}_{10}-F^{10}{}_6)+(F^7{}_{9}-F^{9}{}_7)+(F^8{}_{11}-F^{11}{}_8),
\\
\tilde{E}_{12}=&-(F^1{}_{14}-F^{14}{}_1)+(F^2{}_{15}-F^{15}{}_2)+(F^3{}_{12}-F^{12}{}_3)+(F^4{}_{11}-F^{11}{}_4)-\\&(F^5{}_{16}-F^{16}{}_5)-(F^6{}_{9}-F^{9}{}_6)+(F^7{}_{10}-F^{10}{}_7)-(F^8{}_{13}-F^{13}{}_8),
\endaligned
\]
\[
\aligned
\tilde{E}_{13}=&-(F^1{}_{12}-F^{12}{}_1)+(F^2{}_{16}-F^{16}{}_2)-(F^3{}_{14}-F^{14}{}_3)-(F^4{}_{9}-F^{9}{}_4)+\\&(F^5{}_{15}-F^{15}{}_5)-(F^6{}_{11}-F^{11}{}_6)+(F^7{}_{13}-F^{13}{}_7)+(F^8{}_{10}-F^{10}{}_8),
\\
\tilde{E}_{14}=&(F^1{}_{14}-F^{14}{}_1)+(F^2{}_{15}-F^{15}{}_2)+(F^3{}_{12}-F^{12}{}_3)+(F^4{}_{11}-F^{11}{}_4)+\\&(F^5{}_{16}-F^{16}{}_5)+(F^6{}_{9}-F^{9}{}_6)+(F^7{}_{10}-F^{10}{}_7)+(F^8{}_{13}-F^{13}{}_8),
\\
\tilde{E}_{15}=&-(F^1{}_{2}-F^{2}{}_1)-(F^3{}_{5}-F^{5}{}_3)+(F^4{}_{8}-F^{8}{}_4)+(F^6{}_{7}-F^{7}{}_6)-\\&(F^9{}_{10}-F^{10}{}_9)-(F^{11}{}_{13}-F^{13}{}_{11})+(F^{12}{}_{16}-F^{16}{}_{12})+(F^{14}{}_{15}-F^{15}{}_{14}),
\endaligned
\]
\[
\aligned
\tilde{E}_{16}=&-(F^1{}_{12}-F^{12}{}_1)-(F^2{}_{16}-F^{16}{}_2)+(F^3{}_{14}-F^{14}{}_3)-(F^4{}_{9}-F^{9}{}_4)+\\&(F^5{}_{15}-F^{15}{}_5)+(F^6{}_{11}-F^{11}{}_6)+(F^7{}_{13}-F^{13}{}_7)-(F^8{}_{10}-F^{10}{}_8),
\\
\tilde{E}_{17}=&(F^1{}_{16}-F^{16}{}_1)-(F^2{}_{12}-F^{12}{}_2)+(F^3{}_{15}-F^{15}{}_3)-(F^4{}_{10}-F^{10}{}_4)-\\&(F^5{}_{14}-F^{14}{}_5)-(F^6{}_{13}-F^{13}{}_6)+(F^7{}_{11}-F^{11}{}_7)+(F^8{}_{9}-F^{9}{}_8),
\\
\tilde{E}_{18}=&-(F^1{}_{15}-F^{15}{}_1)+(F^2{}_{14}-F^{14}{}_2)+(F^3{}_{16}-F^{16}{}_3)-(F^4{}_{13}-F^{13}{}_4)-\\&(F^5{}_{12}-F^{12}{}_5)+(F^6{}_{10}-F^{10}{}_6)-(F^7{}_{9}-F^{9}{}_7)+(F^8{}_{11}-F^{11}{}_8),
\endaligned
\]
\[
\aligned
\tilde{E}_{19}=&(F^1{}_{15}-F^{15}{}_1)-(F^2{}_{14}-F^{14}{}_2)+(F^3{}_{16}-F^{16}{}_3)-(F^4{}_{13}-F^{13}{}_4)-\\&(F^5{}_{12}-F^{12}{}_5)-(F^6{}_{10}-F^{10}{}_6)+(F^7{}_{9}-F^{9}{}_7)+(F^8{}_{11}-F^{11}{}_8),
\\
\tilde{E}_{20}=&-(F^1{}_{3}-F^{3}{}_1)-(F^2{}_{5}-F^{5}{}_2)-(F^4{}_{6}-F^{6}{}_4)+(F^7{}_{8}-F^{8}{}_7)-\\&(F^9{}_{11}-F^{11}{}_9)-(F^{10}{}_{13}-F^{13}{}_{10})-(F^{12}{}_{14}-F^{14}{}_{12})+(F^{15}{}_{16}-F^{16}{}_{15}),
\\
\tilde{E}_{21}=&(F^1{}_{5}-F^{5}{}_1)-(F^2{}_{3}-F^{3}{}_2)-(F^4{}_{7}-F^{7}{}_4)-(F^6{}_{8}-F^{8}{}_6)+\\&(F^9{}_{13}-F^{13}{}_9)-(F^{10}{}_{11}-F^{11}{}_{10})-(F^{12}{}_{15}-F^{15}{}_{12})-(F^{14}{}_{16}-F^{16}{}_{14}),
\endaligned
\]
\[
\aligned
\tilde{E}_{22}=&-(F^1{}_{2}-F^{2}{}_1)+(F^3{}_{5}-F^{5}{}_3)-(F^4{}_{8}-F^{8}{}_4)+(F^6{}_{7}-F^{7}{}_6)-\\&(F^9{}_{10}-F^{10}{}_9)+(F^{11}{}_{13}-F^{13}{}_{11})-(F^{12}{}_{16}-F^{16}{}_{12})+(F^{14}{}_{15}-F^{15}{}_{14}),
\\
\tilde{E}_{23}=&(F^1{}_{16}-F^{16}{}_1)-(F^2{}_{12}-F^{12}{}_2)-(F^3{}_{15}-F^{15}{}_3)-(F^4{}_{10}-F^{10}{}_4)+\\&(F^5{}_{14}-F^{14}{}_5)+(F^6{}_{13}-F^{13}{}_6)-(F^7{}_{11}-F^{11}{}_7)+(F^8{}_{9}-F^{9}{}_8),
\\
\tilde{E}_{24}=&(F^1{}_{12}-F^{12}{}_1)+(F^2{}_{16}-F^{16}{}_2)+(F^3{}_{14}-F^{14}{}_3)+(F^4{}_{9}-F^{9}{}_4)+\\&(F^5{}_{15}-F^{15}{}_5)+(F^6{}_{11}-F^{11}{}_6)+(F^7{}_{13}-F^{13}{}_7)+(F^8{}_{10}-F^{10}{}_8),
\endaligned
\]
\[
\aligned
\tilde{E}_{25}=&-(F^1{}_{14}-F^{14}{}_1)-(F^2{}_{15}-F^{15}{}_2)+(F^3{}_{12}-F^{12}{}_3)+(F^4{}_{11}-F^{11}{}_4)+\\&(F^5{}_{16}-F^{16}{}_5)-(F^6{}_{9}-F^{9}{}_6)-(F^7{}_{10}-F^{10}{}_7)+(F^8{}_{13}-F^{13}{}_8),
\\
\tilde{E}_{26}=&-(F^1{}_{2}-F^{2}{}_1)+(F^3{}_{5}-F^{5}{}_3)+(F^4{}_{8}-F^{8}{}_4)-(F^6{}_{7}-F^{7}{}_6)-\\&(F^9{}_{10}-F^{10}{}_9)+(F^{11}{}_{13}-F^{13}{}_{11})+(F^{12}{}_{16}-F^{16}{}_{12})-(F^{14}{}_{15}-F^{15}{}_{14}),
\\
\tilde{E}_{27}=&-(F^1{}_{5}-F^{5}{}_1)+(F^2{}_{3}-F^{3}{}_2)-(F^4{}_{7}-F^{7}{}_4)-(F^6{}_{8}-F^{8}{}_6)-\\&(F^9{}_{13}-F^{13}{}_9)+(F^{10}{}_{11}-F^{11}{}_{10})-(F^{12}{}_{15}-F^{15}{}_{12})-(F^{14}{}_{16}-F^{16}{}_{14}),
\\
\tilde{E}_{28}=&-(F^1{}_{3}-F^{3}{}_1)-(F^2{}_{5}-F^{5}{}_2)+(F^4{}_{6}-F^{6}{}_4)-(F^7{}_{8}-F^{8}{}_7)-\\&(F^9{}_{11}-F^{11}{}_9)-(F^{10}{}_{13}-F^{13}{}_{10})+(F^{12}{}_{14}-F^{14}{}_{12})-(F^{15}{}_{16}-F^{16}{}_{15}),
\endaligned
\]
with $\tilde{E}_{29}=J$ as in \eqref{jot}, 
and $\tilde{E}_{30}=\mathrm{id}_{16\times 16}$. 

 $$\begin{aligned}
   &\lambda_1=\der\Big(x^8-x^7-\tfrac14\big((x^1)^2+(x^2)^2+(x^3)^2+(x^4)^2\big)\Big)+\tfrac12\big(x^2\der x^1-x^1\der x^2+x^4\der x^3-x^3\der x^4\big)\\
   &\lambda_2=\der\Big(x^8-x^5-\tfrac14\big((x^1)^2+(x^2)^2+(x^3)^2+(x^4)^2\big)\Big)+\tfrac12\big(x^3\der x^1-x^1\der x^3+x^2\der x^4-x^4\der x^2\big)\\
    &\lambda_3=\der\Big(x^8-x^6-\tfrac14\big((x^1)^2+(x^2)^2+(x^3)^2+(x^4)^2\big)\Big)+\tfrac12\big(x^4\der x^1-x^1\der x^4+x^3\der x^2-x^2\der x^3\big)\\
        \end{aligned}$$

$$\begin{aligned}
   &\lambda_1=\der\big(x^8-x^7+x^6+x^5\big)+x^2\der x^1-x^1\der x^2+x^4\der x^3-x^3\der x^4,\\
   &\lambda_2=\der\big(x^8+x^7+x^6-x^5\big)+x^3\der x^1-x^1\der x^3+x^2\der x^4-x^4\der x^2,\\
    &\lambda_3=\der\big(x^8+x^7-x^6+x^5\big)+x^4\der x^1-x^1\der x^4+x^3\der x^2-x^2\der x^3.\\
        \end{aligned}$$

\section{Appendix~C}
Here we present the explicit formulae for the CR symmetry generators of the $\sog(\ell+1,\ell -1)$ homogenous CR manifolds described in Theorem \ref{so}. The notation is as in this theorem.
Thus, in coordinates:
\[
z
\,=\,
(z^1,\dots,z^{\ell-1})
\ \ \ \ \ \ \ \ \ \ \ \ \ \ \ \ \ \ \ \
\text{and}
\ \ \ \ \ \ \ \ \ \ \ \ \ \ \ \ \ \ \ \
w
\,=\,
\big(
w^{ij}
\big)_{1\leqslant i<j\leqslant\ell-1},
\]
we have the CR submanifold $M^{N(\ell)} \subset
\C^{\ell(\ell-1)/2}$ of
Subsection~{\ref{subsection-SO-ell-1-ell-1}} defined by
\[
w^{ij}
-
\overline{w}^{ij}
\,=\,
z^i\,\overline{z}^j
-
\overline{z}^i\,
z^j
\eqno
{\scriptstyle{(1\,\leqslant\,i\,<\,j\,\leqslant\,\ell-1)}},
\]
with
\[
{\rm CRdim}\,M
\,=\,
\ell-1
\ \ \ \ \ \ \ \ \ \ \ \ \ \ \ \ \ \ \ \
\text{and}
\ \ \ \ \ \ \ \ \ \ \ \ \ \ \ \ \ \ \ \
{\rm codim}\,M
\,=\,
\tfrac{(\ell-1)(\ell-2)}{2}.
\]

As is known, when the CR structure is embedded,
the real infinitesimal CR symmetries $Y$, as defined in {\eqref{sycr}}, are determined by the holomorphic vector fields:
\[
Y
\,=\,
\sum_{i=1}^{\ell-1}\,
Z_i(z,w)\,
\partial_{z^i}
+
\sum_{j=1}^{\ell-1}\,
\sum_{k=1}^{\ell-1}\,
W_{jk}(z,w)\,
\partial_{w^{jk}},
\]
whose (double) real part $Y+\overline{Y}$ 
is {\em tangent} to the (extrinsic) CR manifold:
\[
\mathfrak{hol}(M)
\,:=\,
\big\{
Y\,\colon\,\,
Y+\overline{Y}\,\,
\text{is tangent to}\,\,
M
\big\}.
\]
Defining symmetries via this requirement is equivalent to the definition \eqref{sycr}. In particular, $\mathfrak{hol}(M)$ is a \emph{real} Lie algebra isomorphic to the Lie algebra of symmetries $\mathfrak{g}_J$ as defined in \eqref{sycr}.

Below, to save space,
only the holomorphic part of the symmetry is written; to get the real symmetry $Y$ one has to add the term `$+\overline{Y}$', in the formulae below.

Attributing the weights:
\[
[z^i]:=1=:[\bar{z}^i],
\ \ \ \ \
\big[\partial_{z^i}\big]
:=-1=:
\big[\partial_{\bar{z}^i}\big]
\ \ \ \ \ 
[w^{jk}]:=2=:[\bar{w}^{jk}],
\ \ \ \ \ 
\big[\partial_{w^{jk}}\big]
:=-2=:
\big[\partial_{\bar{w}^{jk}}\big]
\]
there is a grading:
\[
\mathfrak{hol}(M)
\,=\,
\mathfrak{g}_{-2}
\oplus
\mathfrak{g}_{-1}
\oplus
\mathfrak{g}_{0}
\oplus
\mathfrak{g}_{1}
\oplus
\mathfrak{g}_{2}
\,=\,
\SO(\ell-1,\ell+1)
\ \ \ \ \ \ \ \ \ 
\text{with}
\ \ \ \ \ \ \ \ \
\dim\,\mathfrak{g}_{-\lambda}
\,=\,
\dim\,\mathfrak{g}_{\lambda},
\]
where
\[
\mathfrak{g}_\lambda
\,=\,
\big\{
Y\in\mathfrak{hol}(M)
\colon\,\,
[Y]=\lambda
\big\}.
\]
Dimensions are:
\begin{center}
\def\arraystretch{1.5}
\begin{tabular}{c|c|c|c|c}
$\mathfrak{g}_{-2}$ & $\mathfrak{g}_{-1}$ & $\mathfrak{g}_0$ & 
$\mathfrak{g}_1$ & $\mathfrak{g}_2$
\\
\hline
$\frac{(\ell-1)(\ell-2)}{2}$ & $2\,(\ell-1)$ & $(\ell-1)^2+1$ & 
$2\,(\ell-1)$ & $\frac{(\ell-1)(\ell-2)}{2}$
\end{tabular}
\end{center}

The $\frac{(\ell-1)(\ell-2)}{2}$ generators of $\mathfrak{g}_{-2}$ are:
\[
Y_{w^{ij}}^{-2}
\,:=\,
\partial_{w^{ij}}
\eqno
{\scriptstyle{(1\,\leqslant\,i\,<\,j\,\leqslant\,\ell-1)}}.
\]
The $(\ell-1) + (\ell-1)$ generators of $\mathfrak{g}_{-1}$ are:
\[
\aligned
Y_{z^i}^{-1}
&
\,:=\,
\partial_{z^i}
+
\sum_{1\leqslant k<i}\,
z^k\,\partial_{w^{ki}}
-
\sum_{i<k\leqslant\ell-1}\,
z^k\,
\partial_{w^{ik}}
&
\ \ \ \ \ \ \ \ \ \ \ \
&
{\scriptstyle{(1\,\leqslant\,i\,\leqslant\,\ell-1)}},
\\
IY_{z^i}^{-1}
&
\,:=\,
{\scriptstyle{\sqrt{-1}}}\,\partial_{z^i}
-
\sum_{1\leqslant k<i}\,
{\scriptstyle{\sqrt{-1}}}\,z^k\,\partial_{w^{ki}}
+
\sum_{i<k\leqslant\ell-1}\,
{\scriptstyle{\sqrt{-1}}}\,z^k\,
\partial_{w^{ik}}
&
\ \ \ \ \ \ \ \ \ \ \ \
&
{\scriptstyle{(1\,\leqslant\,i\,\leqslant\,\ell-1)}},
\endaligned
\]
The $\frac{(\ell-1)(\ell-2)}{2} + (\ell-1) + 
\frac{(\ell-1)(\ell-2)}{2} + 1$ 
generators of $\mathfrak{g}_0$ are:
\[
\aligned
Y_{ij}^0
&
\,:=\,
z^i\,\partial_{z^j}
+
\sum_{1\leqslant k<i}\,
w^{ki}\,
\partial_{w^{kj}}
-
\sum_{i<k<j}\,
w^{ik}\,
\partial_{w^{kj}}
+
\sum_{j<k\leqslant\ell-1}\,
w^{ik}\,
\partial_{w^{jk}}
&
\ \ \
&
{\scriptstyle{(1\,\leqslant\,i\,<\,j\,\leqslant\,\ell-1)}},
\\
Y_{ii}^0
&
\,:=\,
z^i\,\partial_{z^i}
+ 
\sum_{1\leqslant k<i}\,
w^{ki}\,
\partial_{w^{ki}}
+
\sum_{i<k\leqslant\ell-1}\,
w^{ik}\,
\partial_{w^{ik}}
&
\ \ \
&
{\scriptstyle{(1\,\leqslant\,i\,\leqslant\,\ell-1)}},
\\
Y_{ij}^0
&
\,:=\,
z^i\,\partial_{z^j}
+
\sum_{1\leqslant k<j}\,
w^{ki}\,
\partial_{w^{kj}}
-
\sum_{j<k<i}\,
w^{ki}\,
\partial_{w^{jk}}
+
\sum_{i<k\leqslant\ell-1}\,
w^{ik}\,
\partial_{w^{jk}},
&
\ \ \
&
{\scriptstyle{(1\,\leqslant\,j\,<\,i\,\leqslant\,\ell-1)}},
\endaligned
\]
together with the rotation:
\[
R
\,:=\,
{\scriptstyle{\sqrt{-1}}}\,z^1\,\partial_{z^1}
+\cdots+
{\scriptstyle{\sqrt{-1}}}\,z^{\ell-1}\,\partial_{z^{\ell-1}}.
\]
The sum of the $Y_{ii}^0$ equals the dilation:
\[
D
\,:=\,
\sum_{1\leqslant k\leqslant\ell-1}\,
z^k\,\partial_{z^k}
+
\sum_{1\leqslant k<m\leqslant\ell-1}\,
w^{km}\,
\partial_{w^{km}}.
\]
The $(\ell-1) + (\ell-1)$ generators of $\mathfrak{g}_1$ are:
\[
\footnotesize
\aligned
Y_{z^iz^i}^1
&
\,:=\,
\sum_{1\leqslant k<i}\,
\big(z^i\,z^k-w^{ki}\big)\,
\partial_{z^k}
+
z^i\,z^i\,\partial_{z^i}
+
\sum_{i<k\leqslant\ell-1}\,
\big(
z^i\,z^k
+
w^{ik}
\big)\,
\partial_{z^k}
\\
&
\ \ \ \ \ 
+
\sum_{1\leqslant k<i}\,
z^i\,w^{ki}\,\partial_{w^{ki}}
+
\sum_{i<m\leqslant\ell-1}\,
z^i\,w^{im}\,\partial_{w^{im}}
\\
&
\ \ \ \ \
+
\sum_{1\leqslant k<m<i}\,
\left\vert\!
\begin{smallmatrix}
z^k & -w^{ki}
\\
z^m & -w^{mi}
\end{smallmatrix}
\!\right\vert\,
\partial_{w^{km}}
+
\sum_{1\leqslant k<i
\atop
i<m\leqslant\ell-1}\,
\left\vert\!
\begin{smallmatrix}
z^k & -w^{ki}
\\
z^m & w^{im}
\end{smallmatrix}
\!\right\vert\,
\partial_{w^{km}}
\\
&
\ \ \ \ \
+
\sum_{i<k<m\leqslant\ell-1}\,
\left\vert\!
\begin{smallmatrix}
z^k & w^{ik}
\\
z^m & w^{im}
\end{smallmatrix}
\!\right\vert\,
\partial_{w^{km}},
\\
IY_{z^iz^i}^1
&
\,:=\,
\sum_{1\leqslant k<i}\,
\big({\scriptstyle{\sqrt{-1}}}\,z^i\,z^k+{\scriptstyle{\sqrt{-1}}}\,w^{ki}\big)\,
\partial_{z^k}
+
{\scriptstyle{\sqrt{-1}}}\,z^i\,z^i\,\partial_{z^i}
+
\sum_{i<k\leqslant\ell-1}\,
\big(
{\scriptstyle{\sqrt{-1}}}\,z^i\,z^k
-
{\scriptstyle{\sqrt{-1}}}\,w^{ik}
\big)\,
\partial_{z^k}
\\
&
\ \ \ \ \ 
+
\sum_{1\leqslant k<i}\,
{\scriptstyle{\sqrt{-1}}}\,z^i\,w^{ki}\,\partial_{w^{ki}}
+
\sum_{i<m\leqslant\ell-1}\,
{\scriptstyle{\sqrt{-1}}}\,z^i\,w^{im}\,\partial_{w^{im}}
+
\sum_{1\leqslant k<m<i}\,
{\scriptstyle{\sqrt{-1}}}\,\left\vert\!
\begin{smallmatrix}
z^k & -w^{ki}
\\
z^m & -w^{mi}
\end{smallmatrix}
\!\right\vert\,
\partial_{w^{km}}
\\
&
\ \ \ \ \
+
\sum_{1\leqslant k<i
\atop
i<m\leqslant\ell-1}\,
{\scriptstyle{\sqrt{-1}}}\,\left\vert\!
\begin{smallmatrix}
z^k & -w^{ki}
\\
z^m & w^{im}
\end{smallmatrix}
\!\right\vert\,
\partial_{w^{km}}
+
\sum_{i<k<m\leqslant\ell-1}\,
{\scriptstyle{\sqrt{-1}}}\,\left\vert\!
\begin{smallmatrix}
z^k & w^{ik}
\\
z^m & w^{im}
\end{smallmatrix}
\!\right\vert\,
\partial_{w^{km}}.
\endaligned
\]
The $\frac{(\ell-1)(\ell-2)}{2}$ generators of $\mathfrak{g}_2$ are, 
with $i < j$:
\[
\footnotesize
\aligned
Y_{w^{ij}w^{ij}}^2
&
\,:=\,
\sum_{1\leqslant k<i}\,
\left\vert\!
\begin{smallmatrix}
z^i & w^{ki}
\\
z^j & w^{kj}
\end{smallmatrix}
\!\right\vert\,
\partial_{z^k}
+
z^i\,w^{ij}\,\partial_{z^i}
+
\sum_{i<k<j}\,
\left\vert\!
\begin{smallmatrix}
z^i & -w^{ik}
\\
z^j & w^{kj}
\end{smallmatrix}
\!\right\vert\,
\partial_{z^k}
+
z^j\,w^{ij}\,\partial_{z^j}
+
\sum_{j<k\leqslant\ell-1}\,
\left\vert\!
\begin{smallmatrix}
z^i & -w^{ik}
\\
z^j & -w^{jk}
\end{smallmatrix}
\!\right\vert\,
\partial_{z^k}
\\
&
\ \ \ \ \
+
\sum_{1\leqslant k<i}\,
w^{ij}\,w^{ki}\,
\partial_{w^{ki}}
+
\sum_{1\leqslant k<i}\,
w^{ij}\,w^{kj}\,
\partial_{w^{kj}}
+
\sum_{i<m<j}\,
w^{ij}\,w^{im}\,
\partial_{w^{im}}
+
w^{ij}\,w^{ij}\,
\partial_{w^{ij}}
\\
&
\ \ \ \ \
+
\sum_{j<m\leqslant\ell-1}\,
w^{ij}\,w^{im}\,
\partial_{w^{im}}
+
\sum_{i<k<j}\,
w^{ij}\,w^{kj}\,
\partial_{w^{kj}}
+
\sum_{j<m\leqslant\ell-1}\,
w^{ij}\,w^{jm}\,
\partial_{w^{jm}}
\\
&
\ \ \ \ \
+
\sum_{1\leqslant k<m<i}\,
\left\vert\!
\begin{smallmatrix}
-w^{ki} & -w^{kj}
\\
-w^{mi} & -w^{mj}
\end{smallmatrix}
\!\right\vert\,
\partial_{w^{km}}
+
\sum_{1\leqslant k<i
\atop
i<m<j}\,
\left\vert\!
\begin{smallmatrix}
-w^{ki} & -w^{kj}
\\
w^{im} & -w^{mj}
\end{smallmatrix}
\!\right\vert\,
\partial_{w^{km}}
+
\sum_{1\leqslant k<i
\atop
j<m\leqslant\ell-1}\,
\left\vert\!
\begin{smallmatrix}
-w^{ki} & -w^{kj}
\\
w^{im} & w^{jm}
\end{smallmatrix}
\!\right\vert\,
\partial_{w^{km}}
\\
&
\ \ \ \ \ \ \ \ \ \ \ \ \ \ \ \ \ \ \ \ \ \ \ \ \ \ \ \ \ \ \ \ \ \ \
\ \ \ \ \ \ \ \ \ \ \ \ \ \
+
\sum_{i<k<m<j}\,
\left\vert\!
\begin{smallmatrix}
w^{ik} & -w^{kj}
\\
w^{im} & -w^{mj}
\end{smallmatrix}
\!\right\vert\,
\partial_{w^{km}}
+
\sum_{i<k<j
\atop
j<m\leqslant\ell-1}\,
\left\vert\!
\begin{smallmatrix}
w^{ik} & -w^{kj}
\\
w^{im} & w^{jm}
\end{smallmatrix}
\!\right\vert\,
\partial_{w^{km}}
\\
&
\ \ \ \ \ \ \ \ \ \ \ \ \ \ \ \ \ \ \ \ \ \ \ \ \ \ \ \ \ \ \ \ \ \ \
\ \ \ \ \ \ \ \ \ \ \ \ \ \ \ \ \ \ \ \ \ \ \ \ \ \ \ \ \ \ \ \ \ \ \
\ \ \ \ \ \ \ \ \ \ \ \ \ \ \ \ 
+
\sum_{j<k<m\leqslant\ell-1}\,
\left\vert\!
\begin{smallmatrix}
w^{ik} & w^{jk}
\\
w^{im} & w^{jm}
\end{smallmatrix}
\!\right\vert\,
\partial_{w^{km}}.
\endaligned
\]

\bigskip
\end{document}